\documentclass[12pt]{amsart}
\usepackage[T1]{fontenc}
\usepackage[latin9]{inputenc}
\usepackage{geometry}
\usepackage{mathrsfs}
\usepackage{amstext}
\usepackage{amsthm}
\usepackage{amssymb}

\makeatletter
\numberwithin{equation}{section}
\numberwithin{figure}{section}
\theoremstyle{plain}
\newtheorem{thm}{\protect\theoremname}[section]
\theoremstyle{definition}
\newtheorem{defn}[thm]{\protect\definitionname}
\theoremstyle{definition}
\newtheorem{rem}[thm]{\protect\remarkname}
\theoremstyle{plain}
\newtheorem{cor}[thm]{\protect\corollaryname}
\theoremstyle{plain}
\newtheorem{lem}[thm]{\protect\lemmaname}
\theoremstyle{plain}
\newtheorem{prop}[thm]{\protect\propositionname}
\theoremstyle{definition}

\usepackage[svgnames]{xcolor}
\usepackage[bookmarksnumbered=true]{hyperref} 
\hypersetup{
	colorlinks = true,
	linkcolor = Blue,
	anchorcolor = blue,
	citecolor = Green,
	filecolor = blue,
	urlcolor = FireBrick
}
\usepackage{bbm}

\MakeRobust{\ref}
\newcommand{\labeltext}[2]{
\@bsphack
\csname phantomsection\endcsname
\def\@currentlabel{#1}{\label{#2}}
\@esphack
}

\usepackage{scalerel}
\usepackage[usestackEOL]{stackengine}
\def\dashint{\,\ThisStyle{\ensurestackMath{%
  \stackinset{c}{.2\LMpt}{c}{.5\LMpt}{\SavedStyle-}{\SavedStyle\phantom{\int}}}%
  \setbox0=\hbox{$\SavedStyle\int\,$}\kern-\wd0}\int}

\usepackage{enumitem}

\usepackage{comment}

\DeclareRobustCommand{\SkipTocEntry}[5]{}

\newcommand{\mR}{\mathbb{R}}   
\newcommand{\abs}[1]{\lvert #1 \rvert}  

\newcommand{\p}{\partial}

\makeatother

\usepackage{babel}
\providecommand{\corollaryname}{Corollary}
\providecommand{\definitionname}{Definition}
\providecommand{\lemmaname}{Lemma}
\providecommand{\propositionname}{Proposition}
\providecommand{\remarkname}{Remark}
\providecommand{\theoremname}{Theorem}
\providecommand{\examplename}{Example}

\begin{document}
\title[A minimization problem with free boundary]{A minimization problem with free boundary and its application to inverse scattering problems}

\begin{sloppypar}

\begin{abstract}
We study a minimization problem with free boundary, resulting in hybrid quadrature domains for the Helmholtz equation, as well as some application to inverse scattering problem.  
\end{abstract}

\author{Pu-Zhao Kow}
\address{Department of Mathematical Sciences, National Chengchi University, No. 64, Sec. 2, ZhiNan Rd., Wenshan District, 116302 Taipei, Taiwan}
\email{pzkow@g.nccu.edu.tw}
\author{Mikko Salo}
\address{Department of Mathematics and Statistics, P.O. Box 35 (MaD), FI-40014 University of Jyv\"{a}skyl\"{a}, Finland.}
\email{mikko.j.salo@jyu.fi}
\author{Henrik Shahgholian}
\address{Department of Mathematics, KTH Royal Institute of Technology, SE-10044 Stockholm, Sweden. }
\email{henriksh@kth.se}

\maketitle

\addtocontents{toc}{\SkipTocEntry}

\section*{Notation index}


\bigskip
\noindent \ref{index:Lebesgue-measure}, page~\pageref{index:Lebesgue-measure}\\
\noindent \ref{index:Hausdorff-measure}, page~\pageref{index:Hausdorff-measure}\\
\noindent \ref{index:constraint-K}, page~\pageref{index:constraint-K}\\
\noindent \ref{index:functional-J}, page~\pageref{index:functional-J}\\
\noindent \ref{index:quadrature-domain}, page~\pageref{index:quadrature-domain}\\
\noindent \ref{index:mes-boundary}, page~\pageref{index:mes-boundary}\\
\noindent \ref{index:red-boundary}, page~\pageref{index:red-boundary}


\tableofcontents{}

\section{Introduction}

Motivated by questions in inverse scattering theory, the article \cite{KLSS22QuadratureDomain} introduced the notion of quadrature domains for the Helmholtz operator $\Delta + k^2$ with $k > 0$, also called $k$-quadrature domains. 
Given any $\mu \in \mathscr{E}'(\mathbb{R}^{n})$, a bounded open set $D \subset \mathbb{R}^n$ is called a \emph{$k$-quadrature domain with respect to $\mu$}
if  $\mu \in \mathscr{E}'(D)$ and  
\begin{equation}
\int_{D}w(x)\,dx=\langle\mu,w\rangle, \label{eq:classical-QD-first}
\end{equation}
for all $w\in L^{1}(D)$ satisfying $(\Delta+k^{2})w=0$ in $D$. The case $k=0$ corresponds to classical quadrature domains for harmonic functions. As a consequence of a mean value theorem for the Helmholtz equation (which goes back to H.\ Weber, 
see e.g.\ \cite[Proposition~A.6]{KLSS22QuadratureDomain} or \cite[p.~289]{CH89MethodsMathematicsPhysicsII}), balls are always $k$-quadrature domains with $\mu$ being a multiple of the Dirac delta function. The work \cite{KLSS22QuadratureDomain} gave further examples of $k$-quadrature domains including cardioid type domains in the plane, implemented a partial balayage procedure to construct such domains, and showed that such domains may be non-scattering domains for certain incident waves. The results were based on the following PDE characterization (see \cite[Proposition 2.1]{KLSS22QuadratureDomain}): a bounded open set $D \subset \mathbb{R}^n$ is a $k$-quadrature domain for $\mu \in \mathscr{E}'(D)$ if and only if there is $u \in \mathscr{D}'(\mathbb{R}^n)$ satisfying  (an obstacle-like free boundary problem)
\begin{equation}
\begin{cases}
(\Delta+k^{2})u=\chi_D-\mu & \text{in }\mathbb{R}^{n},\\
u=\abs{\nabla u} = 0 & \text{in }\mathbb{R}^{n}\setminus D.
\end{cases}\label{eq:Schiffer-first}
\end{equation}

In this work we study $k$-quadrature domains in the presence of densities both on $D$ and $\partial D$. Let $\mu \in \mathscr{E}'(\mathbb{R}^{n})$. If we only have one density $h \geq 0$ on $D$, we may look for bounded domains $D$ for which $\mu \in \mathscr{E}'(D)$ and 
\[
\int_D w(x) h(x) \,dx = \langle \mu, w \rangle
\]
for all $w \in L^1(D)$ solving $(\Delta+k^2)w = 0$ in $D$. Such a set $D$ could be called a weighted $k$-quadrature domain. More generally, if we also have a density $g \geq 0$ on $\partial D$, we consider the following Bernoulli type free boundary  problem generalizing \eqref{eq:Schiffer-first}:
\begin{equation}
\begin{cases}
(\Delta+k^{2})u=h -\mu & \text{in }D,\\
u=0 & \text{on }\partial  D,\\
|\nabla u|=g & \text{on }\partial D,
\end{cases}\label{eq:Schiffer-second}
\end{equation}
where the Bernoulli condition $|\nabla  u |=g$
is in a very weak sense; 
see Proposition~{\rm \ref{prop:PDE2}} or \cite[Theorem~2.3]{GS96FreeBoundaryPotential}.
Given any $\mu \in \mathscr{E}'(\mathbb{R}^{n})$, a bounded domain $D$ for which $\mu \in \mathscr{E}'(D)$ and \eqref{eq:Schiffer-second} has a solution $u$ will be called a \emph{hybrid $k$-quadrature domain}. 
The main theme of this paper is to study such domains. We will establish the existence of hybrid $k$-quadrature domains for suitable $\mu$ via a minimization problem. We will also give examples of such domains with real-analytic boundary, and show that hybrid $k$-quadrature domains may be non-scattering domains in the presence of certain boundary sources. We will closely follow \cite{GS96FreeBoundaryPotential} which studied the case $k=0$. It turns out that many of our results can be reduced to the situation in \cite{GS96FreeBoundaryPotential}, but certain parts will require modifications. Even though part of the treatment is very similar to \cite{GS96FreeBoundaryPotential}, we will try to give enough details so that also readers who are not experts on this topic can follow the presentation.

\addtocontents{toc}{\SkipTocEntry}
\subsection{Minimization problem}

Let $\Omega\subset\mathbb{R}^{n}$ be an open set in $\mathbb{R}^{n}$ (with $n\ge2$). Let $C_{c}^{\infty}(\Omega)$ consist of $C^{\infty}(\mathbb{R}^{n})$
functions which are supported in $\Omega$, and denote by $H_{0}^{1}(\Omega)$ the completion of $C_{c}^{\infty}(\Omega)$ with respect to the $H^{1}(\Omega)$-norm. We define the set
\footnote{In particular, the inequality $u \ge 0$ in the definition of $\mathbb{K}(\Omega)$ can be interpreted in a.e. pointwise sense, see e.g.\ \cite[Definition~II.5.1]{KS00IntroductionVariationalInequalities}.}
\labeltext{constraint set $\mathbb{K}(\Omega)$}{index:constraint-K}
\[
\mathbb{K}(\Omega):=\begin{Bmatrix}\begin{array}{l|l}
u\in H_{0}^{1}(\Omega) & u\ge0\end{array}\end{Bmatrix},
\]
and for each $u\in H_{0}^{1}(\Omega)$ we define  
\[
\{u>0\}=\begin{Bmatrix}\begin{array}{l|l}
x_{0}\in \Omega & \begin{array}{l}
\text{there exists non-negative }\phi\in C_{c}^{\infty}(\Omega)\\
\text{with }\phi(x_{0})>0\text{ such that }u\ge\phi\text{ in } \Omega.
\end{array}\end{array}\end{Bmatrix}
\]

Let $\lambda\in\mathbb{R}$. For given functions $f,g\in L^{\infty}(\mathbb{R}^{n})$
with $g\ge0$, we define the functional \labeltext{functional $\mathcal{J}_{f,g,\lambda,\Omega}$}{index:functional-J}
\begin{equation} \label{jfunctional}
\mathcal{J}_{f,g,\lambda,\Omega}(u):=\int_{\Omega}(|\nabla u(x)|^{2}-\lambda|u(x)|^{2}-2f(x)u(x)+g^{2}(x)\chi_{\{u>0\}})\,dx,
\end{equation}
which is well-defined for all $u\in H_{0}^{1}(\Omega)$. The main
purpose of this paper is to study the following minimization problem:
\begin{equation}
\text{minimize }\mathcal{J}_{f,g,\lambda,\Omega}(u)\text{ subject to }u \in \mathbb{K}(\Omega).\label{eq:minimizing-problem}
\end{equation}
It is easy to see that 
\begin{equation}
\inf_{u\in\mathbb{K}(\Omega)}\mathcal{J}_{f,g,\lambda,\Omega}(u)\le\mathcal{J}_{f,g,\lambda,\Omega}(0)=0.\label{eq:non-positive-functional}
\end{equation}

We will show that there exists a minimizer of \eqref{eq:minimizing-problem}
when $\Omega$ is a bounded Lipschitz domain and $-\infty<\lambda<\lambda^{*}(\Omega)$
(Proposition~\ref{prop:existence-minimizer}). Here $\lambda^{*}(\Omega)$ is the \emph{fundamental tone} of $\Omega$, defined by 
\[
\lambda^{*}(\Omega):=\inf_{\phi\in C_{c}^{\infty}(\Omega), \phi \not\equiv 0}\frac{\|\nabla\phi\|_{L^{2}(\Omega)}^{2}}{\|\phi\|_{L^{2}(\Omega)}^{2}}.
\]
It is well-known that $\lambda^{*}(\Omega)$ is the infimum of the Dirichlet spectrum of $-\Delta$ on $\Omega$.
In addition when $\Omega$ is $C^{1}$, we will show that there 
exists a countable set $Z\subset(-\infty,\lambda^{*}(\Omega))$ such
that the minimizer of \eqref{eq:minimizing-problem} is unique for
all $\lambda\in(-\infty,\lambda^{*}(\Omega))\setminus Z$ (Proposition~\ref{prop:unique}). 
The functional $\mathcal{J}_{f,g,\lambda,\Omega}$
is  unbounded below in $\mathbb{K}(\Omega)$ when $\lambda>\lambda^{*}(\Omega)$
(Lemma~\ref{lem:boundedness}), which shows the non-existence of a global minimizer of \eqref{eq:minimizing-problem} for $\lambda>\lambda^{*}(\Omega)$.

\addtocontents{toc}{\SkipTocEntry}
\subsection{Quadrature domains via minimization}

Given two non-negative functions $h$ and $g$ in $\mathbb{R}^{n}$ ($n \geq 2$), and a positive measure $\mu$ with compact support in $\mathbb{R}^{n}$, we wish to find a bounded domain $D$ with Hausdorff $(n-1)$-dimensional boundary $\partial D$ containing ${\rm supp}\,(\mu)$ such that the potential $\Psi_{k}*\mu$ (see Definition~\ref{def:weighted-QD}) for any fundamental solution $\Psi_k$ of $-(\Delta + k^2)$ agrees outside $D$ with that of the measure
 \labeltext{$\mathscr{L}^{n} \lfloor D$ Lebesgue measure restricted to $D$}{index:Lebesgue-measure}
\labeltext{$\mathscr{H}^{n-1} \lfloor \partial D$ $(n-1)$-dimensional Hausdorff measure on $\partial D$}{index:Hausdorff-measure}
\begin{equation}
\sigma :=h\mathscr{L}^{n}\lfloor D+g\mathscr{H}^{n-1}\lfloor\partial D,\label{eq:measure-nu}
\end{equation}
where $\mathscr{L}^{n}\lfloor D$ and $\mathscr{H}^{n-1}\lfloor\partial D$ denote the Lebesgue measure restricted to $D$ and the $(n-1)$-dimensional Hausdorff measure on $\partial D$, respectively.

We now introduce a hybrid version of a
quadrature domain  in the following definition. \labeltext{Definition of hybrid $k$-quadrature domain}{index:quadrature-domain}

\begin{defn}
\label{def:weighted-QD}Let $k>0$ and let $D$ be a bounded open set in $\mathbb{R}^{n}$ with the  boundary $\partial D$ having  finite $(n-1)$-dimensional Hausdorff measure. Let $\sigma$ be the measure given by \eqref{eq:measure-nu}. The set $D\subset\mathbb{R}^{n}$ is called a \emph{hybrid $k$-quadrature domain}, corresponding
to distribution $\mu\in\mathscr{E}'(D)$ (with ${\rm supp}\,(\mu)\subset D$)
and density $(g,h)\in L^{\infty}(\partial D)\times L^{\infty}(D)$, if
\begin{equation}
\text{$\left( \Psi_{k} * (g \mathscr{H}^{n-1} \lfloor \partial D) \right)(x)$ is well-defined pointwise for all $x \in \partial D$} \label{eq:well-defined-singular-integral}
\end{equation}
and
\begin{equation}
\Psi_{k}*\mu=\Psi_{k}*\sigma\quad\text{in }\mathbb{R}^{n}\setminus D,\label{eq:weighted-QD}
\end{equation}
for all fundamental solutions $\Psi_{k}$ of the Helmholtz operator $-(\Delta + k^{2})$, that is, $-(\Delta + k^{2})\Psi_{k} = \delta_{0}$ in $\mathscr{D}'(\mathbb{R}^{n})$. 
\end{defn}

\begin{rem}
In general, the condition \eqref{eq:well-defined-singular-integral} does not follow from standard elliptic regularity results. Some extra assumptions (see e.g.\ Theorem~{\rm \ref{thm:main-quadrature}}) are required to ensure \eqref{eq:well-defined-singular-integral} holds. 
\end{rem}

\begin{rem}\label{rem:hybrid-QD-Lipschitz}
Let $\Psi_k$ be any fundamental solution for $-(\Delta+k^2)$ in $\mR^n$. Given any bounded Lipschitz domain $D$, let $\gamma: H_{\rm loc}^1(\mR^n) \to H^{1/2}(\partial D)$ be the trace operator. The adjoint $\gamma^{*}$ of $\gamma$ is defined by 
\begin{equation}
\langle \gamma^{*}\psi,\phi \rangle = \langle \psi, \gamma \phi \rangle_{\partial D} \quad \text{for all $\phi \in C_{c}^{\infty}(\mathbb{R}^{n})$.} \label{eq:adjoint-trace-operator}
\end{equation}
In particular, $\gamma^{*}$ maps $H^{-\frac{1}{2}}(\partial D)$ to $H^{-1}(\mathbb{R}^n)$.  Using \cite[Theorem~6.11]{McL00EllipticSystems}, we can define the single-layer potential by 
\begin{equation}
\mathsf{SL}: H^{-\frac{1}{2}}(\partial D) \rightarrow H_{\rm loc}^{1}(\mathbb{R}^{n}) ,\quad \mathsf{SL}(g) := \Psi_{k}*(\gamma^{*} g). \label{eq:single-layer-potential}
\end{equation}
In this case, \eqref{eq:weighted-QD} simply reads 
\begin{equation*}
\Psi_{k}*\mu = \Psi_{k}*(h\chi_{D}) + \mathsf{SL}(g). 
\end{equation*}
\end{rem}

\begin{rem}
When $D$ is a hybrid $k$-quadrature domain with $g\equiv0$, using a
$L^{1}$-density theorem \cite[Proposition~2.4]{KLSS22QuadratureDomain}
and a similar argument as in \cite[Theorem~1.2]{KLSS22QuadratureDomain},
we know that 
\[
\langle\mu,w\rangle=\int_{D}w(x)h(x)\,dx
\]
for all $w\in L^{1}(D)$ such that $(\Delta+k^{2})w=0$ in $D$. In this case, if \eqref{eq:weighted-QD} is true for one fundamental solution, then it is true for all fundamental solutions.
\end{rem}

We have the following theorem: 
\begin{thm}
[See {\rm Theorem~\ref{thm:main-quadrature-detail}} for a more detailed statement]
\label{thm:main-quadrature} Let $n \ge 2$, and assume 
$h$ and $g$ are sufficiently regular. If $\mu$ is a non-negative measure on $\mathbb{R}^{n}$
with mass concentrated near a point and $R > 0$, then for each sufficiently small
$k>0$ there exists a bounded open domain $D$ in $\mathbb{R}^{n}$ with the boundary $\partial D$ having finite $(n-1)$-dimensional Hausdorff measure such that \eqref{eq:well-defined-singular-integral} holds, which is a hybrid $k$-quadrature domain corresponding
to distribution $\mu$ and density $(g,h)$ satisfying $\overline{D}\subset B_{\beta k^{-1}}$. 
In particular when $g>0$ is H\"{o}lder continuous in $\overline{B_{R}}$, then there exists a portion $E \subset \partial D$ with $\mathscr{H}^{n-1}(\partial D \setminus E)=0$ such that $E$ is locally $C^{1,\alpha}$. 
In the case when $n=2$ we even can choose $E = \partial D$. 
\end{thm}

\begin{rem}
The hybrid $k$-quadrature domain constructed in
Theorem~\ref{thm:main-quadrature} can be represented by $D=\{u_{*}>0\}$,
where $u_{*}$ is a minimizer of $\mathcal{J}_{f,g,k^{2},B_{R}}$
in $\mathbb{K}(B_{R})$ with $f=\mu-h\chi_{D}$ when $\mu$ is bounded (for general $\mu$ we consider some suitable mollifiers). Since the minimizer is unique for $k$ outside a countable set, then so is the constructed domain, see Proposition~\ref{prop:unique}. See also Proposition~{\rm \ref{prop:main-almost-there}} for the case when $\mu$ is bounded. 
\end{rem}

\addtocontents{toc}{\SkipTocEntry}
\subsection{Real analytic quadrature domains}

We can construct examples of hybrid $k$-quadrature domains using the Cauchy-Kowalevski theorem. Let  $D$ be a bounded
domain in $\mathbb{R}^{n}$ with real-analytic boundary. Let $g$
be real analytic on a neighborhood of $\partial D$ with $g>0$ on
$\partial D$. For each $k\ge0$, there exists a bounded positive measure $\mu_{1}$ with ${\rm supp}\,(\mu_{1})\subset D$
such that $D$ is a hybrid $k$-quadrature domain corresponding to $\mu_{1}$ with  density $(g,0)$. Moreover, if $0\le k<j_{\frac{n-2}{2},1}R^{-1}$ (where $j_{\alpha,1}$ is the first positive zero of the Bessel function $J_{\alpha}$), $\overline{D}\subset B_{R}$, and if $h$ is a non-negative integrable
function near $\overline{D}$ which is real-analytic
near $\partial D$, then $D$ is a hybrid $k$-quadrature domain  corresponding
to some measure $\mu_{2}$ with density $(0,h)$. The proofs follow easily by solving suitable Cauchy problems near $\partial D$ by the Cauchy-Kowalevski theorem, and defining $\mu_1$ and $\mu_2$ in terms of the obtained solutions. 
For the details see Appendix~{\rm \ref{sec:Real-analytic-case}}.

\addtocontents{toc}{\SkipTocEntry}

\subsection*{Organization}

We first discuss the application to inverse problems in Section~{\rm \ref{sec:Application-to-inverse}}. Then we prove the existence of global minimizers of $\mathcal{J}_{f,g,\lambda,\Omega}$ in $\mathbb{K}(\Omega)$ in Section~{\rm \ref{sec:Existence-of-minimizers}}. We study the relation between local minimizers and partial differential equations in Section~{\rm \ref{sec:Minimizers-and-PDE}}. Next, we study the local minimizers in Section~{\rm \ref{sec:Comparison-of-minimizers:}} and Section~{\rm \ref{sec:Properties-local-minimizers}}. With these ingredients at hand, we prove Theorem~{\rm \ref{thm:main-quadrature}} in Section~{\rm \ref{sec:Relation-with-quadrature}}. 

For reader's convenience, we add several appendices to make the paper self-contained. In Appendix~{\rm \ref{sec:BV-finite-perimeter}} we recall a few facts about functions of bounded variation and sets with finite perimeter. Appendix~{\rm \ref{appen:details}} provides detailed statements and proofs of results analogous to \cite[Section~2]{GS96FreeBoundaryPotential}. We then exhibit the detailed proof of Lemma~{\rm \ref{lem:radially1}} in Appendix~{\rm \ref{appen:computations}}. Appendix~{\rm \ref{sec:Real-analytic-case}} discusses examples of hybrid $k$-quadrature domains with real-analytic boundary. Finally, we give some remarks on null $k$-quadrature domains in Appendix~{\rm \ref{sec:Rellich}}.  

\section{\label{sec:Application-to-inverse}Applications to inverse scattering problems}

We say that a solution $u$ of $(\Delta+k_{0}^{2})u=0$ in $\mathbb{R}^{n}\setminus\overline{B_{R}}$
(for some $R>0$) is \emph{outgoing} if it satisfies the following Sommerfeld radiation condition: 
\begin{equation}
\lim_{|x|\rightarrow\infty}|x|^{\frac{n-1}{2}}(\partial_{r}u-ik_{0}u)=0,\quad\text{uniformly in all directions }\hat{x}=\frac{x}{|x|}\in\mathcal{S}^{n-1},\label{eq:Sommerfeld}
\end{equation}
where $\partial_{r}$ denotes the radial derivative. There exists a unique $u^{\infty}\in L^{2}(\mathcal{S}^{d-1})$, which is called the far-field pattern of $u$, such that  
\[
u(x) = \gamma_{n,k_{0}}\frac{e^{ik_{0}|x|}}{|x|^{\frac{n-1}{2}}}u^{\infty}(\hat{x})+\mathscr{O}(|x|^{-\frac{n+1}{2}})\quad\text{as}\;\;|x|\rightarrow\infty,
\]
uniformly in all direction $\hat{x} \in \mathcal{S}^{d-1}$, where we make the choice (as in \cite[Section 1.2.3]{Yaf10ScatteringAnalyticTheory})
\begin{equation}
\gamma_{n,k_{0}} = \frac{ e^{-\frac{(n-3)\pi i}{4}} }{2(2\pi)^{\frac{n-1}{2}}} k_{0}^{\frac{n-3}{2}}. \label{eq:choice-gamma}
\end{equation}
For $n=2$ we have $\gamma_{2,k_{0}}=\frac{e^{\frac{i\pi}{4}}}{\sqrt{8\pi k_{0}}}$, while when $n=3$ we have $\gamma_{3,k_{0}} = \frac{1}{4\pi}$.

Let $D$ be a bounded domain in $\mathbb{R}^{n}$, which represents a penetrable obstacle with contrast $\rho \in L^{\infty}(D)$ satisfying $|\rho| \ge c > 0$ a.e. near $\partial D$. When one probes the obstacle $(D,\rho)$ using an incident wave $u_{0}$ satisfying $(\Delta + k_{0}^{2})u_{0} = 0$ in $\mathbb{R}^{n}$, it produces an outgoing scattered field $u_{\rm sc}$ solving 
\[
(\Delta + k_{0}^{2} + \rho \chi_{D}) (u_{0} + u_{\rm sc}) = 0 \quad \text{in $\mathbb{R}^{n}$.}
\]
We say that the obstacle $(D,\rho)$ is non-scattering with respect to the incident field $u_{0}$ and the wave number $k_{0}$ if the far-field pattern $u_{\rm sc}^{\infty}$ of the corresponding scattered field $u^{\rm sc}$ vanishes identically. Using Rellich uniqueness theorem \cite{CK19scattering,Hormander_rellich} we know that $u_{\rm sc}^{\infty} \equiv 0$ if and only if $u_{\rm sc} = 0$ in $\mathbb{R}^{n}\setminus\overline{B_{R}}$ for some $R>0$, therefore this definition coincides with \cite[Definition~1.8]{KLSS22QuadratureDomain}. 
The following theorem extend \cite[Corollary~1.9]{KLSS22QuadratureDomain}. We remind the readers that there are some significant differences between 0-quadrature domains and $k$-quadrature domains, see Appendix~{\rm \ref{sec:Rellich}} for more details.

\begin{thm}\label{thm:non-scattering-result1}
Let $D$ be a bounded hybrid $k$-quadrature domain, corresponding to distribution $\mu \in \mathscr{E}'(D)$ and density $(0,h)$ with $h \in L^{\infty}(D)$ and $\abs{h}\ge c > 0$ near $\partial D$. 
Assume that there exist a wave number $k_{0} \ge 0$ (which may differ from $k$) and
\begin{equation}
\text{a solution $u_{0}$ of $(\Delta+k_{0}^{2})u_{0}=0$ in $\mathbb{R}^{n}$ such that $u_{0}>0$ on $\partial D$.} \label{eq:positive-inciden-field}
\end{equation}
Then there exists a contrast $\rho \in L^{\infty}(D)$ satisfying $|\rho| \ge c > 0$ a.e. near $\partial D$ such that $(D,\rho)$ is non-scattering with respect to the incident field $u_{0}$ and the wave number $k_{0}$. 
\end{thm}

\begin{rem}
Using the result in \cite{KSS_PositivitySetsHelmholtz} (see also \cite{SS21NonscatteringFreeBoundary}), we know that \eqref{eq:positive-inciden-field} holds at least when
\begin{enumerate}
\renewcommand{\labelenumi}{\theenumi}
\renewcommand{\theenumi}{(\Alph{enumi})}
\item \label{itm:A} $D$ is Lipschitz so that $\mathbb{R}^{n}\setminus\overline{D}$
is connected and $k_{0}^{2}$ is not a Dirichlet eigenvalue of $-\Delta$
in $D$; or 
\item \label{itm:B} $D$ is a compact set contained in a bounded Lipschitz domain $\Omega$ such that $\mathbb{R}^{n}\setminus\overline{\Omega}$ is connected and $\abs{\Omega} \leq \abs{B_r}$ where $r = j_{\frac{n-2}{2},1} k_{0}^{-1}$. 
\end{enumerate}
\end{rem}

\begin{proof}[Proof of Theorem~{\rm \ref{thm:non-scattering-result1}}]
Following the same ideas in \cite[Theorem~1.2 and Remark~1.3]{KLSS22QuadratureDomain}, one can show that there is a neighborhood $U$ of $\partial D$ in $\mathbb{R}^{n}$ and a distribution $u \in \mathscr{D}'(U)$ satisfying 
\begin{equation}
\begin{cases}
(\Delta+k_{0}^{2}) u = ((k_{0}^{2}-k^{2})u + h) \chi_{D} & \text{in $U$,} \\
u = |\nabla u| = 0 & \text{in $U \setminus D$.}
\end{cases} \label{eq:non-scattering-sufficient}
\end{equation}
By elliptic regularity, one has $u \in C^{1,\alpha}(U)$. The function $v_{0} = u + u_{0}$ satisfies 
\begin{equation*}
\left\{\begin{aligned}
&(\Delta + k_{0}^{2} + \rho_{0}\chi_{D})v_{0} = \left(\rho_{0}v_{0} + (k_{0}^{2}-k^{2})u + h\right)\chi_{D} \quad \text{in $U$,} \\
& \left. v_{0} \right|_{U\setminus D} = \left. u_{0} \right|_{U\setminus D}.
\end{aligned}\right.
\end{equation*}
By choosing $\rho_{0} = -\frac{(k_{0}^{2}-k^{2})u+h}{v_{0}}$ near $\partial D$, one can verify that $v_{0} \in C^{1,\alpha}(U)$ and $\abs{\rho_{0}} \ge c > 0$ near $\partial D$. Hence the theorem follows by the following Lemma~\ref{lem:technical-extension} (with $g\equiv 0$).
\end{proof}

To investigate the case when $g$ is nontrivial, we need the following technical lemma, which is a refinement of \cite[Lemma~2.3]{SS21NonscatteringFreeBoundary}.  

\begin{lem} \label{lem:technical-extension}
Let $k_{0} \ge 0$, let $D$ be a bounded open set in $\mathbb{R}^{n}$ with the boundary $\partial D$ having finite $(n-1)$-dimensional Hausdorff measure, and let $g \in L^{\infty}(\partial D)$. Given any $u_{0}$ as in \eqref{eq:positive-inciden-field}, any open neighborhood $U$ of $\partial D$ in $\mathbb{R}^{n}$, any $\rho_{0} \in L^{\infty}(U)$ with $\abs{\rho_{0}} \ge c > 0$ near $\partial D$ and any $v_{0} \in C_{\rm loc}^{0,1}(U)$ such that 
\begin{equation}
\begin{cases}
(\Delta + k_{0}^{2} + \rho_{0}\chi_{D})v_{0} = g \mathscr{H}^{n-1} \lfloor \partial D & \text{in $U$,} \\
v_{0}|_{U\setminus D} = u_{0}|_{U\setminus D},
\end{cases} \label{eq:assumption-v0}
\end{equation}
there exist $\rho \in L^{\infty}(\mathbb{R}^{n})$ and $v \in C_{\rm loc}^{0,1}(\mathbb{R}^{n})$, with $\rho = \rho_{0}$ and $v=v_{0}$ near $\partial D$, such that 
\begin{equation}
\begin{cases}
(\Delta + k_{0}^{2} + \rho \chi_{D})v = g \mathscr{H}^{n-1} \lfloor \partial D & \text{in $\mathbb{R}^{n}$,} \\
v|_{\mathbb{R}^{n}\setminus D} = u_{0}|_{\mathbb{R}^{n}\setminus D}.
\end{cases} \label{eq:extension-v}
\end{equation}
\end{lem}

\begin{proof}
By \eqref{eq:positive-inciden-field}, \eqref{eq:assumption-v0} and continuity of $v_{0}$, one sees that $v_{0}$ is positive in some neighborhood $U' \subset U$ of $\partial D$. Choose $\psi \in C_{c}^{\infty}(U')$ such that $0 \le \psi \le 1$ and $\psi=1$ near $\partial D$, and define 
\[
v = \begin{cases}
v_{0}\psi + (1 - \psi) &\text{in $D$,} \\
u_{0} &\text{in $\mathbb{R}^{n}\setminus D$.}
\end{cases}
\]
Then $v \in C_{\rm loc}^{0,1}(\mathbb{R}^{n})$ is positive near $\overline{D}$ and satisfies $v = v_{0}$ near $\partial D$. We observe that the function defined by  
\begin{equation}
\rho = \left\{
\begin{aligned}
& - \frac{(\Delta+k_{0}^{2})v}{v} && \text{in $D$,} \\
&\psi \rho_{0} && \text{in $\mathbb{R}^{n} \setminus D$,}
\end{aligned}
\right. \label{eq:extension-rho1}
\end{equation}
is $L^{\infty}$ in $D$ and satisfies $\rho = \rho_{0}$ near $\partial D$, which implies that $\rho \in L^{\infty}(\mathbb{R}^{n})$. From \eqref{eq:extension-rho1} it is not difficult to see that 
\begin{subequations}
\begin{equation}
(\Delta + k_{0}^{2} + \rho \chi_{D})v = 0 = g \mathscr{H}^{n-1} \lfloor \partial D \quad \text{in $D$.} \label{eq:definition-v-prop1}
\end{equation}
Since $v = v_{0}$ and $\rho = \rho_{0}$ near $\partial D$, from  \eqref{eq:assumption-v0} we see that 
\begin{equation}
(\Delta + k_{0}^{2} + \rho \chi_{D})v = g \mathscr{H}^{n-1} \lfloor \partial D \quad \text{near $\partial D$.} \label{eq:definition-v-prop2}
\end{equation}
Since $v = u_{0}$ in $\mathbb{R}^{n} \setminus \overline{D}$, then we also have 
\begin{equation}
(\Delta + k_{0}^{2} + \rho \chi_{D})v = (\Delta + k_{0}^{2})v = 0 = g \mathscr{H}^{n-1} \lfloor \partial D \quad \text{in $\mathbb{R}^{n} \setminus \overline{D}$.} \label{eq:definition-v-prop3}
\end{equation}
\end{subequations}
By combining \eqref{eq:definition-v-prop1}, \eqref{eq:definition-v-prop2} and \eqref{eq:definition-v-prop3}, we conclude \eqref{eq:extension-v}. 
\end{proof}

We are now ready to prove the following theorem.

\begin{thm}\label{thm:non-scattering-result2}
Let $D = \{ u_{*} > 0 \}$ be the hybrid $k$-quadrature domain constructed in Theorem~{\rm \ref{thm:main-quadrature-detail}} {\rm (}or Theorem~{\rm \ref{thm:main-quadrature}}{\rm )} corresponding to density $(g,h)$. Given any $u_{0}$ as in \eqref{eq:positive-inciden-field}, there exists $\rho \in L^{\infty}(\mathbb{R}^{n})$ with $\abs{\rho} \ge c > 0$ near $\partial D$ and $u_{\rho,g} \in C_{\rm loc}^{0,1}(\mathbb{R}^{n})$ such that 
\begin{equation}
\begin{cases}
(\Delta + k_{0}^{2} + \rho \chi_{D})u_{\rho,g} = g \mathscr{H}^{n-1} \lfloor \partial D & \text{in $\mathbb{R}^{n}$,} \\
u_{\rho,g} = u_{0} & \text{in $\mathbb{R}^{n} \setminus D$.}
\end{cases} \label{eq:surface-source-generalize}
\end{equation}
\end{thm}

\begin{proof}
There exists an open neighborhood $U$ of $\partial D$ in $\mathbb{R}^{n}$ such that 
\[
\begin{cases}
(\Delta + k_{0}^{2})u_{*} = ((k_{0}^{2}-k^{2})u_{*} + h) \mathscr{L}^{n} \lfloor D + g \mathscr{H}^{n-1}\lfloor \partial D & \text{in $U$,} \\
u_{*}|_{U \setminus D} = 0.
\end{cases}
\]
Since the function $v_{0} = u_{*} + u_{0}$ satisfies 
\[
(\Delta + k_{0}^{2} + \rho_{0}\chi_{D})v_{0} = (\rho_{0}v_{0} + (k_{0}^{2}-k^{2})u_{*} + h) \mathscr{L}^{n} \lfloor D + g \mathscr{H}^{n-1}\lfloor \partial D,
\]
by choosing $\rho_{0} = -\frac{(k_{0}^{2}-k^{2})u_{*} + h}{v_{0}}$ near $\partial D$, one can verify that $v_{0} \in C^{0,1}(U')$, $\rho_{0} \in L^{\infty}(U')$ with $|\rho_{0}| \ge c > 0$ in $U'$ and 
\begin{equation*}
\begin{cases}
(\Delta + k_{0}^{2} + \rho_{0} \chi_{D}) v_{0} = g \mathscr{H}^{n-1}\lfloor \partial D & \text{in $U'$,} \\
v_{0}|_{U'\setminus D} = u_{0}|_{U'\setminus D} ,
\end{cases} 
\end{equation*}
for some open neighborhood $U'$ of $\partial D$ in $U$. Finally, we conclude Theorem~{\rm \ref{thm:non-scattering-result2}} using Lemma~{\rm \ref{lem:technical-extension}}. 
\end{proof}

We will now discuss how Theorem \ref{thm:non-scattering-result2}  can be interpreted as a nonscattering result. It is easy to see that the function $w_{\rho,g} := u_{\rho,g} - u_{0}$ satisfies 
\[
\begin{cases}
(\Delta + k_{0}^{2})w_{\rho,g} = -\rho u_{\rho,g} \mathscr{L}^{n} \lfloor D + g \mathscr{H}^{n-1} \lfloor \partial D & \text{in $\mathbb{R}^{n}$,} \\
w_{\rho,g} = 0 & \text{in $\mathbb{R}^{n} \setminus D$.}
\end{cases}
\]
Since $D$ is bounded, then $w_{\rho,g} \in \mathscr{E}'(\mathbb{R}^{n})$. Let $\Psi_{k_{0}}$ be any fundamental solution for $-(\Delta + k_{0}^{2})$ in $\mathbb{R}^{n}$. By the properties of convolution for distributions we have 
\[
\begin{aligned}
w_{\rho,g} &= \delta_{0}*w_{\rho,g} = - (\Delta + k_{0}^{2})\Psi_{k_{0}}*w_{\rho,g} = -\Psi_{k_{0}}*(\Delta + k_{0}^{2})w_{\rho,g} \\
&= \Psi_{k_{0}}*(\rho u_{\rho,g} \chi_{D}) - \Psi_{k_{0}}*(g \mathscr{H}^{n-1} \lfloor \partial D),
\end{aligned}
\]
that is, 
\begin{equation}
u_{\rho,g} = u_{0} + \Psi_{k_{0}}*(\rho u_{\rho,g} \chi_{D}) - \Psi_{k_{0}}*(g \mathscr{H}^{n-1} \lfloor \partial D). \label{eq:convolution-fundamental}
\end{equation}
When $\partial D$ above is Lipschitz (by Theorem \ref{thm:main-quadrature-detail} this is true for example when $n=2$), the outer unit normal vector $\nu$ on $\partial D$ is $\mathscr{H}^{n-1}$-a.e. well-defined in the sense of \cite[Theorem~5.8.1]{EG15MeasureTheory}. Now let $\gamma^{*}$ be the adjoint of the trace operator on $\partial D$ as in \eqref{eq:adjoint-trace-operator}. In this case, we can write \eqref{eq:convolution-fundamental} as
\begin{equation}
u_{\rho,g} = u_{0} + \Psi_{k_{0}}*(\rho u_{\rho,g} \chi_{D}) - \mathsf{SL}\,(g), \label{eq:convolution-fundamenta2}
\end{equation}
where $\mathsf{SL}\,(g)$ is the single layer potential as in \eqref{eq:single-layer-potential}. Since $\rho u_{\rho,g} \chi_{D} \in L^{\infty}(\mathbb{R}^{n})$, then one sees that $u_{0} + \Psi_{k_{0}}*(\rho u_{\rho,g} \chi_{D}) \in C_{\rm loc}^{1}(\mathbb{R}^{n})$. Consequently, by using the jump relations of the layer potential in \cite[Theorem~6.11]{McL00EllipticSystems}, we have 
\begin{subequations}
\begin{equation}
\partial_{\nu}^{+}u_{\rho,g} - \partial_{\nu}^{-}u_{\rho,g} = g \quad \text{in $H^{-\frac{1}{2}}(\partial D)$-sense.} \label{eq:jump-relation}
\end{equation}
Here $\partial_{\nu}^{-}$ (resp.\ $\partial_{\nu}^{+}$) denotes the normal derivative from the interior (resp.\ exterior) of $D$. Obviously  $u_{\rho,g} \in C_{\rm loc}^{0,1}(\mathbb{R}^{n})$ satisfies 
\begin{align}
& (\Delta + k_{0}^{2} + \rho)u_{\rho,g} = 0 \quad \text{in $D$,} \label{eq:interior} \\
& u_{\rho,g}|_{\mathbb{R}^{n} \setminus \overline{D}} = u_{0}|_{\mathbb{R}^{n} \setminus \overline{D}}. \label{eq:exterior} 
\end{align}
\end{subequations}

By \eqref{eq:jump-relation}--\eqref{eq:exterior}, we can interpret $g \mathscr{H}^{n-1} \lfloor \partial D$ in \eqref{eq:surface-source-generalize} as a \emph{nonradiating surface source} with respect to incident field $u_{0}$ and potential $\rho \in L^{\infty}(D)$. In other words the obstacle $D$ is nonscattering with respect to both the contrast $\rho$ in $D$ and surface source $g$ on $\partial D$. We could formally also write the equation \eqref{eq:surface-source-generalize} as 
\begin{equation}
\begin{cases}
(\Delta + k_{0}^{2} + (\rho \chi_{D} + \tilde{g} \mathscr{H}^{n-1} \lfloor \partial D))u_{\rho,g} = 0 & \text{in $\mathbb{R}^{n}$,} \\
u_{\rho,g} = u_{0} & \text{in $\mathbb{R}^{n} \setminus D$.}
\end{cases} \label{eq:surface-source-generalize-second}
\end{equation}
where $\tilde{g} = g/u_{\rho,g}$ on $\partial D$, which would correspond to a \emph{nonscattering domain with singular contrast}. See also \cite{KW21CharacterizeNonradiating} for discussion about surface sources on Lipschitz surfaces.

We now discuss the case when the background medium is anisotropic inhomogeneous. Let $m > \frac{n}{2}$ be an integer, let $\rho \in C_{\rm loc}^{m-1,1}(\mathbb{R}^{n})$ and let $A \in (C_{\rm loc}^{m,1}(\mathbb{R}^{n}))_{\rm sym}^{n \times n}$ satisfy the uniform ellipticity  condition: there exists a constant $c_{0}>0$ such that 
\begin{equation}
\xi\cdot A(x)\xi \ge c_{0}|\xi|^{2} \quad \text{for all $x,\xi \in \mathbb{R}^{n}$.} \label{eq:ellipticity-condition}
\end{equation}
Let $u_{0} \in H_{\rm loc}^{1}(\mathbb{R}^{n})$ satisfy 
\begin{equation}
(\nabla\cdot A\nabla + k_{0}^{2}\rho)u_{0}=0 \quad \text{in $\mathbb{R}^{n}$.} \label{eq:anisotropic-u0}
\end{equation}
By using \cite[Theorem~8.10]{GT01Elliptic} and Sobolev embedding, one sees that 
\begin{equation*}
u_{0} \in H_{\rm loc}^{m+2}(\mathbb{R}^{n}) \subset C_{\rm loc}^{2}(\mathbb{R}^{n}). 
\end{equation*}

\begin{defn}
We say that the isotropic homogeneous penetrable obstacle $D$ (which is a bounded domain in $\mathbb{R}^{n}$) is nonscattering with respect to some external source $\mu \in \mathscr{E}'(D)$ and the incident field $u_{0}$ as in \eqref{eq:anisotropic-u0}, if there exists a $u^{\rm to}$, which is in $H_{\rm loc}^{2}$ near $\mathbb{R}^{n}\setminus D$, such that 
\begin{equation}
( \nabla\cdot \tilde{A} \nabla + k_{0}^{2} \tilde{\rho} )u^{\rm to} = -\mu \text{ in $\mathbb{R}^{n}$}, \quad u^{\rm to}|_{\mathbb{R}^{n}\setminus D} = u_{0}|_{\mathbb{R}^{n}\setminus D}, \label{eq:anisotropic-nonscattering}
\end{equation}
where $\tilde{A} := A\chi_{\mathbb{R}^{n}\setminus D} + {\rm Id} \chi_{D}$ and $\tilde{\rho} := \rho \chi_{\mathbb{R}^{n}\setminus D} + 1 \chi_{D}$. 
\end{defn}

By writing $u^{\rm sc} := u^{\rm to} - u_{0}$ in $\mathbb{R}^{n}$, one observes  that, in the case when $D$ is a bounded Lipschitz domain in $\mathbb{R}^{n}$, \eqref{eq:anisotropic-nonscattering} is equivalent to the following transmission problem
\begin{equation}
\begin{cases}
(\Delta + k_{0}^{2}) u^{\rm sc} = -\mu + h \quad \text{in $D$,} \\
u^{\rm sc}|_{\mathbb{R}^{n} \setminus D}=0 ,\quad \partial_{\nu}^{-} u^{\rm sc}|_{\partial D} = -g,
\end{cases} \label{eq:Bernoulli-ITP}
\end{equation}
where 
\begin{equation}
g:=-\nu \cdot (A - {\rm Id}) \nabla u_{0}|_{\partial D} \in L^{\infty}(\partial D) , \quad h = -(\Delta + k_{0}^{2})u_{0} \in L^{\infty}(D). \label{eq:def-g-h}
\end{equation}
Here we used that $\partial_{\nu}^{-} u^{\rm sc}|_{\partial D} = \partial_{\nu}^{-} (u^{\rm to} - u_0)|_{\partial D} = \tilde{A} \nabla^- u^{\rm to} \cdot \nu|_{\partial D} - \p_{\nu} u_0|_{\partial D} = A \nabla^+ u_0 \cdot \nu|_{\partial D} - \p_{\nu} u_0|_{\partial D}$ since $u^{\rm to}$ is $H^2_{\rm loc}$ near $\partial D$. Based on the above observation, we now able to prove the following theorem. 

\begin{thm}
Let $m > \frac{n}{2}$ be an integer, let $\rho \in C_{\rm loc}^{m-1,1}(\mathbb{R}^{n})$, let $A \in (C_{\rm loc}^{m,1}(\mathbb{R}^{n}))_{\rm sym}^{n \times n}$ satisfy the uniform ellipticity condition \eqref{eq:ellipticity-condition}, and let $u_{0}$ be an incident field as in \eqref{eq:anisotropic-u0}. 
If $D$ is a bounded Lipschitz domain in $\mathbb{R}^{n}$ such that it is a hybrid $k$-quadrature domain corresponding to distribution $\mu \in \mathscr{E}'(D)$ and density $(g,h)$ as in \eqref{eq:def-g-h}, then there exists a total field $u^{\rm to}$ satisfying \eqref{eq:anisotropic-nonscattering}.  
\end{thm}

\begin{proof}
Let $\Psi_{k_{0}}$ be any fundamental solution for $-(\Delta + k_{0}^{2})$ in $\mathbb{R}^{n}$, and define 
\[
u^{\rm sc} := \Psi_{k_{0}}*\mu - \Psi_{k_{0}}*(h\chi_{D}) - \mathsf{SL}\,(g) \quad \text{in $\mathbb{R}^{n}$.}
\]
Since $D$ is a hybrid $k$-quadrature domain, by Remark~{\rm \ref{rem:hybrid-QD-Lipschitz}} we know that $u^{\rm sc}|_{\mathbb{R}^{n}\setminus D} = 0$. Since $u^{\rm sc} \in \mathscr{E}'(\mathbb{R}^{n})$ and $\Psi_{k_{0}}*\mu - \Psi_{k_{0}}*(h\chi_{D})$ is $C^{1}$ near $\partial D$, similarly as in \eqref{eq:jump-relation} one sees that $u^{\rm sc}$ satisfies \eqref{eq:Bernoulli-ITP}. By using the equivalence of \eqref{eq:anisotropic-nonscattering} and \eqref{eq:Bernoulli-ITP}, we conclude the theorem. 
\end{proof}

\section{\label{sec:Existence-of-minimizers}Existence of minimizers}

We first show the boundedness of the functional $\mathcal{J}_{f,g,\lambda,\Omega}$ given in \eqref{jfunctional}.
\begin{lem}
\label{lem:boundedness} Let $f,g\in L^{\infty}(\mathbb{R}^{n})$
with $g\ge0$, and let $|\Omega|<\infty$. If $-\infty<\lambda<\lambda^{*}(\Omega)$,
then $\mathcal{J}_{f,g,\lambda,\Omega}$ is coercive in $H_{0}^{1}(\Omega)$.
If $\lambda>\lambda^{*}(\Omega)$, then $\mathcal{J}_{f,g,\lambda,\Omega}$
is unbounded from below in $\mathbb{K}(\Omega)$. 
\end{lem}

\begin{rem}
Here we allow ${\rm supp}\,(f_{+})$ to be unbounded. 
\end{rem}

\begin{proof}
[Proof of Lemma~{\rm \ref{lem:boundedness}}] If $\lambda<\lambda^{*}(\Omega)$,
then there exists a $\gamma>0$ such that 
\[
\int_{\Omega}(|\nabla u|^{2}-\lambda|u|^{2})\,dx\ge\gamma\|u\|_{H^{1}(\Omega)}^{2}\quad\text{for all }u\in H_{0}^{1}(\Omega).
\]
Observe that 
\[
\bigg|2\int_{\Omega}fu\bigg|\le 2 \|u\|_{L^{2}(\Omega)}|\Omega|^{\frac{1}{2}}\|f\|_{L^{\infty}(\Omega)}\le\epsilon\|u\|_{L^{2}(\Omega)}^{2}+\epsilon^{-1}|\Omega|\|f\|_{L^{\infty}(\Omega)}^{2}
\]
for all $\epsilon>0$. Consequently, we have 
\begin{align*}
\mathcal{J}_{f,g,\lambda,\Omega}(u) & =\int_{\Omega}(|\nabla u|^{2}-\lambda|u|^{2})\,dx-2\int_{\Omega}fu\,dx+\int_{\Omega}g^{2}\chi_{\{u>0\}}\,dx\\
 & \ge\gamma\|u\|_{H^{1}(\Omega)}^{2}-\epsilon\|u\|_{L^{2}(\Omega)}^{2}-\epsilon^{-1}|\Omega|\|f\|_{L^{\infty}(\Omega)}^{2}.
\end{align*}
Choosing $\epsilon=\frac{\gamma}{2}$, we reach 
\begin{equation}
\mathcal{J}_{f,g,\lambda,\Omega}(u)\ge\frac{\gamma}{2}\|u\|_{H^{1}(\Omega)}^{2}-\frac{2|\Omega|\|f\|_{L^{\infty}(\Omega)}^{2}}{\gamma}\quad\text{for all }u\in H_{0}^{1}(\Omega).\label{eq:coercivity}
\end{equation}
This proves the claim for $\lambda<\lambda^{*}(\Omega)$. 

We now consider the case when $\lambda>\lambda^{*}(\Omega)$. There
exists $u\in C_{c}^{\infty}(\Omega)$ so that $\|\nabla u\|_{L^{2}(\Omega)}^{2}<\lambda\|u\|_{L^{2}(\Omega)}^{2}$.
Moreover, $|u|\in H_{0}^{1}(\Omega)$ and 
\[
\|\nabla|u|\|_{L^{2}(\Omega)}^{2}=\|\nabla u\|_{L^{2}(\Omega)}^{2}<\lambda\|u\|_{L^{2}(\Omega)}^{2}.
\]
Therefore, we know that $t|u|\in\mathbb{K}(\Omega)$ for all $t\ge0$.
Hence we know that 
\begin{align*}
\mathcal{J}_{f,g,\lambda,\Omega}(t|u|) & =t^{2}(\|\nabla|u|\|_{L^{2}(\Omega)}^{2}-\lambda\|u\|_{L^{2}(\Omega)}^{2})-2t\int_{\Omega}f|u|\,dx+\int_{\Omega}g^{2}\chi_{\{|u|>0\}}\,dx\\
 & \le t^{2}(\overbrace{\|\nabla|u|\|_{L^{2}(\Omega)}^{2}-\lambda\|u\|_{L^{2}(\Omega)}^{2}}^{<0})-2t\overbrace{\int_{\Omega}f|u|\,dx}^{>0}+\overbrace{|{\rm supp}\,(u)|\|g\|_{L^{\infty}(\Omega)}^{2}}^{<\infty},
\end{align*}
which implies 
\[
\limsup_{t\rightarrow\infty}\mathcal{J}_{f,g,\lambda,\Omega}(t|u|)=-\infty,
\]
which proves the second claim of the lemma. 
\end{proof}
\begin{rem}
If $u\in\mathbb{K}(\Omega)$, by observing that 
\begin{align*}
\mathcal{J}_{f,g,\lambda,\Omega}(u) & =\int_{\Omega}(|\nabla u|^{2}-\lambda|u|^{2})\,dx-2\int_{\Omega}f_{+}u\,dx+\overbrace{2\int_{\Omega}f_{-}u\,dx}^{\ge0}+\int_{\Omega}|g|^{2}\chi_{\{u>0\}}\,dx\\
 & \ge\gamma\|u\|_{H^{1}(\Omega)}^{2}-\epsilon\|u\|_{L^{2}(\Omega)}^{2}-\epsilon^{-1}|\Omega|\|f_{+}\|_{L^{\infty}(\Omega)}^{2}.
\end{align*}
we can obtain 
\begin{equation}
\mathcal{J}_{f,g,\lambda,\Omega}(u)\ge\frac{\gamma}{2}\|u\|_{H^{1}(\Omega)}^{2}-\frac{2|\Omega|\|f_{+}\|_{L^{\infty}(\Omega)}^{2}}{\gamma}\quad\text{for all }u\in\mathbb{K}(\Omega).\label{eq:coercivity-refined}
\end{equation}
Note that \eqref{eq:coercivity-refined} is a refinement of \eqref{eq:coercivity}
for functions in $\mathbb{K}(\Omega)$.
\end{rem}

Using Lemma~\ref{lem:boundedness} and following the proof of \cite[Theorem~1 of Section~8.2]{Eva98PDE},
we have the following lemma: 
\begin{lem}
\label{lem:LSC} Let $f,g\in L^{\infty}(\mathbb{R}^{n})$ with $g\ge0$.
Assume that $\Omega$ is bounded with Lipschitz boundary and $-\infty<\lambda<\lambda^{*}(\Omega)$.
Then $\mathcal{J}_{f,g,\lambda,\Omega}$ is weakly lower semi-continuous
on $H_{0}^{1}(\Omega)$, that is, 
\begin{equation}
\mathcal{J}_{f,g,\lambda,\Omega}(u)\le\liminf_{j\rightarrow\infty}\mathcal{J}_{f,g,\lambda,\Omega}(u_{j})\label{eq:LSC}
\end{equation}
whenever $\{u_{j}\}_{j=1}^{\infty}\cup\{u\}\subset H_{0}^{1}(\Omega)$
satisfies 
\[
\begin{cases}
u_{j}\rightarrow u & \text{weakly in }L^{2}(\Omega),\\
\nabla u_{j}\rightarrow\nabla u & \text{weakly in }L^{2}(\Omega).
\end{cases}
\]
\end{lem}

\begin{rem}
Here we remind the readers that the proof of Lemma~\ref{lem:LSC}
involves the compact embedding $H^{1}(\Omega)\hookrightarrow L^{2}(\Omega)$, which follows from the Rellich-Kondrachov theorem as long as there is a bounded extension operator from $H^1(\Omega)$ to $H^1(\mathbb{R}^n)$ which is true e.g.\ for Lipschitz domains.
\end{rem}

Using Lemma~\ref{lem:LSC} and following the proof of \cite[Section~8.2]{Eva98PDE},
we have the following proposition: 
\begin{prop}
\label{prop:existence-minimizer} Let $f,g\in L^{\infty}(\mathbb{R}^{n})$
with $g\ge0$. 
Assume further that $\Omega$ is bounded open set  with Lipschitz boundary
and $-\infty<\lambda<\lambda^{*}(\Omega)$. Then there exists a global minimizer $u_{*}\in\mathbb{K}(\Omega)$
of the functional $\mathcal{J}_{f,g,\lambda,\Omega}$ in $\mathbb{K}(\Omega)$, that
is, 
\[
\mathcal{J}_{f,g,\lambda,\Omega}(u_{*})=\min_{u\in\mathbb{K}(\Omega)}\mathcal{J}_{f,g,\lambda,\Omega}(u).
\]
\end{prop}

\begin{rem}\label{rem:compact} 
Let $u_{0} \in \mathbb{K}(\Omega)$ be any global minimizer of $\mathcal{J}_{f,g,\lambda,\Omega}$ in $\mathbb{K}(\Omega)$. Using \eqref{eq:coercivity-refined}, we see that 
\[
0 = \mathcal{J}_{f,g,\lambda,\Omega}(0) \ge \mathcal{J}_{f,g,\lambda,\Omega}(u_{0}) \ge \frac{\gamma}{2} \| u_{0} \|_{H^{1}(\Omega)}^{2} - \frac{2|\Omega| \| f_{+} \|_{L^{\infty}(\Omega)}^{2}}{\gamma},
\]
which shows that $\| u_{0} \|_{H^{1}(\Omega)} \le C$ for some constant $C$ independent of $u_{0}$. 
In particular, the set of minimizers of $\mathcal{J}_{f,g,\lambda,\Omega}$
in $\mathbb{K}(\Omega)$ is compact in $L^{2}(\Omega)$. 
\end{rem}

\begin{rem}\label{rem:nontrivial1} 
From \eqref{eq:coercivity-refined}, we know that if $f \le 0$ in $\Omega$, then the minimum is zero and attained only by $u=0$. 
\end{rem}

\begin{rem}\label{rem:nontrivial2} 
If the set $\{f>0\}\cap\{g=0\}$ has non-empty
interior in $\Omega$, then $\mathcal{J}_{f,g,\lambda,\Omega}(t\phi)$
is negative for any non-trivial $\phi\in C_{c}^{\infty}(\{f>0\}\cap\{g=0\}\cap\Omega)$
with $t>0$ sufficiently small. Consequently, we have 
\[
\inf_{u\in\mathbb{K}(\Omega)}\mathcal{J}_{f,g,\lambda,\Omega}(u)<0,
\]
and then all minimizers are non-trivial. 
\end{rem}

\section{\label{sec:Minimizers-and-PDE}The Euler-Lagrange equation}

In order to generalize some of our results, we introduce the following
definition: 
\begin{defn}
Let $\Omega$ be an open set in $\mathbb{R}^{n}$ and let $\lambda\in\mathbb{R}$.
A function $u_{*}\in\mathbb{K}(\Omega)$ is called a \emph{local minimizer}
of $\mathcal{J}_{f,g,\lambda,\Omega}$ in $\mathbb{K}(\Omega)$ if
there exists $\epsilon>0$ such that 
\[
\mathcal{J}_{f,g,\lambda,\Omega}(u_{*})\le\mathcal{J}_{f,g,\lambda,\Omega}(u)
\]
for all $u\in\mathbb{K}(\Omega)$ with 
\begin{equation}
\int_{\Omega}(|\nabla(u-u_{*})|^{2}+|\chi_{\{u>0\}}-\chi_{\{u_{*}>0\}}|)\,dx<\epsilon.\label{eq:local-minimizer-metric}
\end{equation}
\end{defn}

Clearly each (global) minimizer is also a local minimizer. We first prove
the following proposition, which is an extension of \cite[Lemma~2.2]{GS96FreeBoundaryPotential}. 
In Proposition~\ref{prop:PDE2} we give an extension of \cite[Theorem~2.3]{GS96FreeBoundaryPotential}. 

\begin{prop}
\label{prop:PDE1} Let $f,g\in L^{\infty}(\mathbb{R}^{n})$ be such
that $g\ge0$. Let $\Omega$ be an open set in $\mathbb{R}^{n}$,
and let $\lambda\in\mathbb{R}$. If $u_{*}\in\mathbb{K}(\Omega)$
is a local minimizer of $\mathcal{J}_{f,g,\lambda,\Omega}$ in $\mathbb{K}(\Omega)$,
then\footnote{When $\lambda\ge0$, \eqref{eq:PDE1a} implies that $(\Delta+\lambda)u_{*}\ge-f_{+}$ in $\Omega$.} 
\begin{subequations}
\begin{align}
\Delta u_{*} & \ge-(f+\lambda u_{*})_{+}\quad\text{in }\Omega,\label{eq:PDE1a}\\
\Delta u_{*} & =-(f+\lambda u_{*})\quad\text{in }\{u_{*}>0\},\label{eq:PDE1b}\\
\Delta u_{*} & \le-(f+\lambda u_{*})\quad\text{in }\Omega\setminus{\rm supp}\,(g).\label{eq:PDE1c}
\end{align}
\end{subequations}
\end{prop}

\begin{proof}
Let $0\le\phi\in C_{c}^{\infty}(\Omega)$ and for each $\epsilon>0$
we define $v_{\epsilon}:=(u_{*}-\epsilon\phi)_{+}$. Since $u_{*}\in\mathbb{K}(\Omega)$,
we know that 
\[
v_{\epsilon}\in\mathbb{K}(\Omega)\quad\text{and}\quad0\le v_{\epsilon}\le u_{*}\text{ in }\Omega.
\]
Since $u_{*}\in\mathbb{K}(\Omega)$ is a local minimizer, then 
\begin{equation}
\mathcal{J}_{f,g,\lambda,\Omega}(u_{*})\le\mathcal{J}_{f,g,\lambda,\Omega}(v_{\epsilon})\quad\text{for all sufficiently small }\epsilon>0.\label{eq:local-min-epsilon}
\end{equation}

\medskip

\noindent We observe that 
\begin{align*}
 & \quad\mathcal{J}_{f,g,\lambda,\Omega}(v_{\epsilon})-\mathcal{J}_{f,g,\lambda,\Omega}(u_{*})
 =\int_{\Omega}(|\nabla v_{\epsilon}|^{2}-\lambda|v_{\epsilon}|^{2})\,dx-\int_{\Omega}(|\nabla u_{*}|^{2}-\lambda|u_{*}|^{2})\,dx
 \\& 
 \quad-2\int_{\Omega}f(v_{\epsilon}-u_{*})\,dx+\int_{\Omega}g^{2}(\chi_{\{v_{\epsilon}>0\}}-\chi_{\{u_{*}>0\}})\,dx\\
 & =\int_{\{v_{\epsilon}>0\}}|\nabla(u_{*}-\epsilon\phi)|^{2}\,dx-\int_{\Omega}|\nabla u_{*}|^{2}\,dx
 \quad-\lambda\bigg(\int_{\{v_{\epsilon}>0\}}|u_{*}-\epsilon\phi|^{2}\,dx-\int_{\Omega}|u_{*}|^{2}\,dx\bigg)\\
 & \quad+2\epsilon\int_{\{v_{\epsilon}>0\}}f\phi\,dx+2\int_{\{v_{\epsilon}=0\}}fu_{*}\,dx-\int_{\{v_{\epsilon}=0\}\cap\{u_{*}>0\}}|g|^{2}\,dx\\
 & =-2\epsilon\int_{\{v_{\epsilon}>0\}}\nabla u_{*}\cdot\nabla\phi\,dx+\epsilon^{2}\int_{\{v_{\epsilon}>0\}}|\nabla\phi|^{2}\,dx\overbrace{-\int_{\{v_{\epsilon}=0\}}|\nabla u_{*}|^{2}\,dx}^{\le\,0}\\
 & \quad+\overbrace{2\epsilon\int_{\{v_{\epsilon}>0\}}(f+\lambda u_{*})\phi\,dx}^{\le\,2\epsilon\int_{\{v_{\epsilon}>0\}}(f+\lambda u_{*})_{+}\phi\,dx}+\overbrace{2\int_{\{v_{\epsilon}=0\}}(f+\lambda u_{*})u_{*}\,dx}^{\le\,2\int_{\{v_{\epsilon}=0\}}(f+\lambda u_{*})_{+}u_{*}\,dx}\\
 & \quad\overbrace{-\lambda\epsilon^{2}\int_{\{v_{\epsilon}>0\}}|\phi|^{2}\,dx-\int_{\{v_{\epsilon}=0\}\cap\{u_{*}>0\}}|g|^{2}\,dx}^{\le\,0}\\
 & \le-2\epsilon\int_{\{v_{\epsilon}>0\}}\nabla u_{*}\cdot\nabla\phi\,dx+\epsilon^{2}\int_{\{v_{\epsilon}>0\}}|\nabla\phi|^{2}\,dx\\
 & \quad+2\epsilon\int_{\{v_{\epsilon}>0\}}(f+\lambda u_{*})_{+}\phi\,dx+\overbrace{2\int_{\{v_{\epsilon}=0\}}(f+\lambda u_{*})_{+}u_{*}\,dx}^{\le\,2\epsilon\int_{\{v_{\epsilon}=0\}}(f+\lambda u_{*})_{+}\phi\,dx}\\
 & \le-2\epsilon\int_{\{v_{\epsilon}>0\}}\nabla u_{*}\cdot\nabla\phi\,dx+2\epsilon\int_{\Omega}(f+\lambda u_{*})_{+}\phi\,dx+\epsilon^{2}\int_{\{v_{\epsilon}>0\}}|\nabla\phi|^{2}\,dx.
\end{align*}
This implies that 
\[
\limsup_{\epsilon\searrow0}\frac{1}{2\epsilon}(\mathcal{J}_{f,g,\lambda,\Omega}(v_{\epsilon})-\mathcal{J}_{f,g,\lambda,\Omega}(u_{*}))\le-\int_{\Omega}\nabla u_{*}\cdot\nabla\phi\,dx+\int_{\Omega}(f+\lambda u_{*})_{+}\phi\,dx.
\]
Combining this inequality with \eqref{eq:local-min-epsilon}, we conclude
\eqref{eq:PDE1a}. 

\medskip

\noindent If ${\rm supp}\,(\phi)\subset\{u_{*}>0\}$,
we define $v_{\epsilon}^{\pm}=u_{*}\pm\epsilon\phi\in\mathbb{K}(\Omega)$ for $\epsilon > 0$ small.
Note that $\{v_{\epsilon}^{\pm}>0\}=\{u_{*}>0\}$ for 
small $\epsilon>0$. Therefore, for small $\epsilon>0$,
we have 
\begin{align*}
 & \quad\mathcal{J}_{f,g,\lambda,\Omega}(v_{\epsilon}^{\pm})-\mathcal{J}_{f,g,\lambda,\Omega}(u_{*})\\
 & =\int_{\{u_{*}>0\}}(|\nabla v_{\epsilon}^{\pm}|^{2}-\lambda|v_{\epsilon}^{\pm}|^{2})\,dx-\int_{\{u_{*}>0\}}(|\nabla u_{*}|^{2}-\lambda|u_{*}|^{2})\,dx-2\int_{\{u_{*}>0\}}f(v_{\epsilon}^{\pm}-u_{*})\,dx\\
 & =\pm2\epsilon\bigg(\int_{\{u_{*}>0\}}\nabla u_{*}\cdot\nabla\phi\,dx-\lambda\int_{\{u_{*}>0\}}u_{*}\phi\,dx\bigg)\mp2\epsilon\int_{\{u_{*}>0\}}f\phi\,dx\\
 & \quad+\epsilon^{2}\bigg(\int_{\{u_{*}>0\}}|\nabla\phi|^{2}\,dx-\lambda\int_{\{u_{*}>0\}}|\phi|^{2}\,dx\bigg).
\end{align*}
Using \eqref{eq:local-min-epsilon} and dividing both sides on the
inequality above by $\epsilon$ and letting $\epsilon\rightarrow0$,
we conclude \eqref{eq:PDE1b}. 

\medskip

\noindent If ${\rm supp}\,(\phi)\cap{\rm supp}\,(g)=\emptyset$,
taking $\tilde{v}_{\epsilon}=u_{*}+\epsilon\phi$, we have 
\begin{align*}
 & \quad\mathcal{J}_{f,g,\lambda,\Omega}(\tilde{v}_{\epsilon})-\mathcal{J}_{f,g,\lambda,\Omega}(u_{*})\\
 & =\int_{\Omega}(|\nabla\tilde{v}_{\epsilon}|^{2}-\lambda|\tilde{v}_{\epsilon}|^{2})\,dx-\int_{\Omega}(|\nabla u_{*}|^{2}-\lambda|u_{*}|^{2})\,dx-2\int_{\Omega}f(\tilde{v}_{\epsilon}-u_{*})\,dx\\
 & =2\epsilon\bigg(\int_{\Omega\setminus{\rm supp}\,(g)}\nabla u_{*}\cdot\nabla\phi\,dx-\lambda\int_{\Omega\setminus{\rm supp}\,(g)}u_{*}\phi\,dx\bigg)-2\epsilon\int_{\Omega\setminus{\rm supp}\,(g)}f\phi\,dx\\
 & \quad+\epsilon^{2}\bigg(\int_{\Omega}|\nabla\phi|^{2}\,dx-\lambda\int_{\Omega}|\phi|^{2}\,dx\bigg).
\end{align*}
Using \eqref{eq:local-min-epsilon} and dividing both sides of the
inequality above by $\epsilon$ and letting $\epsilon\rightarrow0$,
we conclude \eqref{eq:PDE1c}. 
\end{proof}

\begin{rem}\label{rem:Lipschitz-boundary-case}
We now give some observations when $\{u_{*}>0\}$ has Lipschitz boundary. Using integration by parts on the Lipschitz domain $\{u_{*}>0\}$ and from \eqref{eq:PDE1b} we have 
\begin{equation}
\mathcal{J}_{f,g,\lambda,\Omega}(u_{*})=\int_{\Omega}(g^{2}\chi_{\{u_{*}>0\}}-fu_{*})\,dx.\label{eq:representation-local-min}
\end{equation}
If $u_{*}$ is a local minimizer of $\mathcal{J}_{f,g,\lambda,\Omega}$
in $\mathbb{K}(\Omega)$ with $\mathcal{J}_{f,g,\lambda,\Omega}(u_{*})<0$, using \eqref{eq:representation-local-min} we know that 
\[
\int_{\{u_{*}>0\}\cap\{f>0\}}fu_{*}\,dx\equiv\int_{\Omega}fu_{*}\,dx>\int_{\Omega}g^{2}\chi_{\{u_{*}>0\}}\,dx\ge0,
\]
which immediately implies $|\{u_{*}>0\}\cap\{f>0\}|>0$. 
If we additionally assume that $g\equiv0$ and $\lambda\ge0$, from \eqref{eq:PDE1c} we know that $(\Delta+\lambda)u_{*}+f\le0$ in
$\Omega$. From this, we know that 
\[
\Delta u_{*}\le-f-\lambda u_{*}\le0\quad\text{in }\{f>0\}\cap\Omega,
\]
because $u_{*}\ge0$ in $\Omega$. Using strong minimum principle
for super-solutions (as formulated in \cite[Theorem~8.19]{GT01Elliptic}),
we know that $u_{*}>0$ in $\{f>0\}\cap\Omega$. 
\end{rem}

\section{\label{sec:Comparison-of-minimizers:}Comparison of minimizers}

The main purpose of this section is to prove the $L^{\infty}$-regularity
of the minimizers. We first prove the following comparison principle analogous to \cite[Lemma 1.1]{GS96FreeBoundaryPotential}.

\begin{prop}
\label{prop:comparison} For  $f_{m},g_{m}\in L^{\infty}(\mathbb{R}^{n})$
with $g_{m}\ge0$ for each $m=1,2$, suppose
\begin{equation}
f_{1}\le f_{2},\quad g_{1}\ge g_{2}\ge0\quad\text{in }\mathbb{R}^{n}.
\label{eq:comparison}
\end{equation}
Assume further that $-\infty<\lambda_{1}\le\lambda_{2}<+\infty$ and that $\Omega_{m} \subset  \mathbb{R}^{n}$ are open sets satisfying $\Omega_{1}\subset\Omega_{2}$.
For each $u_{m}\in\mathbb{K}(\Omega_{m})$, we define $v:=\min\{u_{1},u_{2}\}$ and $w:=\max\{u_{1},u_{2}\}$. Then
$v\in\mathbb{K}(\Omega_{1})$, $w\in\mathbb{K}(\Omega_{2})$
satisfy 
\[
\mathcal{J}_{f_{1},g_{1},\lambda_{1},\Omega_{1}}(v)+\mathcal{J}_{f_{2},g_{2},\lambda_{2},\Omega_{2}}(w)\le\mathcal{J}_{f_{1},g_{1},\lambda_{1},\Omega_{1}}(u_{1})+\mathcal{J}_{f_{2},g_{2},\lambda_{2},\Omega_{2}}(u_{2}).
\]
We also have the following statements: 
\begin{enumerate}[leftmargin = 25pt]
\renewcommand{\labelenumi}{\theenumi}
\renewcommand{\theenumi}{{\rm (\arabic{enumi})}}
\item If $u_{1}$ is a (global) minimizer of $\mathcal{J}_{f_{1},g_{1},\lambda_{1},\Omega_{1}}$
in $\mathbb{K}(\Omega_{1})$, then
\[
\mathcal{J}_{f_{2},g_{2},\lambda_{2},\Omega_{2}}(w)\le\mathcal{J}_{f_{2},g_{2},\lambda_{2},\Omega_{2}}(u_{2}).
\]
\item If $u_{2}$ is a (global) minimizer of $\mathcal{J}_{f_{2},g_{2},\lambda_{2},\Omega_{2}}$
in $\mathbb{K}(\Omega_{2})$, then
\[
\mathcal{J}_{f_{1},g_{1},\lambda_{1},\Omega_{1}}(v)\le\mathcal{J}_{f_{1},g_{1},\lambda_{1},\Omega_{1}}(u_{1}).
\]
\item If each $u_{m}$ is a (global) minimizer of $\mathcal{J}_{f_{m},g_{m},\lambda_{m},\Omega_{m}}$
in $\mathbb{K}(\Omega_{m})$, respectively, then $v$ is a (global)
minimizer of $\mathcal{J}_{f_{1},g_{1},\lambda_{1},\Omega_{1}}$ in
$\mathbb{K}(\Omega_{1})$, while $w$ is a (global) minimizer of $\mathcal{J}_{f_{2},g_{2},\lambda_{2},\Omega_{2}}$
in $\mathbb{K}(\Omega_{2})$. 
\end{enumerate}
\end{prop}

\begin{rem}
When $\Omega_{1}=\Omega_{2}=\Omega$, we only need to assume \eqref{eq:comparison}
in $\Omega$. 
\end{rem}

\begin{proof}[Proof of Proposition~{\rm \ref{prop:comparison}}]
We first show that, if $\Phi(t)$
is a non-decreasing function of $t\in[0,\infty)$ and $h_{1}\le h_{2}$
in $\mathbb{R}^{n}$, then 
\begin{equation}
\int_{\mathbb{R}^{n}}(h_{1}\Phi(u_{1})+h_{2}\Phi(u_{2}))\,dx\le\int_{\mathbb{R}^{n}}(h_{1}\Phi(v)+h_{2}\Phi(w))\,dx.\label{eq:comparison-observation}
\end{equation}
Indeed, by definition we have 
\begin{align*}
\int_{\mathbb{R}^{n}}h_{1}(\Phi(u_{1})-\Phi(v))\,dx & =\int_{\{u_{1}\ge u_{2}\}}h_{1}(\Phi(u_{1})-\Phi(u_{2}))\,dx\\
 & \le\int_{\{u_{1}\ge u_{2}\}}h_{2}(\Phi(u_{1})-\Phi(u_{2}))\,dx\\
 & =\int_{\mathbb{R}^{n}}h_{2}(\Phi(w)-\Phi(u_{2}))\,dx,
\end{align*}
where we used that $\Phi(u_{1})-\Phi(u_{2})\ge0$ in $\{u_{1}\ge u_{2}\}$.
Hence we conclude \eqref{eq:comparison-observation}.

\medskip

\noindent From \eqref{eq:comparison-observation}, we find that
\begin{align*}
\int_{\mathbb{R}^{n}}(f_{1}u_{1}+f_{2}u_{2})\,dx & \le\int_{\mathbb{R}^{n}}(f_{1}v+f_{2}w)\,dx,\quad(\text{choosing }h_{j}=f_{j}\text{ and }\Phi(t)=t)\\
\int_{\mathbb{R}^{n}}(k_{1}^{2}u_{1}^{2}+k_{2}^{2}u_{2}^{2})\,dx & \le\int_{\mathbb{R}^{n}}(k_{1}^{2}v^{2}+k_{2}^{2}w^{2})\,dx, \quad(\text{choosing }h_{j}=k_{j}^{2}\text{ and }\Phi(t)=t^{2}).
\end{align*}
On the other hand, choosing $h_{j}=-g_{j}^{2}$ and 
\[
\Phi(t)=\begin{cases}
0, & t\le0,\\
1, & t>0,
\end{cases}
\]
we obtain 
\[
\int_{\mathbb{R}^{n}}(g_{1}^{2}\chi_{\{u_{1}>0\}}+g_{2}^{2}\chi_{\{u_{2}>0\}})\,dx\ge\int_{\mathbb{R}^{n}}(g_{1}^{2}\chi_{\{v>0\}}+g_{2}^{2}\chi_{\{w>0\}})\,dx.
\]
By observing that 
\[
\int_{\mathbb{R}^{n}}(|\nabla u_{1}|^{2}+|\nabla u_{2}|^{2})\,dx=\int_{\mathbb{R}^{n}}(|\nabla v|^{2}+|\nabla w|^{2})\,dx,
\]
we conclude the proposition by putting these inequalities together. 
\end{proof}
We now prove the following lemma using Proposition~\ref{prop:existence-minimizer}
and Proposition~\ref{prop:comparison}: 
\begin{lem}
\label{lem:monotonicity} Let $f,g\in L^{\infty}(\mathbb{R}^{n})$
be such that $g\ge0$. Let $\Omega$ be a bounded domain with $C^{1}$
boundary. Assume that $u_{0}$ is a non-trivial (global) minimizer
of $\mathcal{J}_{f,g,\lambda_{0},\Omega}$ in $\mathbb{K}(\Omega)$
with $-\infty<\lambda_{0}<\lambda^{*}(\Omega)$. Then for each $\lambda$
with $\lambda_{0}<\lambda<\lambda^{*}(\Omega)$, there exists a non-trivial
(global) minimizer  of $\mathcal{J}_{f,g,\lambda,\Omega}$
in $\mathbb{K}(\Omega)$. 
In addition, any (global) minimizer   $u_{*}$ of  $\mathcal{J}_{f,g,\lambda,\Omega}$
in $\mathbb{K}(\Omega)$ satisfies
\begin{subequations}
\begin{align}
u_{*} & \ge u_{0}\quad\text{in }\Omega,\label{eq:monotonicity1a}\\
u_{*} & >u_{0}\quad\text{in }\{u_{0}>0\}.\label{eq:monotonicity1b}
\end{align}
\end{subequations}
\end{lem}

\begin{proof}
We first show existence of nontrivial minimizer.
Using Proposition~\ref{prop:existence-minimizer},
there exists a (global) minimizer $u_{*}$ of $\mathcal{J}_{f,g,\lambda,\Omega}$
in $\mathbb{K}(\Omega)$. Note that 
\[
\inf_{u\in\mathbb{K}(\Omega)}\mathcal{J}_{f,g,\lambda,\Omega}(u)\le\mathcal{J}_{f,g,\lambda,\Omega}(u_{0})=\mathcal{J}_{f,g,\lambda_{0},\Omega}(u_{0})-(\lambda-\lambda_{0})\|u_{0}\|_{L^{2}(\Omega)}^{2}.
\]
Since $u_{0}$ is nontrivial, then $\|u_{0}\|_{L^{2}(\Omega)}^{2}>0$.
Therefore, using \eqref{eq:non-positive-functional}, we have
\[
\inf_{u\in\mathbb{K}(\Omega)}\mathcal{J}_{f,g,\lambda,\Omega}(u)<0,
\]
which shows that 0 is not a minimizer of $\mathcal{J}_{f,g,\lambda,\Omega}(u)$
in $\mathbb{K}(\Omega)$. Consequently  $u_{*}\not\equiv0$. 

\medskip

To prove  \eqref{eq:monotonicity1a},  we  let $u_{*}$
be any (global) minimizer of $\mathcal{J}_{f,g,\lambda,\Omega}(u)$
in $\mathbb{K}(\Omega)$. Using Proposition~\ref{prop:comparison} we know that $w:=\max\{u_{0},u_{*}\}$ is also a (global) minimizer
of $\mathcal{J}_{f,g,\lambda,\Omega}(u)$ in $\mathbb{K}(\Omega)$.
Using \eqref{eq:PDE1b} in Proposition~\ref{prop:PDE1}, we know
that 
\begin{align*}
\Delta u_{0} & =-f-\lambda_{0}u_{0}\quad\text{in }\{u_{0}>0\},\\
\Delta w & =-f-\lambda w\quad\text{in }\{w>0\}.
\end{align*}
Note that $u_{0}=w$ in $\{u_{0}>u_{*}\}$ and $\{u_{0}>u_{*}\}\subset\{w>0\}\cap\{u_{0}>0\}$,
then restricting the above two identities in $\{u_{0}>u_{*}\}$ yields
\[
\Delta w+\lambda w=\Delta w+\lambda_{0}w=-f\quad\text{in }\{u_{0}>u_{*}\}.
\]
Since $\lambda_{0}<\lambda$, then
\[
w=u_{0}=0\quad\text{in }\{u_{0}>u_{*}\},
\]
and hence we know that $|\{u_{0}>u_{*}\}|=0$, which concludes \eqref{eq:monotonicity1a}. 

\medskip

Finally to  prove  \eqref{eq:monotonicity1b},
We define $v=u_{*}-u_{0}\ge0$
in $B_{R}$, and note that 
\begin{align*}
\Delta v & =-(\lambda-\lambda_{0})v\le0\quad\text{in }\{u_{0}>0\},\\
v & \ge0\quad\text{on }\partial\{u_{0}>0\}.
\end{align*}
Using the strong minimum principle for super-solutions (as formulated
in \cite[Theorem~8.19]{GT01Elliptic}), we know that $u>u_{0}$ in
$\{u_{0}>0\}$, because $u\not\equiv u_{0}$, which concludes \eqref{eq:monotonicity1b}. 
\end{proof}
We now prove the the minimizer is unique for all except for countably
many $\lambda$. 
\begin{prop}
\label{prop:unique} We assume that $\Omega$ is bounded with $C^{1}$
boundary and $-\infty<\lambda<\lambda^{*}(\Omega)$. Let $f,g\in L^{\infty}(\Omega)$
with $g\ge0$. Then there exist smallest and largest (in pointwise
sense) minimizers of $\mathcal{J}_{f,g,\lambda,\Omega}$ in $\mathbb{K}(\Omega)$.
We define the functions 
\begin{align*}
m(\lambda) & :=\min \left\{ 
\|u\|_{L^{2}(\Omega)} \, \big| \, \mathcal{J}_{f,g,\lambda,\Omega}(u)={\displaystyle \inf_{v\in\mathbb{K}(\Omega)}}\mathcal{J}_{f,g,\lambda,\Omega}(v) \right\}\quad\text{for all }-\infty<\lambda<\lambda^{*}(\Omega),\\
M(\lambda) & :=\max \left\{
\|u\|_{L^{2}(\Omega)} \, \big| \, \mathcal{J}_{f,g,\lambda,\Omega}(u)={\displaystyle \inf_{v\in\mathbb{K}(\Omega)}}\mathcal{J}_{f,g,\lambda,\Omega}(v) \right\} \quad\text{for all }-\infty<\lambda<\lambda^{*}(\Omega).
\end{align*}
Then the functions 
\[
m:(-\infty,\lambda^{*}(\Omega))\rightarrow\mathbb{R},\quad M:(-\infty,\lambda^{*}(\Omega))\rightarrow\mathbb{R}
\]
are strictly increasing. Moreover, we have 
\begin{equation}
M(\lambda-\epsilon)<m(\lambda)\quad\text{for all }-\infty<\lambda<\lambda^{*}(\Omega)\text{ and }\epsilon>0.\label{eq:minimizer-bound}
\end{equation}
Consequently, there exists a countable set $Z\subset(-\infty,\lambda^{*}(\Omega))$ such that the minimizer of $\mathcal{J}_{f,g,\lambda,\Omega}$ in $\mathbb{K}(\Omega)$ is unique for all $\lambda\in(-\infty,\lambda^{*}(\Omega))\setminus Z$. 
\end{prop}

\begin{proof}
From Remark~\ref{rem:compact}, we know that the set of minimizers
is compact in $L^{2}(\Omega)$. Since $L^{2}(\Omega)$ is separable,
there exists a countable dense set in the set of minimizers. Taking
pointwise supremum, as well as pointwise infimum, in this countable
set produces two new minimizers. This proves the first part of the
proposition, and hence the functions $m$ and $M$ are well-defined. 

The strict monotonicity of $m$ and $M$ follows from Lemma~\ref{lem:monotonicity}.
Choosing $\lambda_{0}=\lambda-\epsilon$ in Lemma~\ref{lem:monotonicity},
we also know that any minimizer of $\mathcal{J}_{f,g,\lambda,\Omega}$
is larger than all minimizers of $\mathcal{J}_{f,g,\lambda-\epsilon,\Omega}$
(in particular the largest one), and we conclude \eqref{eq:minimizer-bound}.
The final claim follows from \eqref{eq:minimizer-bound} and the fact
that monotone functions are continuous except for a countable set
of jump discontinuities. 
\end{proof}
\begin{lem}
We assume that $\Omega$ is bounded with $C^{1}$ boundary. Let $f,g\in L^{\infty}(\Omega)$
with $g\ge0$, and  define 
\[
\Phi(\lambda):=\inf_{v\in\mathbb{K}(\Omega)}\mathcal{J}_{f,g,\lambda,\Omega}(v)\quad\text{for all }\lambda\in(-\infty,\lambda^{*}(\Omega)).
\]
Then $\Phi:(-\infty,\lambda^{*}(\Omega))\rightarrow(-\infty,0]$ is
concave {\rm (}and hence continuous{\rm )}. In addition, if $\Phi(\lambda_{0})<0$,
then it is strictly decreasing near $\lambda=\lambda_{0}$. 
\end{lem}

\begin{proof}
Fix $-\infty<\lambda_{0}<\lambda^{*}(\Omega)$ and let $u_{0}\in\mathbb{K}(\Omega)$
be such that 
\[
\mathcal{J}_{f,g,\lambda_{0},\Omega}(u_{0})=\inf_{v\in\mathbb{K}(\Omega)}\mathcal{J}_{f,g,\lambda_{0},\Omega}(v)\equiv\Phi(\lambda_{0}).
\]
For each $\lambda\in(-\infty,\lambda^{*}(\Omega))$, we have 
\begin{align*}
\Phi(\lambda)=\inf_{v\in\mathbb{K}(\Omega)}\mathcal{J}_{f,g,\lambda,\Omega}(v) & \le\mathcal{J}_{f,g,\lambda,\Omega}(u_{0})=\mathcal{J}_{f,g,\lambda_{0},\Omega}(u_{0})-(\lambda-\lambda_{0})\|u_{0}\|_{L^{2}(\Omega)}^{2}\\
 & =\Phi(\lambda_{0})-(\lambda-\lambda_{0})\|u_{0}\|_{L^{2}(\Omega)}^{2},
\end{align*}
which proves the claimed concavity. In addition, if $\Phi(\lambda_{0})<0$,
then $u_{0}\not\equiv0$ (since $\mathcal{J}_{f,g,\lambda_{0},\Omega}(0)=0$),
which implies that the function is strictly decreasing near $\lambda=\lambda_{0}$. 
\end{proof}
We now prove the $L^{\infty}$-regularity of the minimizers. 
\begin{prop}
\label{prop:regularity-minimizer} We assume that $\Omega$ is bounded
with $C^{1}$ boundary and $-\infty<\lambda<\lambda^{*}(\Omega)$.
Let $f,g\in L^{\infty}(\Omega)$ with $g\ge0$. Let $v\in H_{0}^{1}(\Omega)$
be the unique solution of $(\Delta+\lambda)v=-1$ in $\Omega$. If
$u_{*}\in\mathbb{K}(\Omega)$ is a {\rm (}global{\rm )} minimizer $\mathcal{J}_{f,g,\lambda,\Omega}$
in $\mathbb{K}(\Omega)$, then 
\[
0\le u_{*}(x)\le\|f_{+}\|_{L^{\infty}(\Omega)}v(x)\quad\text{for all }x\in\Omega.
\]
\end{prop}

\begin{rem}
Using \cite[Theorem~8.15]{GT01Elliptic}, we also know that $v\in L^{\infty}(\Omega)$.
Therefore we know that $\|u_{*}\|_{L^{\infty}(\Omega)}\le C(\lambda,\Omega)\|f_{+}\|_{L^{\infty}(\Omega)}$. 
\end{rem}

\begin{proof}
[Proof of Proposition~{\rm \ref{prop:regularity-minimizer}}] In view
of Remark~\ref{rem:nontrivial1}, we only need to consider the case
when $\|f_{+}\|_{L^{\infty}(\Omega)}>0$. We define $v_{0}:=\|f_{+}\|_{L^{\infty}(\Omega)}v$, which 
is a  minimizer of $\mathcal{J}_{\|f_{+}\|_{L^{\infty}(\Omega)},0,\lambda,\Omega}$
in $\mathbb{K}(\Omega)$. Using Proposition~\ref{prop:comparison},
we know that $\max\{u_{*},v_{0}\}$ is also a minimizer of $\mathcal{J}_{\|f_{+}\|_{L^{\infty}(\Omega)},0,\lambda,\Omega}$
in $\mathbb{K}(\Omega)$. By \eqref{eq:PDE1a} and \eqref{eq:PDE1c}
in Proposition~\ref{prop:PDE1}, it follows  that both $v_{0}$ and $\max\{u_{*},v_{0}\}$
satisfy
\begin{equation}
\begin{cases}
(\Delta+\lambda)u = -\|f_{+}\|_{L^{\infty}(\Omega)} & \text{in }\Omega,\\
u = 0 & \text{on }\partial\Omega.
\end{cases}\label{eq:Euler-Lagrange}
\end{equation}
Since $-\infty<\lambda<\lambda^{*}(\Omega)$, the solution of
\eqref{eq:Euler-Lagrange} is unique. The uniqueness of solution of
\eqref{eq:Euler-Lagrange} implies 
\[
v_{0}=\max\{u_{*},v_{0}\}\quad\text{in }\Omega,
\]
which concludes the pointwise bound. 
\end{proof}

\section{\label{sec:Properties-local-minimizers}Some properties of local minimizers}

We shall  study the regularity of local minimizers, and obtain some consequences for the case when $\lambda \ge 0$. 
\begin{lem}
\label{lem:zero-freq} Let $\Omega$ be an open set in $\mathbb{R}^{n}$,
and let $\lambda \ge 0$. Let $f,g\in L^{\infty}(\Omega)$
be such that $g\ge0$. If $u_{*}$ is a local minimizer of $\mathcal{J}_{f,g,\lambda,\Omega}$
in $\mathbb{K}(\Omega)$, then it is also a local minimizer of $\mathcal{J}_{f+\lambda u_{*},g,0,\Omega}$
in $\mathbb{K}(\Omega)$. 
\end{lem}

\begin{proof}
Write $\tilde{f} = f + \lambda u_{*}$. For each $v\in H_{0}^{1}(\Omega)$, we have 
\begin{equation*}
\begin{aligned}
\mathcal{J}_{\tilde{f},g,0,\Omega}(v) & =\int_{\Omega}(|\nabla v|^{2}-2\tilde{f}v+g^{2}\chi_{\{v>0\}})\,dx  =\mathcal{J}_{f,g,\lambda,\Omega}(v)+\lambda\int_{\Omega}(v^{2}-2u_{*}v)\,dx \\
& =\mathcal{J}_{f,g,\lambda,\Omega}(v)+\lambda\int_{\Omega}(v-u_{*})^{2}\,dx-\lambda\int_{\Omega}u_{*}^{2}\,dx.
\end{aligned}
\end{equation*}
Since $\lambda \ge 0$, we see that 
\begin{equation}
\mathcal{J}_{\tilde{f},g,0,\Omega}(v) \ge \mathcal{J}_{f,g,\lambda,\Omega}(u_{*}) - \lambda \int_{\Omega} u_{*}^{2} \, dx = \mathcal{J}_{\tilde{f},g,0,\Omega}(u_{*}), \label{eq:inequality-transfer}
\end{equation}
and the equality holds in \eqref{eq:inequality-transfer} if and only
if $v=u_{*}$. Hence we conclude our lemma. 
\end{proof}

With this lemma at hand, one can prove that the minimizer $u_{*}$ is Lipschitz continuous, as well as some results analogous to \cite[Sections~2 and 5]{GS96FreeBoundaryPotential}, by using the corresponding results in \cite{GS96FreeBoundaryPotential} where one just replaces $f \in L^{\infty}(\Omega)$ with $f + \lambda u_{*} \in L^{\infty}(\Omega)$ (see Proposition~{\rm \ref{prop:regularity-minimizer}}). This works since the proofs in \cite{GS96FreeBoundaryPotential} only rely on variations of $u_{*}$ locally. 
The detailed statements and proofs can be found at Appendix~\ref{appen:details}. 
Here we highlight some results which we will use later. The following proposition concerns the PDE characterization of the minimizer $u_{*}$. 
\begin{prop}
\label{prop:PDE3} Let $\Omega$ be a bounded open set in $\mathbb{R}^{n}$ with $C^{1}$ boundary and $0 \le \lambda<\lambda^{*}(\Omega)$. Let $f,g\in L^{\infty}(\Omega)$ be such that $g\ge0$ and $g^{2}\in W^{1,1}(\Omega)$. Suppose that $u_{*}$ is a local minimizer of $\mathcal{J}_{f,g,\lambda,\Omega}$ in $\mathbb{K}(\Omega)$. If $\partial\{g>0\}\cap\Omega\ne\emptyset$, we further assume that there exists $0<\alpha\le1$ such that $g$ is $C^{\alpha}$ near $\partial\{g>0\}\cap\Omega$ and $\mathscr{H}^{n-1+\alpha}(\partial\{g>0\}\cap\Omega)=0$. We assume that $\overline{\{u_{*}>0\}}\subset\Omega$. Then $\{u_{*}>0\}$ has locally finite perimeter in $\{g>0\}$, 
\begin{equation}
\mathscr{H}^{n-1}((\partial\{u_{*}>0\}\setminus\partial_{{\rm red}}\{u_{*}>0\})\cap\{g>0\})=0,\label{eq:boundary-and-reduced-boundary}
\end{equation}
and 
\begin{equation}
(\Delta+\lambda)u_{*}+f\mathscr{L}^{n}\lfloor\{u_{*}>0\}=g\mathscr{H}^{n-1}\lfloor\partial\{u_{*}>0\}=g\mathscr{H}^{n-1}\lfloor\partial_{{\rm red}}\{u_{*}>0\}.\label{eq:PDE3-main}
\end{equation}
\end{prop}

The following proposition concerns the regularity of the reduced free boundary $\partial_{\rm red}\{u_{*}>0\}$. 

\begin{prop}\label{prop:regularity-free-boundary}
Let $\Omega$ be a bounded open set in $\mathbb{R}^{n}$ with $C^{1}$ boundary and $0 \le \lambda<\lambda^{*}(\Omega)$. Let $f$, $g$ and $u_{*}$ be functions given in Proposition~{\rm \ref{prop:PDE3}}. If there exists a ball $B_{r}(x_{0}) \subset \Omega$ such that $g$ is H\"{o}lder continuous and satisfies $g \ge {\rm constant} > 0$ in $B_{r}(x_{0})$, then $\partial_{\rm red}\{u_{*}>0\}$ is locally $C^{1,\alpha}$ in such a ball $B_{r}(x_{0})$, and in the case when $n=2$ we even have $\partial_{\rm red}\{u_{*}>0\} = \partial \{u_{*}>0\}$. 
\end{prop}

\begin{rem}
If $g>0$ is H\"{o}lder continuous in $\Omega$, together with \eqref{eq:boundary-and-reduced-boundary}  we then know that $\partial_{\rm red} \{u_{*}>0\}$ is locally $C^{1,\alpha}$ with $\mathscr{H}^{n-1}(\partial \{u_{*}>0\} \setminus \partial_{\rm red} \{u_{*}>0\} ) = 0$. 
\end{rem}

\section{\label{sec:Relation-with-quadrature}Relation with hybrid quadrature domains}

We now obtain the following simple lemma: 
\begin{lem}
\label{lem:PDE4} Suppose the assumptions in Proposition~{\rm \ref{prop:PDE3}} hold and write $\lambda = k^{2}$. If we further assume that $\overline{\{u_{*}>0\}}\subset\Omega$ and $f = \mu - h \chi_{\{u_{*} > 0\}} \in L^{\infty}(\Omega)$ for some $\mu \in \mathscr{E}'(\{u_{*} > 0\})$ and $h \in L^{\infty}(\{u_{*} > 0\})$, then $\left( \Psi_{k} * (g \mathscr{H}^{n-1} \lfloor \partial \{u_{*} > 0\}) \right)(x)$ is pointwise well-defined for all $x \in \partial \{u_{*} > 0\}$. Moreover, we also know that $\{u_{*} > 0 \}$ is a hybrid $k$-quadrature domain {\rm (}Definition~{\rm \ref{def:weighted-QD}}{\rm )}, corresponding to distribution $\mu$ and density
$(g,h)$. 
\end{lem}

\begin{proof}
Let $\Psi_{k} \in L_{\rm loc}^{1}(\mathbb{R}^{n}) \cap C^{\infty}(\mathbb{R}^{n}\setminus \{0\})$ be any fundamental solution of the Helmholtz operator $-(\Delta + k^{2})$ and let $D = \{ u_* > 0 \}$. By the properties of convolution for distributions and by \eqref{eq:PDE3-main} we have 
\begin{equation}
\begin{aligned}
u_{*} &= \delta_{0}*u_{*} = -(\Delta + k^{2})\Psi_{k} * u_{*} \\
&= - \Psi_{k}*(\Delta + k^{2})u_{*} = \Psi_{k}*(f \mathscr{L}^{n}\lfloor D - g \mathscr{H}^{n-1} \lfloor \partial D). 
\end{aligned} \label{eq:00-properties-convolution-distribution}
\end{equation}
By using the fact $u_{*} \in C_{\rm loc}^{0,1}(\Omega)$
(see Appendix~{\rm \ref{appen:details}})
and the assumption $\overline{\{u_{*}>0\}}\subset\Omega$, from \eqref{eq:00-properties-convolution-distribution} we conclude the first result. The second result immediately follows from the observation $u_{*}=0$ in $\mathbb{R}^{n}\setminus \{ u_{*} > 0 \}$. 
\end{proof}

Finally, we want to show that there exists some $\mu$ so that 
\[
{\rm supp}\,(\mu)\subset\{u_{*}>0\}\quad\text{and}\quad\overline{\{u_{*}>0\}}\subset\Omega.
\]
We first study a particular radially symmetric case (the case when
$\lambda=0$ was considered in \cite[Lemma~1.2]{GS96FreeBoundaryPotential}): 
\begin{lem}
\label{lem:radially1} Let $\Omega=B_{R}$ {\rm (}with $R>0${\rm )} and $0<\lambda<\lambda^{*}(B_{R})\equiv j_{\frac{n-2}{2},1}^{2}R^{-2}$.
Suppose that 
\[
f=a\chi_{B_{r_{1}}}-b\quad\text{with }a>b>0\text{ and }0<r_{1}<R,
\]
and let $g$ be a radially non-decreasing function $g$ with $g=0$
in $\overline{B_{r_{1}}}$. Then there exists $u_{*}\in\mathbb{K}(B_{R})$
such that 
\[
\mathcal{J}_{f,g,\lambda,B_{R}}(u_{*})=\inf_{v\in\mathbb{K}(B_{R})}\mathcal{J}_{f,g,\lambda,B_{R}}(v)<0.
\]
Moreover, the following holds: 
\begin{enumerate}[leftmargin = 25pt]
\renewcommand{\labelenumi}{\theenumi}
\renewcommand{\theenumi}{{\rm (\arabic{enumi})}}
\item Any {\rm (}global{\rm )} minimizer $u_{*}$ of $\mathcal{J}_{f,g,\lambda,B_{R}}$ in $\mathbb{K}(B_{R})$
 is continuous, 
radially symmetric and radially non-increasing,
and satisfies 
\begin{equation}
\overline{B_{r_{1}}}\subset\{u_{*}>0\}.\label{eq:minimizer-interesting}
\end{equation}
\item If we set 
\begin{equation}
R'=\max\begin{Bmatrix}\begin{array}{l|l}
\rho\in(r_{1},R] & {\displaystyle \frac{b}{a}-\frac{r_{1}^{\frac{n}{2}}J_{\frac{n}{2}}(\sqrt{\lambda}r_{1})}{\rho^{\frac{n}{2}}J_{\frac{n}{2}}(\sqrt{\lambda}\rho)}}\le0\end{array}\end{Bmatrix}>r_{1},\label{eq:R-prime-choice}
\end{equation}
\end{enumerate}
then  $u_{*}$ has support in the ball $\overline{B_{R'}}$.
In particular  $R'<R$, whenever
\[
\frac{b}{a}>\frac{r_{1}^{\frac{n}{2}}J_{\frac{n}{2}}(\sqrt{\lambda}r_{1})}{R^{\frac{n}{2}}J_{\frac{n}{2}}(\sqrt{\lambda}R)}.
\]

\end{lem}

\begin{proof}
The existence of minimizers
was established in Proposition~\ref{prop:existence-minimizer}. Since
\[
B_{r_{1}}\subset\{f>0\}\cap\{g=0\}\cap B_{R},
\]
from Remark~\ref{rem:nontrivial2}, we know that all minimizers are
nontrivial. 

\bigskip

\noindent \textbf{Step 1: (Rearrangement)} Given
any $u\in\mathbb{K}(B_{R})$,  let $u^{{\rm rad}}$ denote its
radially symmetric decreasing rearrangement, that is, 
\[
u^{{\rm rad}}(x):=\int_{0}^{\infty}\chi_{\{u>t\}^{{\rm rad}}}(x)\,dt,
\]
where $A^{{\rm rad}}=\begin{Bmatrix}\begin{array}{l|l}
x\in\mathbb{R}^{n} & \omega_{n}|x|^{n}<|A|\end{array}\end{Bmatrix}$ for any measurable set $A\subset\mathbb{R}^{n}$. Then $u^{{\rm rad}}\in\mathbb{K}(B_{R})$
and 
\begin{subequations}
\begin{align}
\int_{B_{R}}|u^{{\rm rad}}|^{2}\,dx & =\int_{B_{R}}|u|^{2}\,dx,\\
\int_{B_{R}}|\nabla u^{{\rm rad}}|^{2}\,dx & \le\int_{B_{R}}|\nabla u|^{2}\,dx,\label{eq:PolyaSzego}\\
\int_{B_{R}}fu^{{\rm rad}}\,dx & \ge\int_{B_{R}}fu\,dx,\label{eq:rearrangement1a}\\
\int_{B_{R}}|g|^{2}\chi_{\{u^{{\rm rad}}>0\}}\,dx & \le\int_{B_{R}}|g|^{2}\chi_{\{u>0\}}\,dx.\label{eq:rearrangement1b}
\end{align}
\end{subequations}
Here, \eqref{eq:PolyaSzego} is the classical P\'{o}lya-Szeg\H{o}
inequality \cite[Theorem~1.1]{BZ88PolyaSzego}, while \eqref{eq:rearrangement1a}
and \eqref{eq:rearrangement1b} follow by the fact that $f$ is non-increasing
and $g$ is non-decreasing as functions of $r=|x|$. It follows that
\begin{equation}
\mathcal{J}_{f,g,\lambda,B_{R}}(u^{{\rm rad}})\le\mathcal{J}_{f,g,\lambda,B_{R}}(u).\label{eq:restrictionJ}
\end{equation}

We define 
\[
\mathbb{K}^{{\rm rad}}(B_{R})=\begin{Bmatrix}\begin{array}{l|l}
u\in\mathbb{K}(B_{R}) & u=u^{{\rm rad}}\end{array}\end{Bmatrix}.
\]
Using \eqref{eq:restrictionJ}, there exists $u_{*}^{{\rm rad}}\in\mathbb{K}^{{\rm rad}}(B_{R})$
 such that 
\begin{equation}
\mathcal{J}_{f,g,\lambda,B_{R}}(u_{*}^{{\rm rad}})=\inf_{v\in\mathbb{K}(B_{R})}\mathcal{J}_{f,g,\lambda,B_{R}}(v).\label{eq:rad-minimizer}
\end{equation}

\bigskip

\noindent \textbf{Step 2: (Minimizers in $\mathbb{K}^{{\rm rad}}(B_{R})$)} Let $\tilde{u}\in\mathbb{K}^{{\rm rad}}(B_{R})$ be any function such
that 
\[
\mathcal{J}_{f,g,\lambda,B_{R}}(\tilde{u})=\inf_{v\in\mathbb{K}(B_{R})}\mathcal{J}_{f,g,\lambda,B_{R}}(v).
\]
From \eqref{eq:PDE1b} in Proposition~\ref{prop:PDE1}, we know that
$\tilde{u}$ satisfies the equation 
\[
(\Delta+\lambda)\tilde{u}+f=0\quad\text{in }\{\tilde{u}>0\}.
\]
\begin{subequations}
In polar coordinates, the above equation reads 
\begin{equation}
|\tilde{u}(0)|<\infty,\quad\tilde{u}''(r)+\frac{n-1}{r}\tilde{u}'(r)+\lambda\tilde{u}(r)+a\chi_{\{r<r_{1}\}}-b=0\quad\text{for }r\in(0,\rho), \label{eq:ODE-radially}
\end{equation}
with $\tilde{u}'(r) \le 0$ for all $r \in (0,\rho)$ and $\tilde{u}(r)=0$ for all $r \ge \rho$, where $\rho \in (0,R]$. In addition, one has
(see Proposition~{\rm \ref{prop:PDE2}}) 
\begin{equation}
\tilde{u}'(\rho) = - g(\rho) \le 0, \label{eq:ODE-radially-derivative}
\end{equation}
\end{subequations}
and $\tilde{u}$ is the unique solution of the ODE system \eqref{eq:ODE-radially}--\eqref{eq:ODE-radially-derivative}.

We now compute an explicit formula for $\tilde{u}$. Let $u$ be the unique solution of  
\[
\left\{
\begin{aligned}
& u''(r)+\frac{n-1}{r}u'(r)+\lambda u(r)+a\chi_{\{r<r_{1}\}}-b=0\quad\text{for }r\in(\rho,\infty) \\
& u(\rho)=\tilde{u}(\rho) ,\quad u'(\rho)=\tilde{u}'(\rho).
\end{aligned}
\right.
\]
By defining $u|_{(0,\rho)}=\tilde{u}$, one sees that $u \in C_{\rm loc}^{1}(\mathbb{R})$ and 
\begin{equation}
u''(r)+\frac{n-1}{r}u'(r)+\lambda u(r)+a\chi_{\{r<r_{1}\}}-b=0\quad\text{for }r\in(0,\infty).
\label{eq:ODE-radially2}
\end{equation}
By direct computations 
(see Appendix~\ref{appen:computations} for details), 
one sees that the general solution of \eqref{eq:ODE-radially2} is  
\begin{equation}
\begin{aligned}
u(r) & =\frac{b-a}{\lambda}+c_{1}r^{\frac{2-n}{2}}J_{\frac{n-2}{2}}(\sqrt{\lambda}r) \\
 & \quad+\chi_{\{r>r_{1}\}}\bigg[\frac{a}{\lambda}+\frac{a\pi r_{1}^{\frac{n}{2}}}{2\sqrt{\lambda}}r^{\frac{2-n}{2}}\left(Y_{\frac{n}{2}}(\sqrt{\lambda}r_{1})J_{\frac{n-2}{2}}(\sqrt{\lambda}r)-J_{\frac{n}{2}}(\sqrt{\lambda}r_{1})Y_{\frac{n-2}{2}}(\sqrt{\lambda}r)\right)\bigg]
\end{aligned}\label{eq:explicit-formula}
\end{equation}
with $c_{1} \in \mathbb{R}$. Since $u=\tilde{u}$ is positive and decreasing near 0, then $c_{1}>0$. 
By direct computations 
(see Appendix~\ref{appen:computations} for details), 
one sees that there exists a zero $\rho_{0} \in (0,R]$ of $u$ such that
\begin{equation}
\text{$u$ is positive and non-increasing on $(0,\rho_{0})$,} \label{eq:ODE-radially-BC}
\end{equation}
therefore $\rho_{0} = \rho$, where $\rho$ is the constant given in \eqref{eq:ODE-radially}. We now impose the boundary condition $u'(\rho)=-g(\rho)$. Using assumptions on $g$, direct computations 
(see Appendix~\ref{appen:computations} for details) 
yield 
\begin{equation}
\rho \in (r_{1},R'), \quad \text{where }R'\text{ is given in \eqref{eq:R-prime-choice}.} \label{eq:rho-special}
\end{equation}
From this we conclude that $\overline{B_{r_{1}}}\subset\{\tilde{u}>0\}$ as well as ${\rm supp}\,(\tilde{u})=\overline{B_{\rho}}\subset B_{R'}$.

\bigskip

\noindent \textbf{Step 3: (All minimizers belongs to $\mathbb{K}^{{\rm rad}}(B_{R})$)}
Let $u_{*}\in\mathbb{K}(B_{R})$ be a minimizer of $\mathcal{J}_{f,g,\lambda,B_{R}}$
in $\mathbb{K}(B_{R})$. Using \eqref{eq:restrictionJ}, then its
radially symmetric decreasing rearrangement $u_{*}^{{\rm rad}}\in\mathbb{K}^{{\rm rad}}(B_{R})$
satisfies \eqref{eq:rad-minimizer}, that is, $u_{*}^{{\rm rad}}$
is one of our radial solutions, and we have 
\[
\int_{B_{R}}|\nabla u_{*}^{{\rm rad}}|^{2}\,dx=\int_{B_{R}}|\nabla u_{*}|^{2}\,dx.
\]
Since the radial solutions are radially strictly decreasing on the
positivity set, we deduce that $u_{*}^{{\rm rad}}$ is strictly decreasing
on $(0,\rho)$ with 
\[
{\rm supp}\,(u_{*}^{{\rm rad}})=\overline{B_{\rho}}.
\]
Therefore, from \cite[Theorem~1.1]{BZ88PolyaSzego} we know that 
\[
u_{*}(x)=u_{*}^{{\rm rad}}(x-x_{0})\quad\text{for some }x_{0}.
\]
Now, by way of  contradiction, suppose that $x_{0}\neq0$. Since $u_{*}^{{\rm rad}}$ satisfies
\eqref{eq:rad-minimizer}, Proposition~\ref{prop:comparison} tells
us that so is $w=\max\{u_{*},u_{*}^{{\rm rad}}\}$, but $w$ is not
radially decreasing around some $x_{0}$, which contradicts the
minimality of $u_{*}$. 
\end{proof}
\begin{rem}
If $r_{1}=R$, from the general solution and the boundary condition
$\tilde{u}(\rho)=0$, we know that 

\[
\tilde{u}(r)=\frac{b-a}{k^{2}}\bigg(1-\frac{r^{\frac{2-n}{2}}J_{\frac{n-2}{2}}(\sqrt{\lambda}r)}{\rho^{\frac{2-n}{2}}J_{\frac{n-2}{2}}(\sqrt{\lambda}\rho)}\bigg)\chi_{\{r<\rho\}}\quad\text{for some }\rho\in[0,R].
\]
Since $\{\tilde{u}>0\}$ is a Lipschitz domain, using Remark~{\rm \ref{rem:Lipschitz-boundary-case}} we compute that 
\begin{align*}
\mathcal{J}_{f,g,\lambda,B_{R}}(\tilde{u}) & =\int_{B_{\rho}}g^{2}\,dx-(a-b)\int_{B_{\rho}}\tilde{u}\,dx\\
 & =\int_{B_{\rho}}g^{2}\,dx+\frac{(a-b)^{2}}{\lambda}\bigg(|B_{\rho}|-\frac{1}{\rho^{\frac{2-n}{2}}J_{\frac{n-2}{2}}(k\rho)}|B_{1}|\int_{0}^{\rho}J_{\frac{n-2}{2}}(kr)r^{\frac{n}{2}}\,dr\bigg)\\
 & =\int_{B_{\rho}}g^{2}\,dx+\frac{(a-b)^{2}}{\lambda}\bigg(|B_{\rho}|-\frac{1}{\rho^{\frac{2-n}{2}}J_{\frac{n-2}{2}}(k\rho)}|B_{1}|\frac{\rho^{\frac{n}{2}}J_{\frac{n}{2}}(k\rho)}{k}\bigg)\\
 & =\int_{B_{\rho}}g^{2}\,dx+\frac{(a-b)^{2}}{\lambda}|B_{\rho}|\bigg(1-\frac{J_{\frac{n}{2}}(k\rho)}{k\rho J_{\frac{n-2}{2}}(k\rho)}\bigg)\\
 & =\int_{B_{\rho}}g^{2}\,dx+\frac{(a-b)^{2}}{\lambda}|B_{\rho}|\bigg(1-\frac{\Gamma(\frac{n}{2})}{2\Gamma(1+\frac{n}{2})}\bigg)\\
 & =\int_{B_{\rho}}g^{2}\,dx+\frac{(a-b)^{2}}{\lambda}|B_{\rho}|(n-1)\frac{\Gamma(\frac{n}{2})}{2\Gamma(1+\frac{n}{2})}.
\end{align*}
Since $a>b$, from \eqref{eq:non-positive-functional} we conclude
that $\rho=0$, that is, $\tilde{u}\equiv0$ in $B_{R}$. Since all
minimizers belong to $\mathbb{K}^{{\rm rad}}(B_{R})$,
in this case each minimizer of $\mathcal{J}_{f,g,\lambda,B_{R}}$
in $\mathbb{K}(B_{R})$ must trivial. 
\end{rem}

Combining Lemma~\ref{lem:radially1} with the comparison principle
(Proposition~\ref{prop:comparison}), we have the following proposition. 
\begin{prop}
\label{prop:existence-minimizer1} Let $\Omega=B_{R}$ {\rm (}with $R>0${\rm )}
and $0<\lambda<\lambda^{*}(B_{R})\equiv j_{\frac{n-2}{2},1}^{2}R^{-2}$.
Suppose that $f=\mu-h$ with 
\begin{equation}
a_{0}\chi_{B_{r_{1}}}\le\mu(x)\le a\chi_{B_{r_{2}}},\quad b\le h(x)\le b_{0}\quad\text{for all $x\in B_{R}$} \label{eq:assumption-measure}
\end{equation}
for some constants $r_{1},r_{2},a,a_{0},b,b_{0}$ satisfying 
\begin{equation}
\begin{aligned}
& 0<b\le b_{0}<a_{0}\le a, \quad 0<r_{1}\le r_{2}<R \\
& {\displaystyle \frac{r_{1}^{\frac{n}{2}}J_{\frac{n}{2}}(\sqrt{\lambda}r_{1})}{r_{2}^{\frac{n}{2}}J_{\frac{n}{2}}(\sqrt{\lambda}r_{2})}}>\frac{b_{0}}{a_{0}}\ge\frac{b}{a}>\frac{r_{2}^{\frac{n}{2}}J_{\frac{n}{2}}(\sqrt{\lambda}r_{2})}{R^{\frac{n}{2}}J_{\frac{n}{2}}(\sqrt{\lambda}R)}. \end{aligned}\label{eq:assumption-parameters}
\end{equation}
We also assume that $g\in L^{\infty}(\mathbb{R}^{n})$ with 
\[
g=0\quad\text{in }{\rm supp}\,(\mu)\equiv\overline{B_{r_{1}}}.
\]
There exists $u_{*}$ such that 
\[
\mathcal{J}_{f,g,\lambda,B_{R}}(u_{*})=\inf_{u\in\mathbb{K}(B_{R})}\mathcal{J}_{f,g,\lambda,B_{R}}(u_{*}).
\]
Moreover, each minimizer $u_{*}$ of $\mathcal{J}_{f,g,\lambda,B_{R}}$ in
$\mathbb{K}(B_{R})$ satisfies 
\[
{\rm supp}\,(\mu)\subset B_{R_{0}'}\subset\{u_{*}>0\}\quad\text{and}\quad{\rm supp}\,(u_{*})\subset B_{R}.
\]
for some $R_{0}'>0$. 
\end{prop}

\begin{proof}
Since $0<\lambda<j_{\frac{n-2}{2},1}^{2}R^{-2}$, then $r\mapsto r^{\frac{n}{2}}J_{\frac{n}{2}}(\sqrt{\lambda}r)$
is monotone increasing on $(0,R)$. Then we have 
\[
\frac{b_{0}}{a_{0}}>\frac{b}{a}>\frac{r_{2}^{\frac{n}{2}}J_{\frac{n}{2}}(\sqrt{\lambda}r_{2})}{R^{\frac{n}{2}}J_{\frac{n}{2}}(\sqrt{\lambda}R)}\ge\frac{r_{1}^{\frac{n}{2}}J_{\frac{n}{2}}(\sqrt{\lambda}r_{1})}{R^{\frac{n}{2}}J_{\frac{n}{2}}(\sqrt{\lambda}R)}.
\]
By \eqref{eq:assumption-measure} and \eqref{eq:assumption-parameters} we know that $\mu-h=f\le\tilde{f}=a\chi_{B_{r_{2}}}-b$.
Let $u$ and $\tilde{u}$ be the respective  minimizers of $\mathcal{J}_{f,g,\lambda,B_{R}}$
and $\mathcal{J}_{\tilde{f},0,\lambda,B_{R}}$ in $\mathbb{K}(B_{R})$.
Using Proposition~\ref{prop:comparison}, we know that
$\max\{u,\tilde{u}\}$ minimizes $\mathcal{J}_{\tilde{f},g,\lambda,B_{R}}$.
By Lemma~\ref{lem:radially1}, we know that 
\[
{\rm supp}\,(u)\subset{\rm supp}\,(\max\{u,\tilde{u}\})\subset\overline{B_{R'}}\subset B_{R},
\]
with 
\[
R'=\max\begin{Bmatrix}\begin{array}{l|l}
\rho\in(0,R] & {\displaystyle \frac{b}{a}-\frac{r_{1}^{\frac{n}{2}}J_{\frac{n}{2}}(\sqrt{\lambda}r_{2})}{\rho^{\frac{n}{2}}J_{\frac{n}{2}}(\sqrt{\lambda}\rho)}}\le0\end{array}\end{Bmatrix}>r_{2}.
\]
On the other hand, by \eqref{eq:assumption-measure} and \eqref{eq:assumption-parameters} we know that $\mu-h=f\ge\tilde{f}_{0}=a_{0}\chi_{B_{r_{1}}}-b_{0}$.
Let $u_{0}$ and $\tilde{u}_{0}$ be minimizers of $\mathcal{J}_{f,g,\lambda,B_{R}}$
and $\mathcal{J}_{\tilde{f}_{0},\tilde{g}_{0},\lambda,B_{R}}$ in
$\mathbb{K}(B_{R})$, respectively, where 
\[
\tilde{g}_{0}=\|g\|_{L^{\infty}(\mathbb{R}^{n})}\chi_{\mathbb{R}^{n}\setminus\overline{B_{r_{1}}}}.
\]
Using Proposition~\ref{prop:comparison},
we know that $\max\{u_{0},\tilde{u}_{0}\}$ minimizes $\mathcal{J}_{f,g,\lambda,B_{R}}$
in $\mathbb{K}(B_{R})$. By choosing $u_{0}$ to  be the largest (pointwise)
minimizer of $\mathcal{J}_{f,g,\lambda,B_{R}}$ in $\mathbb{K}(B_{R})$, we have 
\[
u_{0}\ge\max\{u_{0},\tilde{u}_{0}\}\quad\text{in }B_{R},
\]
which implies $u_{0}\ge\tilde{u}_{0}$ in $B_{R}$. By Lemma~\ref{lem:radially1},
we know that $\tilde{u}_{0}>0$ in $B_{R_{0}'}$ with 
\[
R_{0}'=\max\begin{Bmatrix}\begin{array}{l|l}
\rho\in(0,R] & {\displaystyle \frac{b_{0}}{a_{0}}-\frac{r_{1}^{\frac{n}{2}}J_{\frac{n}{2}}(\sqrt{\lambda}r_{1})}{\rho^{\frac{n}{2}}J_{\frac{n}{2}}(\sqrt{\lambda}\rho)}}\le0\end{array}\end{Bmatrix}>r_{1}.
\]
Since we have 
\[
{\displaystyle \frac{b_{0}}{a_{0}}-\frac{r_{1}^{\frac{n}{2}}J_{\frac{n}{2}}(\sqrt{\lambda}r_{1})}{r_{2}^{\frac{n}{2}}J_{\frac{n}{2}}(\sqrt{\lambda}r_{2})}}<0,
\]
then we have $R_{0}'>r_{2}$, which implies that 
\[
{\rm supp}\,(\mu)\equiv\overline{B_{r_{2}}}\subset B_{R_{0}'}\subset\{u_{*}>0\}.
\]
This  completes the proof of the proposition. 
\end{proof}

Combining Proposition~{\rm \ref{prop:PDE3}}, Proposition~{\rm \ref{prop:regularity-free-boundary}}, Lemma~{\rm \ref{lem:PDE4}} and Proposition~{\rm \ref{prop:existence-minimizer1}} with $\lambda=k^{2}$ and $f=\mu-h\chi_{D}$, we arrive at  the following theorem (with $D=\{u_{*}>0\}$):

\begin{prop}
\label{prop:main-almost-there} Let $0<k<j_{\frac{n-2}{2},1}R^{-1}$,
\[
a_{0}\chi_{B_{r_{1}}}\le\mu(x)\le a\chi_{B_{r_{2}}},\quad b\le h(x)\le b_{0}\quad\text{for all }x\in B_{R}
\]
for some constants $r_{1},r_{2},a,a_{0},b,b_{0}$ with $0<b\le b_{0}<a_{0}\le a$ and $0<r_{1}\le r_{2}<R$ satisfying\footnote{Since $t\mapsto t^{\frac{n}{2}}J_{\frac{n}{2}}(t)$ is strictly increasing on $[0,j_{\frac{n-2}{2},1}]$, then the second condition of \eqref{eq:sufficient2} implies $b_{0}<a_{0}$.}
\begin{equation}
{\displaystyle \frac{r_{1}^{\frac{n}{2}}J_{\frac{n}{2}}(kr_{1})}{r_{2}^{\frac{n}{2}}J_{\frac{n}{2}}(kr_{2})}}>\frac{b_{0}}{a_{0}}\ge\frac{b}{a}>\frac{r_{2}^{\frac{n}{2}}J_{\frac{n}{2}}(kr_{2})}{R^{\frac{n}{2}}J_{\frac{n}{2}}(kR)},\label{eq:sufficient2}
\end{equation}
such that $g\in L^{\infty}(B_{R})$ with $g\ge0$ and $g^{2}\in W^{1,1}(B_{R})$. If $\partial\{g>0\}\cap B_{R}\neq\emptyset$, we further assume that there exists $0<\alpha\le1$ such that $g$ is $C^{\alpha}$ near $\partial\{g>0\}\cap B_{R}$ and $\mathscr{H}^{n-1+\alpha}(\partial\{g>0\}\cap B_{R})=0$. Then there exists a bounded open domain $D$ in $\mathbb{R}^{n}$ with the boundary $\partial D$ having finite $(n-1)$-dimensional Hausdorff measure such that 
\[
\text{$\left( \Psi_{k} * (g \mathscr{H}^{n-1} \lfloor \partial D) \right)(x)$ is pointwise well-defined for all $x \in \partial D$}
\]
for all fundamental solutions $\Psi_{k}$ of the Helmholtz operator $-(\Delta + k^{2})$. The set $D$ is a hybrid $k$-quadrature domain $D$, corresponding to distribution $\mu$ and density $(g,h)$, with $\overline{D}\subset B_{R}$. 
Moreover, there exists $u_{*} \in C_{\rm loc}^{0,1}(B_{\beta k^{-1}})$ such that $D = \{ u_{*} > 0 \}$ and 
\[
(\Delta + k^{2})u_{*} = - \mu + h \mathscr{L}^{n}\lfloor D + g \mathscr{H}^{n-1}\lfloor \partial D. 
\]
If we additionally assume that $g>0$ is H\"{o}lder continuous in $\overline{B_{R}}$, then $\partial_{\rm red} D$ is locally $C^{1,\alpha'}$ with $\mathscr{H}^{n-1}(\partial D \setminus \partial_{\rm red} D) = 0$. 
In the case when $n=2$ we even have $\partial D = \partial_{\rm red} D$.  
\end{prop}

Finally, we want to generalize Proposition~\ref{prop:main-almost-there}
for unbounded non-negative measures $\mu$. Assume that $\mu$ satisfies  
\begin{equation}
\mu=0\text{ outside }B_{\epsilon}\label{eq:concentrate}
\end{equation}
for some parameter $\epsilon>0$. We define 
\[
\phi_{2\epsilon}:=(c_{n,k,2\epsilon}^{{\rm MVT}})^{-1}\chi_{B_{2\epsilon}}\quad\text{with}\quad c_{n,k,2\epsilon}^{{\rm MVT}}:=(2\pi)^{\frac{n}{2}}k^{-\frac{n}{2}}(2\epsilon)^{\frac{n}{2}}J_{\frac{n}{2}}(2k\epsilon).
\]
It is easy to see that 
\[
\mu*\phi_{2\epsilon}\text{ is supported in }B_{3\epsilon}\text{ and }\mu*\phi_{2\epsilon}(x)=(c_{n,k,2\epsilon}^{{\rm MVT}})^{-1}\mu(B_{2\epsilon}(x))\text{ for all }x\in B_{3\epsilon}.
\]
Thus we see that 
\[
\mu*\phi_{2\epsilon}(x)\begin{cases}
\le(c_{n,k,2\epsilon}^{{\rm MVT}})^{-1}\mu(\mathbb{R}^{n}) & \text{for all }x\in B_{3\epsilon},\\
=(c_{n,k,2\epsilon}^{{\rm MVT}})^{-1}\mu(\mathbb{R}^{n}) & \text{for all }x\in B_{\epsilon},
\end{cases}
\]
and thus 
\[
(c_{n,k,2\epsilon}^{{\rm MVT}})^{-1}\mu(\mathbb{R}^{n})\chi_{B_{\epsilon}}\le\mu*\phi_{2\epsilon}\le(c_{n,k,2\epsilon}^{{\rm MVT}})^{-1}\mu(\mathbb{R}^{n})\chi_{B_{3\epsilon}}.
\]
We choose $r_{1}=\epsilon$ and $r_{2}=3\epsilon$, as well as 
\[
a_{0}=a=(c_{n,k,2\epsilon}^{{\rm MVT}})^{-1}\mu(\mathbb{R}^{n}).
\]
Then \eqref{eq:sufficient2} is equivalent to 
\begin{subequations}
\begin{align}
\frac{(\epsilon)^{\frac{n}{2}}J_{\frac{n}{2}}(k\epsilon)}{(3\epsilon)^{\frac{n}{2}}J_{\frac{n}{2}}(3k\epsilon)} & >\frac{b_{0}}{(c_{n,k,2\epsilon}^{{\rm MVT}})^{-1}\mu(\mathbb{R}^{n})}.\label{eq:sufficient3a}\\
\frac{b}{(c_{n,k,2\epsilon}^{{\rm MVT}})^{-1}\mu(\mathbb{R}^{n})} & >\frac{(3\epsilon)^{\frac{n}{2}}J_{\frac{n}{2}}(3k\epsilon)}{R^{\frac{n}{2}}J_{\frac{n}{2}}(kR)},\label{eq:sufficient3b}
\end{align}
\end{subequations}
We can write \eqref{eq:sufficient3a} as 
\[
\mu(\mathbb{R}^{n})>3^{\frac{n}{2}}\frac{J_{\frac{n}{2}}(3k\epsilon)}{J_{\frac{n}{2}}(k\epsilon)}\frac{b_{0}}{(c_{n,k,\delta}^{{\rm MVT}})^{-1}(2\epsilon)^{n}}(2\epsilon)^{n}.
\]
Using \cite[(8.19)]{KLSS22QuadratureDomain}, we know that 
\[
\frac{\pi^{\frac{n}{2}}}{\Gamma(1+\frac{n}{2})}\ge\frac{1}{(c_{n,k,\delta}^{{\rm MVT}})^{-1}(2\epsilon)^{n}}.
\]
Hence  \eqref{eq:sufficient3a} is fulfilled provided
\begin{equation}
\mu(\mathbb{R}^{n})>C_{n}b_{0}\epsilon^{n}\quad\text{with}\quad C_{n}\ge2^{n}\frac{(3\pi)^{\frac{n}{2}}}{\Gamma(1+\frac{n}{2})}\frac{J_{\frac{n}{2}}(j_{\frac{n-2}{2},1})}{J_{\frac{n}{2}}(j_{\frac{n-2}{2},1}/3)}.\label{eq:large-mass-mu}
\end{equation}
Using the definition of $c_{n,k,2\epsilon}^{{\rm MVT}}$, we now write \eqref{eq:sufficient3b} as 
\[
k^{n}<\bigg(\frac{4\pi}{3}\bigg)^{\frac{n}{2}}\frac{1}{\mu(\mathbb{R}^{n})}b(kR)^{\frac{n}{2}}J_{\frac{n}{2}}(kR)\frac{J_{\frac{n}{2}}(2k\epsilon)}{J_{\frac{n}{2}}(3k\epsilon)}.
\]
We now fix any parameter $0<\beta<j_{\frac{n-2}{2},1}$, and we choose $R=\beta k^{-1}$. If 
\begin{equation}
k<\min\bigg\{\frac{1}{3},\bigg(C_{n,\beta}\frac{b}{\mu(\mathbb{R}^{n})}\bigg)^{\frac{1}{n}}\bigg\}\quad\text{with}\quad C_{n,\beta}=\bigg(\frac{4\pi}{3}\bigg)^{\frac{n}{2}}\beta^{\frac{n}{2}}J_{\frac{n}{2}}(\beta)\frac{J_{\frac{n}{2}}(\frac{2}{3}j_{\frac{n-2}{2},1})}{J_{\frac{n}{2}}(j_{\frac{n-2}{2},1})},\label{eq:small-frequency-assumption}
\end{equation}
then \eqref{eq:sufficient3b} holds. 
The above discussion is 
valid for $0<\epsilon<\beta$. 

Using Proposition~\ref{prop:main-almost-there} on $\mu*\phi_{2\epsilon}$,
we then know that there exists a hybrid $k$-quadrature domain $D$,
corresponding to the distribution $\mu*\phi_{2\epsilon}$ and the  density
$(g,h)$, with $\overline{D}\subset B_{R}$. Using the mean value
theorem for Helmholtz equation \cite[Appendix~A]{KLSS22QuadratureDomain},
we have 
\[
\langle\mu*\phi_{2\epsilon},w\rangle=\langle\mu,w*\phi_{2\epsilon}\rangle=\langle\mu,w\rangle
\]
for all $w$ satisfying $(\Delta+k^{2})w=0$ in
$D$. Hence  such a $D$ is indeed also a hybrid $k$-quadrature domain $D$, corresponding to distribution $\mu$ and density $(g,h)$.
We now conclude the above discussions in the following theorem: 
\begin{thm}
[See {\rm Theorem~\ref{thm:main-quadrature}}] \label{thm:main-quadrature-detail}
Fix parameters $0<b\le b_{0}$ and
$0<\beta<j_{\frac{n-2}{2},1}$. Let  $0<\epsilon<\beta$,  
\[
b\le h(x)\le b_{0}\quad\text{for all }x\in B_{\beta k^{-1}},
\]
and $g\in L^{\infty}(B_{\beta k^{-1}})$ with $g\ge0$ and $g^{2}\in W^{1,1}(B_{\beta k^{-1}})$.
If $\partial\{g>0\}\cap B_{R}\neq\emptyset$, we further assume
that there exists $0<\alpha\le1$ such that $g$ is $C^{\alpha}$
near $\partial\{g>0\}\cap B_{R}$ and $\mathscr{H}^{n-1+\alpha}(\partial\{g>0\}\cap B_{R})=0$.
If $\mu$ is a non-negative measure satisfying \eqref{eq:concentrate} and \eqref{eq:large-mass-mu},
then for each $k$ that satisfies \eqref{eq:small-frequency-assumption}
there exists a bounded open domain $D$ in $\mathbb{R}^{n}$ with the boundary $\partial D$ having finite $(n-1)$-dimensional Hausdorff measure such that 
\[
\text{$\left( \Psi_{k} * (g \mathscr{H}^{n-1} \lfloor \partial D) \right)(x)$ is pointwise well-defined for all $x \in \partial D$}
\]
for all fundamental solutions $\Psi_{k}$ of the Helmholtz operator $-(\Delta + k^{2})$. This domain $D$ is a hybrid $k$-quadrature domain corresponding
to distribution $\mu$ and density $(g,h)$ and it satisfies $\overline{D}\subset B_{\beta k^{-1}}$. 
Moreover, there exists a nonnegative function $u_{*} \in C_{\rm loc}^{0,1}(B_{\beta k^{-1}})$ such that $D = \{ u_{*} > 0 \}$ and 
\[
(\Delta + k^{2})u_{*} = - \tilde{\mu} + h \mathscr{L}^{n}\lfloor D + g \mathscr{H}^{n-1}\lfloor \partial D
\]
for some non-negative $\tilde{\mu} \in L^{\infty}(D)\cap \mathscr{E}'(D)$. 
If we additionally assume that $g>0$ is H\"{o}lder continuous in $\overline{B_{R}}$, then $\partial_{\rm red} D$ is locally $C^{1,\alpha'}$ with $\mathscr{H}^{n-1}(\partial D \setminus \partial_{\rm red} D) = 0$. 
In the case when $n=2$ we even have $\partial D = \partial_{\rm red} D$. 
\end{thm}


\appendix

\section{\label{sec:BV-finite-perimeter}Functions of bounded variation and sets with finite perimeter}

We recall a few facts about functions of bounded variation and sets with finite perimeter. Here we refer to the monographs \cite{EG15MeasureTheory,Giu84MinimalSurfaces} for more details. The following definition can be found in \cite[Definition~1.6]{Giu84MinimalSurfaces}: 
\begin{defn}
Let $E$ be a Borel set and $\Omega$ an open set in $\mathbb{R}^{n}$.
We define the \emph{perimeter} of $E$ in $E_{0}$ as 
\[
\mathcal{P}(E,E_{0}):=\int_{E_{0}}|\nabla\chi_{E}|\,dx\equiv\sup_{\phi\in(C_{c}^{1}(\Omega))^{n},|\phi(x)|\le1}\int_{E_{0}}\nabla\cdot\phi\,dx.
\]
We say that $E$ is a \emph{Caccioppoli set}, if $E$ has \emph{locally
finite perimeter}, i.e. 
\[
\mathcal{P}(E,K)<\infty\quad\text{for every compact set }K\text{ in }\mathbb{R}^{n}.
\]
In other words, the function $\chi_{E}$ has locally bounded variation
in $\mathbb{R}^{n}$, see \cite[Section~5.1]{EG15MeasureTheory}. 
\end{defn}

The following definition can be found in \cite[Definition~3.3]{Giu84MinimalSurfaces}
(this concept was introduced by De Giorgi \cite{DG55ReducedBoundary},
see also \cite[Section~5.7]{EG15MeasureTheory}): 
\begin{defn}
Assuming that $E$ is a Caccioppoli set, we define the \emph{reduced
boundary} $\partial_{{\rm red}}E$ of $E$ by the set of points $x\in\mathbb{R}^{n}$
for which the followings hold: \labeltext{$\partial_{\rm red} E$ reduced boundary}{index:red-boundary}
\begin{enumerate}
\renewcommand{\labelenumi}{\theenumi}
\renewcommand{\theenumi}{{\rm (\arabic{enumi})}}
\item $\int_{B_{r}(x)}|\nabla\chi_{E}|\,dx>0$ for all $r>0$, 
\item the limit $\nu_{E}(x):=\lim_{r\rightarrow0}\nu_{E}^{r}(x)$ exists,
where 
\[
\nu_{E}^{r}(x):=-\frac{\int_{B_{r}(x)}\nabla\chi_{E}\,dx}{\int_{B_{r}(x)}|\nabla\chi_{E}|\,dx},
\]
and $|\nu_{E}(x)|=1$. 
\end{enumerate}
\end{defn}

From the Besicovitch differentiation of measures, it follows that
$\nu_{E}(x)$ exists and $|\nu_{E}(x)|=1$ for $|\nabla\chi_{E}|$-almost
all $x\in\mathbb{R}^{n}$, and furthermore that 
\[
\nabla\chi_{E}=-\nu_{E}|\nabla\chi_{E}|.
\]
Using \cite[Theorem~4.4]{Giu84MinimalSurfaces}, we indeed know that
\[
|\nabla\chi_{E}|=\mathscr{H}^{n-1}\lfloor\partial_{{\rm red}}E\quad\text{and}\quad\partial_{{\rm red}}E\text{ is a dense subset of }\partial E.
\]
Thus, we have 
\begin{equation}
\mathcal{P}(E,E_{0})=\mathscr{H}^{n-1}(E_{0}\cap\partial_{{\rm red}}E)\quad\text{and}\quad\nabla\chi_{E}=-\nu_{E}\mathscr{H}^{n-1}\lfloor\partial_{{\rm red}}E,\label{eq:characterization-Caccioppoli}
\end{equation}
and then we immediately have the following generalized Gauss-Green
theorem: 
\begin{equation}
\int_{E}\nabla\cdot\phi\,dx=\int_{\partial_{{\rm mes}}E}\varphi\cdot\nu_{E}\,d\mathscr{H}^{n-1}\quad\text{for all }\phi\in(C_{c}^{1}(\mathbb{R}^{n}))^{n},\label{eq:Gauss-Green}
\end{equation}
see also \cite[Theorem~1 in Section~5.8]{EG15MeasureTheory}. 
\begin{rem}
If $\partial E$ is a $C^{1}$ hypersurface, then $\partial_{{\rm red}}E=\partial E$
and $\nu_{E}(x)$ is the unit outward normal vector to $\partial E$
at $x$, however, if $\partial E$ is Lipschitz, $\partial_{{\rm red}}E$
in general strictly contained in $\partial E$, see \cite[Remark~3.4]{Giu84MinimalSurfaces}
details. Therefore, we also refer $\nu_{E}$ the \emph{measure theoretic
outward unit normal vector} of $E$ on $\partial_{{\rm red}}E$. 
\end{rem}

From \cite[Lemma~1 in Section~5.8]{EG15MeasureTheory}, we also know that
\begin{equation}
\partial_{{\rm red}}E\subset\partial_{{\rm mes}}E\quad\text{and}\quad\mathscr{H}^{n-1}(\partial_{{\rm mes}}E\setminus\partial_{{\rm red}}E)=0.\label{eq:reduced-boundary-measure}
\end{equation}
Combining \eqref{eq:characterization-Caccioppoli} and \eqref{eq:reduced-boundary-measure},
we then know that 
\[
E\text{ is a Caccioppoli set if and only if }\mathscr{H}^{n-1}(\partial_{{\rm mes}}E\cap K)<\infty\text{ for each compact set }K\text{ in }\mathbb{R}^{n},
\]
see also \cite[Theorem~1 in Section~5.11]{EG15MeasureTheory}. We also
recall \cite[Theorem~2 in Section~2.3]{EG15MeasureTheory} regarding Hausdorff
measure: 
\begin{lem}
\label{lem:Hausdorff-density} Let $0<s<n$. If $\mathscr{H}^{s}(E)<\infty$,
then 
\[
\frac{1}{2^{s}}\le\limsup_{r\rightarrow0}\frac{\mathscr{H}^{s}(B_{r}(x)\cap E)}{\omega_{s}r^{s}}\le1\quad\text{for }\mathscr{H}^{s}\text{-a.e. }x\in E,
\]
where $\omega_{s}=\frac{\pi^{s/2}}{\Gamma(\frac{s}{2}+1)}$. 
\end{lem}

\section{\label{appen:details}Further properties of local minimizers}

In this appendix we provide detailed statements and proofs results analogue to \cite[Section~2]{GS96FreeBoundaryPotential}. 

The following lemma concerning the growth rate of the integral-mean of minimizers. 
\begin{lem}
\label{lem:sufficient-continuity} Let $\Omega$ be an open set in
$\mathbb{R}^{n}$ and let $\lambda \ge 0$. Let $f,g\in L^{\infty}(\Omega)$ be such that $g\ge0$.
If $u_{*}$ is a local minimizer of $\mathcal{J}_{f,g,\lambda,\Omega}$
in $\mathbb{K}(\Omega)$, then there is an $r_{0}>0$ such that for
any $B_{r}(x_{0})$ with $0<r<r_{0}$ and $\overline{B_{r}(x_{0})}\subset\Omega$,
we have 
\begin{equation}
\begin{aligned}
 & \frac{1}{r}\dashint_{\partial B_{r}(x_{0})}u_{*}\,dS>2^{n}\bigg(\frac{r}{n}\|(f+\lambda u_{*})_{-}\|_{L^{\infty}(B_{r}(x_{0}))}+\|g\|_{L^{\infty}(B_{r}(x_{0}))}\bigg)\\
 & \quad \implies B_{r}(x_{0})\subset\{u_{*}>0\}\text{ and }u_{*}\text{ is continuous in }B_{r}(x_{0}), 
\end{aligned}\label{eq:continuity-sufficient}
\end{equation}
where $\dashint_{\partial B_{r}(x_{0})}=\frac{1}{\mathscr{H}^{n-1}(\partial B_{r})}\int_{\partial B_{r}(x_{0})}$
denotes the average integral. 
\end{lem}

\begin{proof}
Let $u_{*}$ be a local minimizer $\mathcal{J}_{f,g,\lambda,\Omega}$
in $\mathbb{K}(\Omega)$. Using Lemma~\ref{lem:zero-freq}, we know
that $u_{*}$ is also a local minimizer of $\mathcal{J}_{\tilde{f},g,0,\Omega}$
in $\mathbb{K}(\Omega)$ with $\tilde{f}=f+\lambda u_{*}$. Without
loss of generality, we may assume that $x_{0}=0$, and we write $B_{r}=B_{r}(x_{0})$. Define $v\in H^{1}(\Omega)$ by
\begin{equation}
\begin{cases}
\Delta v=-\tilde{f} & \text{in }B_{r},\\
v=u_{*} & \text{in }\Omega\setminus\overline{B_{r}}.
\end{cases}\label{eq:harmonic-lifting-v}
\end{equation}
It is easy to see that $v\in C(B_{r})$. In particular, using elliptic
regularity and Sobolev embedding, we know that 
\begin{equation}
v\in\bigcap_{1<p<\infty}W_{{\rm loc}}^{2,p}(B_{r})\subset\bigcap_{0<\alpha<1}C_{{\rm loc}}^{1,\alpha}(B_{r}).\label{eq:harmonic-liftin-regularity}
\end{equation}
It is easy to compute that 
\begin{equation}
\begin{aligned}
 & \mathcal{J}_{\tilde{f},g,0,\Omega}(u_{*})-\mathcal{J}_{\tilde{f},g,0,\Omega}(v) \\
 & \quad =\int_{\Omega}(|\nabla u_{*}|^{2}-|\nabla v|^{2})\,dx-2\int_{\Omega}\tilde{f}(u_{*}-v)\,dx+\int_{\Omega}g^{2}(\chi_{\{u_{*}>0\}}-\chi_{\{v>0\}})\,dx \\
 & \quad =\int_{\Omega}(|\nabla u_{*}|^{2}-|\nabla v|^{2})\,dx+2\int_{\Omega}\Delta v(u_{*}-v)\,dx-\int_{B_{r}}g^{2}(\chi_{\{v>0\}\cap\{u_{*}=0\}})\,dx \\
 & \quad =\int_{\Omega}(|\nabla u_{*}|^{2}-|\nabla v|^{2})\,dx-2\int_{\Omega}\nabla v\cdot\nabla u_{*}\,dx +2\int_{\Omega}|\nabla v|^{2}\,dx\\
 & \qquad -\int_{B_{r}}g^{2}(\chi_{\{v>0\}\cap\{u_{*}=0\}})\,dx \\
 & \quad =\int_{B_{r}}|\nabla(u_{*}-v)|^{2}\,dx-\int_{B_{r}}g^{2}(\chi_{\{v>0\}\cap\{u_{*}=0\}})\,dx \\
 & \quad \ge\int_{B_{r}}|\nabla(u_{*}-v)|^{2}\,dx-|\{u=0\}\cap B_{r}|\sup_{B_{r}}g^{2}.
\end{aligned}\label{eq:difference-energy}
\end{equation}
Next, we want to show that $v\in\mathbb{K}(\Omega)$. Applying \cite[Lemma~2.4(a)]{GS96FreeBoundaryPotential} to $v$, we show that 
\begin{equation}
v(x)\ge r^{n}\frac{r-|x|}{(r+|x|)^{n-1}}\bigg(\frac{1}{r^{2}}\dashint_{\partial B_{r}}u_{*}\,dS-\frac{2^{n-1}M}{n}\bigg)\quad\text{for all }x\in B_{r},\label{eq:Poisson-v}
\end{equation}
where $M=\|\tilde{f}_{-}\|_{L^{\infty}(B_{r})}$. The assumption in
\eqref{eq:continuity-sufficient} implies 
\[
\frac{1}{r^{2}}\dashint_{\partial B_{r}}u_{*}\,dS\ge\frac{2^{n}M}{n},
\]
then we have
\begin{equation}
v(x)\ge\frac{r^{n}}{2}\frac{r-|x|}{(r+|x|)^{n-1}}\frac{1}{r^{2}}\dashint_{\partial B_{r}}u_{*}\,dS\ge2^{-n}(r-|x|)\frac{1}{r}\dashint_{\partial B_{r}}u_{*}\,dS\quad\text{for all }x\in B_{r},\label{eq:Poisson0}
\end{equation}
which shows that $v\in\mathbb{K}(\Omega)$. Since $u_{*}$ is a local
minimizer, by choosing $r_{0}>0$ sufficiently small, we know that
\[
\mathcal{J}_{\tilde{f},g,0,\Omega}(u_{*})\le\mathcal{J}_{\tilde{f},g,0,\Omega}(v),
\]
and hence from \eqref{eq:difference-energy} we know that 
\begin{equation}
\int_{B_{r}}|\nabla(u_{*}-v)|^{2}\,dx\le|\{u_{*}=0\}\cap B_{r}|\sup_{B_{r}}g^{2}\quad\text{for each }0<r<r_{0}.\label{eq:Poisson0-1}
\end{equation}
To estimate the left-hand-side of \eqref{eq:Poisson0-1}, from \eqref{eq:Poisson0} we have 
\begin{equation}
\chi_{\{u_{*}=0\}}\bigg(\frac{1}{r}\dashint_{\partial B_{r}}u_{*}\,dS\bigg)^{2}(r-|x|)^{2}\le2^{2n}\chi_{\{u_{*}=0\}}|v(x)|^{2}\quad\text{for all }x\in B_{r}.\label{eq:Poisson1}
\end{equation}
For each $0\neq x\in B_{r}$, writing $\hat{x}=x/|x|$, note that
\begin{equation*}
\begin{aligned}
\chi_{\{u_{*}=0\}}|v(x)|^{2} & =\chi_{\{u_{*}=0\}}\bigg(\int_{|x|}^{r}\partial_{|z|}(u_{*}-v)(s\hat{x})\,ds\bigg)^{2}\\
 & \le\chi_{\{u_{*}=0\}}\bigg(\int_{|x|}^{r}1^{2}\,ds\bigg)\bigg(\int_{|x|}^{r}|\partial_{|z|}(u_{*}-v)(s\hat{x})|^{2}\,ds\bigg)\\
 & \le\chi_{\{u_{*}=0\}}(r-|x|)\bigg(\int_{|x|}^{r}|\nabla(u_{*}-v)(s\hat{x})|^{2}\,ds\bigg),
\end{aligned}
\end{equation*}
and hence from \eqref{eq:Poisson1} we have 
\begin{equation}
\begin{aligned}
& (r-|x|)\chi_{\{u_{*}=0\}}\bigg(\frac{1}{r}\dashint_{\partial B_{r}}u_{*}\,dS\bigg)^{2} \\
& \quad \le2^{2n}\chi_{\{u_{*}=0\}}\bigg(\int_{|x|}^{r}|\nabla(u_{*}-v)(s\hat{x})|^{2}\,ds\bigg)\quad\text{for all }0\neq x\in B_{r}.
\end{aligned} \label{eq:Poisson2}
\end{equation}
We consider such $\theta\in\mathcal{S}^{n-1}$ be such that $u_{*}(s\theta)=0$ for some $0<s<r$. For such $\theta$, we can define 
\begin{equation*}
s_{\theta}:=\inf\begin{Bmatrix}\begin{array}{l|l}
0<s<r & u_{*}(s\theta)=0\end{array}\end{Bmatrix}.
\end{equation*}
Then \eqref{eq:Poisson2} implies 
\begin{equation}
\begin{aligned}
& (r-s_{\theta})\chi_{\{u_{*}=0\}\cap B_{r}}(s_{\theta}\theta)\bigg(\frac{1}{r}\dashint_{\partial B_{r}}u_{*}\,dS\bigg)^{2} \\
& \quad \le2^{2n}\chi_{\{u_{*}=0\}\cap B_{r}}(s_{\theta}\theta)\int_{s_{\theta}}^{r}|\nabla(u_{*}-v)(s\theta)|^{2}\,ds\quad\text{for all }\theta\in\mathcal{S}^{n-1}.
\end{aligned} \label{eq:Poisson3}
\end{equation}
Integrating \eqref{eq:Poisson3} over $\theta\in\mathcal{S}^{n-1}$,
we reach 
\begin{equation}
|\{u_{*}=0\}\cap B_{r}|\bigg(\frac{1}{r}\dashint_{\partial B_{r}}u_{*}\,dS\bigg)^{2}\le2^{2n}\int_{B_{r}}|\nabla(u_{*}-v)|^{2}\,dx.\label{eq:Poisson4}
\end{equation}
Combining \eqref{eq:Poisson0-1} and
\eqref{eq:Poisson4}, we reach 
\[
|\{u_{*}=0\}\cap B_{r}|\bigg(\frac{1}{r}\dashint_{\partial B_{r}}u_{*}\,dS\bigg)^{2}\le2^{2n}|\{u=0\}\cap B_{r}|\sup_{B_{r}}g^{2}.
\]
The assumption in \eqref{eq:continuity-sufficient} implies
\[
\frac{1}{r}\dashint_{\partial B_{r}(x_{0})}u_{*}\,dS>2^{n}\sup_{B_{r}(x_{0})}g,
\]
then it is necessarily 
\begin{equation}
|\{u_{*}=0\}\cap B_{r}|=0.\label{eq:almost-disjoint}
\end{equation}
From \eqref{eq:Poisson0-1} we know that 
\[
\int_{B_{r}}|\nabla(u_{*}-v)|^{2}\,dx=0,
\]
and thus we also showed that $u_{*}=v\in C(B_{r})$. Using \eqref{eq:harmonic-liftin-regularity}
and \eqref{eq:almost-disjoint}, we conclude that $B_{r}\subset\{u_{*}>0\}$. 
\end{proof}

The following proposition concerning the continuity of the local minimizers. 
\begin{prop}\label{prop:continuity-minimizer}
Let $\Omega$ be a bounded open set in $\mathbb{R}^{n}$ with $C^{1}$
boundary and $0 \le \lambda<\lambda^{*}(\Omega)$. Let $f,g\in L^{\infty}(\Omega)$
be such that $g\ge0$. If $u_{*}$ is a local minimizer of $\mathcal{J}_{f,g,\lambda,\Omega}$
in $\mathbb{K}(\Omega)$, then $u_{*}\in C(\Omega)$ and there exists
a constant $C_{n}$ such that
\begin{equation}
u_{*}(x)\le C_{n}\mathsf{d}(x)\bigg(\|g\|_{L^{\infty}(B_{2\mathsf{d}(x)}(x))}+\|f+\lambda u_{*}\|_{L^{\infty}(B_{2\mathsf{d}(x)}(x))}\mathsf{d}(x)\bigg),\label{eq:continuous-near-free-boundary}
\end{equation}
for all $x\in\Omega$ near $\partial\{u_{*}>0\}$, where $\mathsf{d}(x)={\rm dist}\,(x,\mathbb{R}^{n}\setminus\{u_{*}>0\})$. 
\end{prop}

\begin{rem}
The assumptions on $\Omega$ is to ensure that $u_{*}\in L^{\infty}(\Omega)$,
see Proposition~\ref{prop:regularity-minimizer}. From this, we know
that 
\[
\|f+\lambda u_{*}\|_{L^{\infty}(\Omega)}\le C(\lambda,\Omega)\|f\|_{L^{\infty}(\Omega)}.
\]
\end{rem}

\begin{proof}
[Proof of Proposition~{\rm \ref{prop:continuity-minimizer}}]
Let $u_{*}$ be a local minimizer $\mathcal{J}_{f,g,\lambda,\Omega}$
in $\mathbb{K}(\Omega)$. Using Lemma~\ref{lem:zero-freq}, we know
that $u_{*}$ is also a local minimizer of $\mathcal{J}_{\tilde{f},g,0,\Omega}$
in $\mathbb{K}(\Omega)$ with $\tilde{f}=f+\lambda u_{*}$. Using
Proposition~\ref{prop:regularity-minimizer}, we know that $\tilde{f}$
is bounded. 

We first showing that $\{u_{*}>0\}$ is an open set and $u_{*}$ is continuous in it. Let $x\in\{u_{*}>0\}$. Using \eqref{eq:PDE1a}
of Proposition~\ref{prop:PDE1}, we know that 
\[
\Delta u_{*}\ge-\tilde{f}_{+}\ge-\|\tilde{f}_{+}\|_{L^{\infty}(B_{r}(x))}\quad\text{in }B_{r}(x)
\]
for all $r>0$ whenever $B_{r}(x)\subset\Omega$. Therefore, using
\cite[Lemma~2.4(b)]{GS96FreeBoundaryPotential} we know that 
\begin{equation}
u_{*}(y+x)\le r^{n}\frac{r+|y|}{(r-|y|)^{n-1}}\bigg(\frac{1}{r^{2}}\dashint_{\partial B_{r}(x)}u_{*}\,dS+\frac{\|\tilde{f}_{+}\|_{L^{\infty}(B_{r}(x))}}{2n}\bigg)\quad\text{for all }y\in B_{r}(0).\label{eq:Poisson5}
\end{equation}
Choosing $y=0$ in \eqref{eq:Poisson5}, we have 
\[
u_{*}(x)\le\dashint_{\partial B_{r}(x)}u_{*}\,dS+r^{2}\frac{\|\tilde{f}_{+}\|_{L^{\infty}(B_{r}(x))}}{2n}\quad\text{for all }r>0\text{ with }B_{r}(x)\subset\Omega,
\]
that is, 
\begin{equation}
\frac{1}{r}\dashint_{\partial B_{r}(x)}u_{*}\,dS\ge\frac{1}{r}u_{*}(x)-r\frac{\|\tilde{f}_{+}\|_{L^{\infty}(B_{r}(x))}}{2n}\quad\text{for all }r>0\text{ with }B_{r}(x)\subset\Omega.\label{eq:Poisson6}
\end{equation}
Since $u_{*}(x_{0})>0$, then we can choose $r_{0}>0$ sufficiently
small such that 
\begin{equation}
\frac{1}{r}u_{*}(x)-r\frac{\|\tilde{f}_{+}\|_{L^{\infty}(B_{r}(x))}}{2n}>2^{n}\bigg(\frac{r}{n}\|\tilde{f}_{-}\|_{L^{\infty}(B_{r}(x))}+\|g\|_{L^{\infty}(B_{r}(x))}\bigg)\quad\text{for all }0<r\le r_{0},\label{eq:sufficient-condition}
\end{equation}
and hence \eqref{eq:continuity-sufficient} satisfies for all $0<r<r_{0}$.
Therefore Lemma~\ref{lem:sufficient-continuity} implies that $B_{r_{0}}(x)\subset\{u_{*}>0\}$
and $u_{*}$ is continuous in $B_{r_{0}}(x)$, which shows that 
\begin{equation}
\{u_{*}>0\}\text{ is an open set and }u_{*}\text{ is continuous in }\{u_{*}>0\}.\label{eq:continuous-interior}
\end{equation}
To prove \eqref{eq:continuous-near-free-boundary}, we only need to show \eqref{eq:continuous-near-free-boundary} for
$x\in\{u_{*}>0\}$. Clearly, \eqref{eq:sufficient-condition} cannot
holds when $r=2\mathsf{d}(x)$, otherwise using the same argument
will shows that $B_{2\mathsf{d}(x)}(x)\subset\{u_{*}>0\}$, which
contradicts with the fact that $B_{\mathsf{d}(x)}(x)$ touches $\partial\{u_{*}>0\}$.
Thus, we have 
\begin{equation*}
\frac{1}{2\mathsf{d}(x)}u_{*}(x)-2\mathsf{d}(x)\frac{\|\tilde{f}_{+}\|_{L^{\infty}(B_{2\mathsf{d}(x)}(x))}}{2n} 
\le 2^{n}\bigg(\frac{2\mathsf{d}(x)}{n}\|\tilde{f}_{-}\|_{L^{\infty}(B_{2\mathsf{d}(x)}(x))}+\|g\|_{L^{\infty}(B_{2\mathsf{d}(x)}(x))}\bigg) 
\end{equation*}
for all $x\in\{u_{*}>0\}$ near $\partial\{u_{*}>0\}$, which implies \eqref{eq:continuous-near-free-boundary}. Combining
\eqref{eq:continuous-interior} and \eqref{eq:continuous-near-free-boundary},
we know that $u_{*}\in C(\Omega)$. 
\end{proof}

Using  Sard's theorem and the coarea formula one can show  that $|\nabla u_{*}|=g$ in some suitable sense. This is stated in the  next proposition. 

\begin{prop}
\label{prop:PDE2} Let $f,g\in L^{\infty}(\mathbb{R}^{n})$ be such
that $g\ge0$. Let $\Omega$ be an open set in $\mathbb{R}^{n}$ and let $\lambda \ge 0$. Let $u_{*}\in\mathbb{K}(\Omega)$
be a local minimizer of $\mathcal{J}_{f,g,\lambda,\Omega}$ in $\mathbb{K}(\Omega)$.
If $g^{2}\in W^{1,1}(\Omega)$, then (for a suitable sequence of $\epsilon_j$)
\[
\lim_{\epsilon_j\searrow0}\int_{\partial\{u>\epsilon_j\}}(|\nabla u_{*}|^{2}-g^{2})\eta\cdot\nu_{\epsilon_j}\,dS=0,
\]
for all $\eta\in(C_{c}^{\infty}(B_{R}))^{n}$. Here, $\nu_{\epsilon_j}$
denotes the outward normal vector of $\partial\{u_{*}>\epsilon_j\}$. 
\end{prop}

\begin{proof}
Let $\eta\in(C_{c}^{\infty}(\Omega))^{n}$ and $\epsilon>0$ small.
We define $u_{\epsilon}\in\mathbb{K}(\Omega)$ by $u_{\epsilon}(\tau_{\epsilon}(x))=u_{*}(x)$,
where $\tau_{\epsilon}(x)=x+\epsilon\eta(x)$. From \eqref{eq:PDE1b}
in Proposition~\ref{prop:PDE1}, it is easy to see that $u_{*}\in C^{1}(\{u_{*}>0\})$.
Note that 
\begin{equation*}
\begin{aligned}
0 & \le\mathcal{J}_{f,g,\lambda,\Omega}(u_{\epsilon})-\mathcal{J}_{f,g,\lambda,\Omega}(u_{*})\\
 & \quad =\int_{\{u_{*}>0\}}\bigg(|\nabla u_{*}|^{2}|\nabla\tau_{\epsilon}|^{-2}-\lambda|u_{*}|^{2}-2(f\circ\tau_{\epsilon})u_{*}+g^{2}\circ\tau_{\epsilon}\bigg)\det(\nabla\tau_{\epsilon})\,dx\\
 & \qquad-\int_{\{u_{*}>0\}}(|\nabla u_{*}|^{2}-\lambda|u_{*}|^{2}-2fu_{*}+g^{2})\,dx\\
 & \quad =\epsilon\int_{\{u_{*}>0\}}(|\nabla u_{*}|^{2}-g^{2})\nabla\cdot\eta\,dx\\
 & \qquad+\epsilon\int_{\{u_{*}>0\}}(-2\nabla u_{*}\cdot(\nabla\eta)\nabla u_{*}+\nabla(g^{2})\cdot\eta)\,dx+o(\epsilon),
\end{aligned}
\end{equation*}
where 
\begin{equation*}
\lim_{\epsilon\searrow0}\frac{o(\epsilon)}{\epsilon}=0.
\end{equation*}
We denote $\otimes:\mathbb{R}^{n}\times\mathbb{R}^{n}\rightarrow\mathbb{R}^{n\times n}$ be the juxtaposition operator defined by $a\otimes b:=ab^{T}$ for all $a,b\in\mathbb{R}^{n}$. Dividing both sides on the inequality above by $\epsilon$ and letting $\epsilon\rightarrow0$, we obtain
\begin{equation*}
\begin{aligned}
0 & =\int_{\{u_{*}>0\}}\nabla\cdot((|\nabla u_{*}|^{2}+g^{2})\eta-2(\eta\otimes\nabla u_{*})\nabla u_{*}\,dx\\
 & =\lim_{\epsilon\searrow0}\int_{\{u_{*}>\epsilon\}}\nabla\cdot((|\nabla u_{*}|^{2}+g^{2})\eta-2(\eta\otimes\nabla u_{*})\nabla u_{*})\,dx\\
 & =\lim_{\epsilon\searrow0}\int_{\partial\{u_{*}>\epsilon\}}\bigg((|\nabla u_{*}|^{2}+g^{2})\eta\cdot\nu_{\epsilon}-2\overbrace{((\eta\otimes\nabla u_{*})\nabla u_{*})\cdot\nu_{\epsilon}}^{=|\nabla u_{*}|^{2}\eta\cdot\nu_{\epsilon}}\bigg)\,dS\\
 & =\lim_{\epsilon\searrow0}\int_{\partial\{u_{*}>\epsilon\}}(-|\nabla u_{*}|^{2}+g^{2})\eta\cdot\nu_{\epsilon}\,dS,
\end{aligned}
\end{equation*}
which conclude our proposition. 
\end{proof}

We now show that the local minimizer is Lipschitz. 
\begin{prop}
\label{prop:Lipschitz-regularity-minimizers} Let $\Omega$ be a bounded open set in $\mathbb{R}^{n}$ with $C^{1}$ boundary and $0 \le \lambda<\lambda^{*}(\Omega)$. Let $f,g\in L^{\infty}(\Omega)$ be such that be such that $g\ge0$ and $g^{2}\in W^{1,1}(\Omega)$. If $u_{*}$ is a local minimizer of $\mathcal{J}_{f,g,\lambda,\Omega}$ in $\mathbb{K}(\Omega)$, then $u_{*}\in C_{{\rm loc}}^{0,1}(\Omega)$ and there exists a constant
$C_{n}$ such that 
\begin{equation}
|\nabla u_{*}(x)|\le C_{n}\bigg(\|g\|_{L^{\infty}(B_{2\mathsf{d}(x)}(x))}+\|f+\lambda u_{*}\|_{L^{\infty}(B_{2\mathsf{d}(x)}(x))}\mathsf{d}(x)\bigg),\label{eq:Lipschitz-near-free-boundary}
\end{equation}
for all $x\in\Omega$ near $\partial\{u_{*}>0\}$. 
\end{prop}

\begin{proof}
It is easy to see that $|\nabla u_{*}(x)|=0$ for all $x\in\Omega\setminus\overline{\{u_{*}>0\}}$.
Since $g^{2}\in W^{1,1}(\Omega)$ and $g\in C(\Omega)$, then we know
that $|\nabla u_{*}(x)|=g(x)=\|g\|_{L^{\infty}(B_{\mathsf{d}(x)}(x))}$
for all $x\in\partial\{u_{*}>0\}$. It is remains to show \eqref{eq:Lipschitz-near-free-boundary}
for $x\in\{u_{*}>0\}$. 

Since $B_{\mathsf{d}(x)}(x)\subset\{u_{*}>0\}$, using \eqref{eq:PDE1b}
of Proposition~\ref{prop:PDE1}, we have $|\Delta u_{*}|=|\tilde{f}|\le\|\tilde{f}\|_{L^{\infty}(B_{\mathsf{d}(x)}(x))}$
in $B_{\mathsf{d}(x)}(x)$. Therefore, using \cite[Lemma~2.4(d)]{GS96FreeBoundaryPotential}
we know that 
\begin{equation}
|\nabla u_{*}(x)|\le n\bigg(\frac{1}{\mathsf{d}(x)}\dashint_{\partial B_{\mathsf{d}(x)}(x)}u_{*}\,dS+\frac{\|\tilde{f}\|_{L^{\infty}(B_{\mathsf{d}(x)}(x))}}{n+1}\mathsf{d}(x)\bigg).\label{eq:Poisson7}
\end{equation}
Since $B_{\mathsf{d}(x)+\epsilon}(x)\not\subset\{u_{*}>0\}$, using
Lemma~\ref{lem:sufficient-continuity} we know that \footnote{If $f,g\in C(\Omega)$, we can improve \eqref{eq:Lipschitz-near-free-boundary}
by replacing $2\mathsf{d}$ by $\mathsf{d}$.}
\begin{equation*}
\begin{aligned}
\frac{1}{\mathsf{d}(x)+\epsilon}\dashint_{\partial B_{\mathsf{d}(x)+\epsilon}(x)}u_{*}\,dS & \le2^{n}\bigg(\frac{\mathsf{d}(x)+\epsilon}{n}\|\tilde{f}_{-}\|_{L^{\infty}(B_{\mathsf{d}(x)+\epsilon}(x))}+\|g\|_{L^{\infty}(B_{\mathsf{d}(x)+\epsilon}(x))}\bigg)\\
 & \le2^{n}\bigg(\frac{\mathsf{d}(x)+\epsilon}{n}\|\tilde{f}_{-}\|_{L^{\infty}(B_{2\mathsf{d}(x)}(x))}+\|g\|_{L^{\infty}(B_{2\mathsf{d}(x)}(x))}\bigg)
\end{aligned}
\end{equation*}
for all sufficiently small $\epsilon>0$. Using the continuity of
$u_{*}$, taking $\epsilon\rightarrow0_{+}$ yields 
\begin{equation}
\frac{1}{\mathsf{d}(x)}\dashint_{\partial B_{\mathsf{d}(x)}(x)}u_{*}\,dS\le2^{n}\bigg(\frac{\mathsf{d}(x)}{n}\|\tilde{f}_{-}\|_{L^{\infty}(B_{2\mathsf{d}(x)}(x))}+\|g\|_{L^{\infty}(B_{2\mathsf{d}(x)}(x))}\bigg).\label{eq:contrapositive}
\end{equation}
Combining \eqref{eq:Poisson7} and \eqref{eq:contrapositive}, we
conclude \eqref{eq:Lipschitz-near-free-boundary} holds for all $x\in\{u_{*}>0\}$
near $\partial\{u_{*}>0\}$. 
\end{proof}

The following lemma gives a sufficient condition in terms of mean average to ensure the local vanishing property of local minimizer. 

\begin{lem}
\label{lem:vanishing-sufficient-condition} Let $\Omega$ be a bounded open set in $\mathbb{R}^{n}$ with $C^{1}$ boundary and $0 \le \lambda<\lambda^{*}(\Omega)$. Let $f,g\in L^{\infty}(\Omega)$ be such that $g\ge0$. Suppose that $u_{*}$ is a local minimizer of $\mathcal{J}_{f,g,\lambda,\Omega}$ in $\mathbb{K}(\Omega)$. If there exists an open set $\Omega'$ with $\overline{\Omega'}\subset\Omega$ such that 
\[
g\ge c>0\quad\text{in }\Omega',
\]
then there exists a constant $C>0$ such that for any sufficiently small ball $B_{r}(x_{0})\subset\Omega'$ we have 
\begin{equation}
\frac{1}{r}\dashint_{\partial B_{r}(x_{0})}u_{*}\le C\implies u_{*}=0\text{ in }B_{\frac{r}{4}}(x_{0}).\label{eq:vanishing-sufficient-condition}
\end{equation}
The constant $C$ depends only on ${\displaystyle \inf_{B_{r}(x_{0})}}g$,
$r\|(f+\lambda u_{*})_{+}\|_{L^{\infty}(B_{r}(x_{0}))}$ and $n$.
Moreover, the constant $C$ is positive whenever ${\displaystyle \inf_{B_{r}(x_{0})}}g>0$
and $r\|(f+\lambda u_{*})_{+}\|_{L^{\infty}(B_{r}(x_{0}))}$ is sufficiently
small. 
\end{lem}

\begin{rem}
\label{rem:distinguishability} If $g\ge c>0$ in a neighborhood of
a point $x_{0}\in\partial\{u_{*}>0\}$, then 
\begin{equation}
u_{*}(x)\ge C\,{\rm dist}\,(x,\mathbb{R}^{n}\setminus\{u_{*}>0\}).\label{eq:distinguishability-near-boundary}
\end{equation}
for some constant $C>0$. We only need to show \eqref{eq:distinguishability-near-boundary}
for $x\in\{u_{*}>0\}$ near $x_{0}$. Writing $r={\rm dist}\,(x,\mathbb{R}^{n}\setminus\{u_{*}>0\})$.
In particular, from the contra-positive statement of \eqref{eq:vanishing-sufficient-condition}
we know that 
\[
\frac{1}{r}\dashint_{\partial B_{r}(x)}u_{*}>C_{1}
\]
for some constant $C_{1}>0$. Using \eqref{eq:PDE1b} in Proposition~\ref{prop:PDE1},
we have 
\[
\Delta u_{*}=-\tilde{f}\le C_{2}:=\|\tilde{f}\|_{L^{\infty}(B_{r}(x))}\quad\text{in }B_{r}(x)\subset\{u_{*}>0\}
\]
with $\tilde{f}=f+\lambda u_{*}$. Using \cite[Lemma~2.4(a)]{GS96FreeBoundaryPotential},
we have 
\[
u_{*}(x)\ge\dashint_{\partial B_{r}}u-\frac{2^{n-1}C_{2}}{n}r^{2}\ge C_{1}r-\frac{2^{n-1}C_{2}}{n}r^{2}\ge Cr,
\]
provided $r={\rm dist}\,(x,\mathbb{R}^{n}\setminus\{u_{*}>0\})$ sufficiently
small. 
\end{rem}

\begin{proof}
[Proof of Lemma~{\rm \ref{lem:vanishing-sufficient-condition}}] 
Without loss of generality, we may assume that $x_{0}=0$.
From \eqref{eq:Poisson5}, we have 
\begin{equation}
u_{*}(y)\le C_{1}\dashint_{\partial B_{r}}u_{*}\,dS+C_{2}r^{2}\|\tilde{f}_{+}\|_{L^{\infty}(B_{r})}\quad\text{for all }y\in B_{\frac{r}{2}}\label{eq:Poisson5-1}
\end{equation}
for some absolute constants $C_{1}$ and $C_{2}$, with $\tilde{f}=f+\lambda u_{*}$.
We define $m:=\inf_{B_{r}}g$, $M:=\|\tilde{f}_{+}\|_{L^{\infty}(B_{r})}$
and 
\begin{equation*}
\begin{aligned}
\mathcal{J}_{r}(v) & :=\int_{B_{\frac{r}{2}}}(|\nabla v|^{2}-2\tilde{f}v+g^{2}\chi_{\{v>0\}})\,dx,\\
\tilde{\mathcal{J}}_{r}(v) & :=\int_{B_{\frac{r}{2}}}(|\nabla v|^{2}-2Mv+m^{2}\chi_{\{v>0\}})\,dx.
\end{aligned}
\end{equation*}
Now given a constant
$\beta>0$ we  consider the problem of minimizing $\tilde{\mathcal{J}}_{r}$
over 
\[
\mathbb{K}_{\beta}=\begin{Bmatrix}\begin{array}{l|l}
v\in H^{1}(B_{\frac{r}{2}}) & v\ge0\text{ in }B_{\frac{r}{2}}\text{ and }v=\beta\text{ on }\partial B_{\frac{r}{2}}\end{array}\end{Bmatrix}.
\]
Following similar   arguments as in \cite[Conclusion of Lemma~2.8]{GS96FreeBoundaryPotential},
we know that there exists a sufficiently small constant $\beta_{0}=\beta_{0}(r,m,M)$
such that if 
\[
0<\frac{r}{2}<\frac{nm}{M}\quad\text{and}\quad0<\beta\le\beta_{0},
\]
then the largest minimizer $v_{\beta}$ of $\tilde{\mathcal{J}}_{r}$
in $\mathbb{K}_{\beta}$ satisfies 
\begin{equation}
v_{\beta}=0\quad\text{in }B_{\frac{r}{4}}.\label{eq:v-beta-vanishing}
\end{equation}
From \cite[(2.13)]{GS96FreeBoundaryPotential}, we have 
\[
\beta_{0}(r,m,M)=r\beta_{0}(1,m,rM).
\]
Next we let $v_{\beta_{0}}$ be the largest minimizer
of $\tilde{\mathcal{J}}_{r}$ in $\mathbb{K}_{\beta_{0}}$, and let
\[
w=\begin{cases}
\min\{u_{*},v_{\beta_{0}}\} & \text{in }B_{\frac{r}{2}},\\
u_{*} & \text{outside }B_{\frac{r}{2}}.
\end{cases}
\]
If 
\[
C_{1}\frac{1}{r}\dashint_{\partial B_{r}(x_{0})}u_{*}<2\beta_{0}(1,m,rM),
\]
then for each sufficiently small $r$ we have (because $\beta_{0}(1,m,0)>0$
when $m>0$ and $\beta_{0}(1,m,\cdot)$ is continuous)
\[
C_{1}\frac{1}{r}\dashint_{\partial B_{r}(x_{0})}u_{*}+C_{2}rM<\beta_{0}(1,m,rM).
\]
Consequently, from \eqref{eq:Poisson5-1} we have 
\[
u_{*}<\beta_{0}(r,m,M)\equiv r\beta_{0}(1,m,rM)\quad\text{on }\partial B_{\frac{r}{2}}.
\]
Since we consider $r>0$ is small, then we know that $w\in\mathbb{K}(\Omega)$
and it is close to $u_{*}$ in sense of \eqref{eq:local-minimizer-metric}.
Since $u_{*}$ is a local minimizer of $\mathcal{J}_{f,g,\lambda,\Omega}$
in $\mathbb{K}(\Omega)$, then by Lemma~\ref{lem:zero-freq} we know
that it is also a local minimizer of $\mathcal{J}_{\tilde{f},g,0,\Omega}$
in $\mathbb{K}(\Omega)$. Hence we know that 
\begin{equation}
\mathcal{J}_{r}(u_{*})\le\mathcal{J}_{r}(w)\equiv\mathcal{J}_{r}(\min\{u_{*},v_{\beta_{0}}\}).\label{eq:vanishing-sufficient1}
\end{equation}
Since $v_{\beta_{0}}$ is a minimizer of $\tilde{\mathcal{J}}_{r}$
in $\mathbb{K}_{\beta_{0}}$, then 
\begin{equation}
\tilde{\mathcal{J}}_{r}(v_{\beta_{0}})\le\tilde{\mathcal{J}}_{r}(\max\{u_{*},v_{\beta_{0}}\}).\label{eq:vanishing-sufficient2}
\end{equation}
Combining \eqref{eq:vanishing-sufficient1} and \eqref{eq:vanishing-sufficient2},
we have 
\begin{equation}
\mathcal{J}_{r}(u_{*})+\tilde{\mathcal{J}}_{r}(v_{\beta_{0}})\le\mathcal{J}_{r}(\min\{u_{*},v_{\beta_{0}}\})+\tilde{\mathcal{J}}_{r}(\max\{u_{*},v_{\beta_{0}}\}).\label{eq:vanishing-sufficient3}
\end{equation}
Following the proof of Proposition~\ref{prop:comparison}, we can
show that 
\begin{equation}
\mathcal{J}_{r}(\min\{u_{1},u_{2}\})+\tilde{\mathcal{J}}_{r}(\max\{u_{1},u_{2}\})\le\mathcal{J}_{r}(u_{1})+\tilde{\mathcal{J}}_{r}(u_{2})\label{eq:comparison-small}
\end{equation}
for any $u_{1},u_{2}\in H^{1}(B_{\frac{r}{2}})$ with $u_{1}\ge0$
and $u_{2}\ge0$. Combining \eqref{eq:vanishing-sufficient3} and
\eqref{eq:comparison-small}, we obtain 
\[
\mathcal{J}_{r}(u_{*})+\tilde{\mathcal{J}}_{r}(v_{\beta_{0}})=\mathcal{J}_{r}(\min\{u_{*},v_{\beta_{0}}\})+\tilde{\mathcal{J}}_{r}(\max\{u_{*},v_{\beta_{0}}\}).
\]
Since $v_{\beta_{0}}$ be the largest minimizer of $\tilde{\mathcal{J}}_{r}$
in $\mathbb{K}_{\beta_{0}}$, then $\max\{u_{*},v_{\beta_{0}}\}\le v_{\beta_{0}}$,
which implies $0\le u_{*}\le v_{\beta_{0}}$ in $B_{\frac{r}{2}}$.
Combining this with \eqref{eq:v-beta-vanishing}, we conclude \eqref{eq:vanishing-sufficient-condition}
with $C=2C_{1}^{-1}\beta_{0}(1,m,rM)$. 
\end{proof}

The following lemma concerns the density of the boundary of  $\{u_{*}>0\}$. 

\begin{lem}
\label{lem:density} Let $\Omega$ be a bounded open set in $\mathbb{R}^{n}$ with $C^{1}$ boundary and $0 \le \lambda<\lambda^{*}(\Omega)$. Let $f,g\in L^{\infty}(\Omega)$ be such that $g\ge0$ and $g^{2}\in H^{1,1}(\Omega)$. Suppose that $u_{*}$ is a local minimizer of $\mathcal{J}_{f,g,\lambda,\Omega}$ in $\mathbb{K}(\Omega)$. If $g\ge c>0$ in a neighborhood of a point $x_{0}\in\partial\{u_{*}>0\}$, then there exist constants $c_{1}$ and $c_{2}$ such that 
\begin{equation}
0<c_{1}\le\frac{|B_{r}(x_{0})\cap\{u_{*}>0\}|}{|B_{r}(x_{0})|}\le c_{2}<1\quad\text{for all sufficiently small }r>0.\label{eq:density-boundary-point}
\end{equation}
\end{lem}

\begin{rem}
\label{rem:boundary-measure} For each (Lebesgue) measurable set $E\subset\mathbb{R}^{n}$,
it is well-known that 
\[
\lim_{r\rightarrow0}\frac{|B_{r}(x_{0})\cap E|}{|B_{r}(x_{0})|}=\begin{cases}
1 & \text{for a.e. }x\in E,\\
0 & \text{for a.e. }x\in\mathbb{R}^{n}\setminus E.
\end{cases}
\]
Therefore, the \emph{measure theoretic boundary} $\partial_{{\rm mes}}E$
of $E$ is defined to be the set of points $x\in\mathbb{R}^{n}$ such
that both $E$ and $\mathbb{R}^{n}\setminus E$ have positive upper
Lebesgue density at $x$. In particular, $\partial_{{\rm mes}}E\subset\partial E$.
As an immediate corollary of Lemma~\ref{lem:density}, we know that
\labeltext{$\partial_{\rm mes} E$ measure theoretic boundary}{index:mes-boundary}
\[
\partial\{u_{*}>0\}\cap\{g>0\}=\partial_{{\rm mes}}\{u_{*}>0\}\cap\{g>0\}.
\]
\end{rem}

\begin{proof}
[Proof of Lemma~{\rm \ref{lem:density}}] Without loss of generality,
we may assume that $x_{0}=0$. Using Remark~\ref{rem:distinguishability},
we know that there exists a point $y\in\partial B_{\frac{r}{2}}$
and a constant $c>0$ such that with $u_{*}(y)\ge cr$. Using \eqref{eq:Poisson6}
on $B_{\kappa r}(y)\subset\Omega$ provided $\kappa>0$ is small,
we have 
\[
\frac{1}{\kappa r}\dashint_{\partial B_{\kappa r}(y)}u_{*}\,dS\ge\frac{1}{\kappa r}u_{*}(y)-\kappa r\frac{\|\tilde{f}_{+}\|_{L^{\infty}(B_{\kappa r}(y))}}{2n}\ge\frac{c}{\kappa}.
\]
Using Lemma~\ref{lem:sufficient-continuity}, we know that $B_{\kappa r}(y)\subset\{u_{*}>0\}\cap B_{r}$
for sufficiently small $\kappa>0$, and hence 
\begin{equation}
\frac{|B_{r}\cap\{u_{*}>0\}|}{|B_{r}|}\ge\frac{|B_{\kappa r}(y)|}{|B_{r}|}=\frac{|B_{\kappa r}|}{|B_{r}|}=\kappa^{n},\label{eq:density-boudary-point-lower-bound}
\end{equation}
which prove the lower bound of \eqref{eq:density-boundary-point}
with $c_{1}=\kappa^{n}$. 

Combining \eqref{eq:density-boudary-point-lower-bound} with Lemma~\ref{lem:vanishing-sufficient-condition},
we know that 
\begin{equation}
\dashint_{\partial B_{r}}u_{*}\ge cr.\label{eq:u-star-lower-bound}
\end{equation}
Now let $v$ be as in \eqref{eq:harmonic-lifting-v}. From \eqref{eq:Poisson0-1}
and Poincar\'{e} inequality, we have 
\begin{equation}
|B_{r}\cap\{u_{*}=0\}|\ge c\int_{B_{r}}|\nabla(u_{*}-v)|^{2}\,dx\ge\frac{c}{r^{2}}\int_{B_{r}}|u_{*}-v|^{2}\,dx\label{eq:density-boundary-point-upper-bound}
\end{equation}
for each sufficiently small $r>0$. By restricting \eqref{eq:Poisson-v}
on $B_{\kappa r}$, we obtain 
\begin{equation}
v(y)\ge\frac{(1-\kappa)}{(1+\kappa)^{n-1}}\bigg(\dashint_{\partial B_{r}}u_{*}\,dS-Cr^{2}\bigg)\quad\text{for all }y\in B_{\kappa r}.\label{eq:v-lower-bound}
\end{equation}
Using Proposition~\ref{prop:Lipschitz-regularity-minimizers} (which
required $g^{2}\in W^{1,1}(\Omega)$), we know that $u_{*}$ is Lipschitz,
hence we know that 
\begin{equation}
u_{*}(y)\le C\kappa r\quad\text{for all }y\in B_{\kappa r},\label{eq:u-star-lipschitz}
\end{equation}
because $0\in\partial\{u_{*}>0\}$. Combining \eqref{eq:u-star-lipschitz}
with \eqref{eq:v-lower-bound} as well as \eqref{eq:u-star-lower-bound},
for all sufficiently small $r>0$ we have 
\[
(v-u_{*})(y)\ge\frac{(1-\kappa)}{(1+\kappa)^{n-1}}\dashint_{\partial B_{r}}u_{*}\,dS-C\kappa r-Cr^{2}\ge(c-C\kappa)r\ge c'r\quad\text{for all }y\in B_{\kappa r},
\]
provided $\kappa>0$ sufficiently small, hence 
\[
|(v-u_{*})|=v-u_{*}\ge c'r\quad\text{in }B_{\kappa r}.
\]
Therefore from \eqref{eq:density-boundary-point-upper-bound}, we
have 
\[
|B_{r}\cap\{u_{*}=0\}|\ge\frac{c}{r^{2}}\int_{B_{\kappa r}}|u_{*}-v|^{2}\,dx\ge\frac{c}{r^{2}}|B_{\kappa r}|(c')^{2}r^{2}=c''|B_{r}|\kappa^{n},
\]
and consequently 
\[
\frac{|B_{r}\cap\{u_{*}>0\}|}{|B_{r}|}=1-\frac{|B_{r}\cap\{u_{*}=0\}|}{|B_{r}|}\le1-\frac{c''|B_{r}|\kappa^{n}}{|B_{r}|}=1-c''\kappa^{n},
\]
which prove the upper bound of \eqref{eq:density-boundary-point}
with $c_{2}=1-c''\kappa^{n}$. 
\end{proof}

\begin{proof}[Proof of Proposition~{\rm \ref{prop:PDE3}}]
Here we only prove the result when $\partial\{g>0\}\cap\Omega\ne\emptyset$,
as the case when $\partial\{g>0\}\cap\Omega=\emptyset$ can be easily
prove using the same idea by omitting some paragraphs. 

\bigskip

\noindent \textbf{Step 1: Initialization.} Using Proposition~\ref{prop:regularity-minimizer},
we know that $\tilde{f}=f+\lambda u_{*}\in L^{\infty}(\Omega)$. Using
\eqref{eq:PDE1a} in Proposition~\ref{prop:PDE1}, we know that $\Delta u_{*}$ is a signed Radon measure in $\Omega$ and $\Delta u_{*}\ge-\tilde{f}\ge-\|\tilde{f}\|_{L^{\infty}(\Omega)}$
in $\Omega$. Then we see that 
\[
\Delta\bigg(u_{*}+\frac{\|\tilde{f}\|_{L^{\infty}(\Omega)}}{2n}|x|^{2}\bigg)=\Delta u_{*}+\|\tilde{f}\|_{L^{\infty}(\Omega)}\ge0\quad\text{in }\Omega.
\]
Since $u_{*}\ge0$ in $\Omega$, using \cite[Lemma~2.16]{GS96FreeBoundaryPotential}, we know that $\Delta u_{*}\ge0$ in $\Omega\setminus\{u_{*}>0\}$. We now define
\[
\rho:=\Delta u_{*}+\tilde{f}\chi_{\{u_{*}>0\}}.
\]
Using \eqref{eq:PDE1b} in Proposition~\ref{prop:PDE1}, we know
that 
\begin{equation}
\rho=0\quad\text{in }\{u_{*}>0\}.\label{eq:lambda0-in-free}
\end{equation}
Clearly, $\rho=0$ in the open set $\Omega\setminus\overline{\{u_{*}>0\}}$,
therefore, we know that 
\[
\rho:=\Delta u_{*}+\tilde{f}\chi_{\{u_{*}>0\}}\text{ is a non-negative Radon measure supported on }\partial\{u_{*}>0\}.
\]
From \eqref{eq:PDE1c}, we further know that 
\begin{equation}
\begin{aligned}
& \rho:=\Delta u_{*}+\tilde{f}\chi_{\{u_{*}>0\}}\text{ is a non-negative Radon measure} \\
& \text{supported on }\partial\{u_{*}>0\}\cap\{g>0\}.
\end{aligned}\label{eq:lambda-positivity}
\end{equation}
For each $x_{0}\in\Omega$, we estimate 
\begin{equation}
\bigg|\int_{B_{r}(x_{0})}\Delta u_{*}\,dx\bigg|\le\int_{\partial B_{r}(x_{0})}|\nabla u_{*}|\,dS\le Cr^{n-1}\sup_{\partial B_{r}(x)}|\nabla u_{*}|\le Cr^{n-1} \label{eq:Upper-bound}
\end{equation}
for all sufficiently small $r>0$, where the last inequality follows from \eqref{eq:Lipschitz-near-free-boundary}
in Proposition~\ref{prop:Lipschitz-regularity-minimizers} and the
assumption $\overline{\{u_{*}>0\}}\subset\Omega$. This shows that
$\Delta u_{*}$, as well as $\lambda$, is absolutely continuous with
respect to $\mathscr{H}^{n-1}$. 

\bigskip

\noindent \textbf{Step 2: Showing that $\rho=0$ on $\partial\{u_{*}>0\}\setminus\{g>0\}$.}
For each $x_{0}\in\partial\{u_{*}>0\}\setminus\{g>0\}$, using \eqref{eq:Lipschitz-near-free-boundary}
in Proposition~\ref{prop:Lipschitz-regularity-minimizers}, we have
\footnote{In particular when $g\equiv0$ in $\Omega$ (i.e. $G\cap\Omega=\emptyset$),
we even can choose $\alpha=1$.}
\begin{equation}
\bigg|\int_{B_{r}(x_{0})}\Delta u_{*}\,dx\bigg|\le Cr^{n-1+\alpha}\quad\text{for all sufficiently small }r>0,\label{eq:laplace-u-estimate2}
\end{equation}
which shows that $\Delta u_{*}$, as well as $\rho$, is absolutely
continuous with respect to $\mathscr{H}^{n-1+\alpha}$ on $\partial\{u_{*}>0\}\setminus\{g>0\}$.
Using the assumption $\mathscr{H}^{n-1+\alpha}(\partial\{g>0\}\cap\Omega)=0$,
we know that 
\begin{equation}
\rho=\Delta u_{*}=0\quad\text{on }\partial\{u_{*}>0\}\cap\partial\{g>0\}.\label{eq:lambda0-case1}
\end{equation}

On the other hand, using \eqref{eq:PDE1a} and \eqref{eq:PDE1c} in
Proposition~\ref{prop:PDE1}, we know that 
\[
-(f+\lambda u_{*})_{+}\le\Delta u_{*}\le-(f+\lambda u_{*})\quad\text{in }\Omega\setminus\overline{\{g>0\}},
\]
and thus $\Delta u_{*}\in L^{\infty}(\Omega\setminus\overline{\{g>0\}})$.
Using \cite[Lemma~A.4 of Chapter~II]{KS00IntroductionVariationalInequalities},
we know that 
\begin{equation}
\rho=\Delta u_{*}=0\quad\text{on }\partial\{u_{*}>0\}\setminus\overline{\{g>0\}}.\label{eq:lambda0-case2}
\end{equation}
Combining \eqref{eq:lambda0-case1} and \eqref{eq:lambda0-case2},
we know that 
\begin{equation}
\rho=0\quad\text{on }\partial\{u_{*}>0\}\setminus\{g>0\}.\label{eq:lambda0-case-combine}
\end{equation}
Next, we want to study the behavior of $\rho$ on $\partial\{u_{*}>0\}\cap\{g>0\}$. 

\bigskip

\noindent \textbf{Step 3: Proving $\{u_{*}>0\}$ has locally finite perimeter
in $\{g>0\}$ and \eqref{eq:boundary-and-reduced-boundary}.} We first
show that there exists a constant $C>0$, depending on $\inf_{B_{r}(x_{0})}g$,
such that 
\begin{equation}
\int_{B_{r}(x_{0})}\Delta u_{*}\,dx\ge Cr^{n-1}\label{eq:Lower-bound}
\end{equation}
for all sufficiently small $r>0$ and for all $x_{0}\in\partial\{u_{*}>0\}\cap\{g>0\}$. 

Let $\Phi_{y}$ be the (positive) Green function for $-\Delta$ in
$B_{r}(x_{0})$ with pole $y\in B_{r}(x_{0})$, that is, 
\[
\begin{cases}
\Delta\Phi_{y}=-\delta_{y} & \text{in }B_{r}(x_{0}),\\
\Phi_{y}\ge0 & \text{in }B_{r}(x_{0}),\\
\Phi_{y}=0 & \text{on }\partial B_{r}(x_{0}).
\end{cases}
\]
Using integration by parts, we can easily see that 
\[
\int_{B_{r}(x_{0})}\Phi_{y}\Delta u_{*}\,dx=-u_{*}(y)+\int_{\partial B_{r}(x_{0})}u_{*}\partial_{-\nu}\Phi_{y}\,dS.
\]
Using Lemma~\ref{lem:sufficient-continuity}, for each sufficiently
small $\kappa>0$, there is a point $y\in\partial B_{\kappa r}(x_{0})$
with 
\[
u_{*}(y)\ge c\kappa r>0,
\]
and since $u_{*}$ is Lipschitz, we have 
\[
u_{*}(y)\le C\kappa r\quad\text{and}\quad u_{*}>0\text{ in }B_{c(\kappa)r}(y)
\]
for some constant $c(\kappa)$. Hence we have 
\begin{equation}
\int_{B_{r}(x_{0})}\Phi_{y}\Delta u_{*}\,dx\ge-u_{*}(y)+c\dashint_{\partial B_{r}(x_{0})}u_{*}\,dS\ge-C\kappa r+cr\ge c'r,\label{eq:Lower-bound-intermediate1}
\end{equation}
which can be done by possibly replacing a smaller $\kappa>0$. 

On the other hand, using \eqref{eq:lambda0-in-free} we know that
$\rho=0$ in $B_{c(\kappa)r}(y)\subset\{u_{*}>0\}$, hence we have
\begin{equation}
\begin{aligned}
& \int_{B_{r}(x_{0})}\Phi_{y}\Delta u_{*}\,dx \\
& \quad =\int_{B_{r}(x_{0})\setminus B_{c(\kappa)r}(y)}\Phi_{y}\lambda\,dx-\int_{B_{r}(x_{0})}\Phi_{y}\tilde{f}\chi_{\{u_{*}>0\}}\,dx \\
 & \quad =\int_{B_{r}(x_{0})\setminus B_{c(\kappa)r}(y)}\Phi_{y}\Delta u_{*}\,dx-\int_{B_{c(\kappa)r}(y)}\Phi_{y}\tilde{f}\,dx \\
 & \quad \le\|\Phi_{y}\|_{L^{\infty}(B_{r}(x_{0})\setminus B_{c(\kappa)r}(y))}\int_{B_{r}(x_{0})}\Delta u_{*}\,dx+\|\tilde{f}\|_{L^{\infty}(\Omega)}\int_{B_{c(\kappa)r}(y)}\Phi_{y}\,dx \\
 & \quad \le C(\kappa)r^{2-n}\int_{B_{r}(x_{0})}\Delta u_{*}\,dx+C(\kappa)r^{2}
\end{aligned}\label{eq:Lower-bound-intermediate2}
\end{equation}
for all sufficiently small $r>0$. Combining \eqref{eq:Lower-bound-intermediate1}
and \eqref{eq:Lower-bound-intermediate2}, we conclude \eqref{eq:Lower-bound}. 

Let $K$ be any compact set in $\partial\{u_{*}>0\}\cap\{g>0\}$.
We now cover $K$ by the balls $B_{r}(x_{0})$ given in \eqref{eq:Lower-bound},
and we know that 
\begin{equation}
\mathscr{H}^{n-1}(K)\le C\int_{K}\Delta u_{*}\,dx\overset{\eqref{eq:Upper-bound}}{<}\infty,
\end{equation}
which shows that $\{u_{*}>0\}$ has locally finite perimeter in $\{g>0\}$.
Combining with Remark~\ref{rem:boundary-measure} and \eqref{eq:reduced-boundary-measure},
we conclude \eqref{eq:boundary-and-reduced-boundary}. 

\bigskip 

\noindent \textbf{Step 4: Sketching the proof of \eqref{eq:PDE3-main}.} Combining
\eqref{eq:boundary-and-reduced-boundary} and \eqref{eq:lambda-positivity},
we see that 
\[
\Delta u_{*}+\tilde{f}\chi_{\{u_{*}>0\}}=h\mathscr{H}^{n-1}\lfloor\partial_{{\rm red}}\{u_{*}>0\}
\]
for some Borel function $h\ge0$ on $\partial_{{\rm red}}\Omega\cap\{g>0\}$.
It just remains to show 
\begin{equation}
h(x_{0})=g(x_{0})\quad\text{for }\mathscr{H}^{n-1}\text{-a.e. }x_{0}\in\partial_{{\rm red}}\{u_{*}>0\}\cap\{g>0\}.\label{eq:to-show-technical1}
\end{equation}
Despite the ideas are virtually same as in \cite[4.7--5.5]{AC81RegularityFreeBoundary},
here we still present the details for reader's convenience. It is
enough to prove \eqref{eq:to-show-technical1} for those $x_{0}\in\partial_{{\rm red}}\{u_{*}>0\}\cap\{g>0\}$
which satisfies 
\begin{subequations}
\begin{align}
\limsup_{r\rightarrow\infty}\frac{\mathscr{H}^{n-1}(B_{r}(x_{0})\cap\partial\{u_{*}>0\})}{\omega_{n-1}r^{n-1}} & \le1\quad\text{(see Lemma \ref{lem:Hausdorff-density})},\label{eq:to-show-technical1-reduction1a}\\
\limsup_{r\rightarrow\infty}\dashint_{\partial\{u_{*}>0\}\cap B_{r}(x)}|h(x)-h(x_{0})|\,dS(x) & =0,\label{eq:to-show-technical1-reduction1b}
\end{align}
\end{subequations}
see \cite[Remark~4.9]{AC81RegularityFreeBoundary}. Without loss of
generalization, we assume that 
\[
x_{0}=0\quad\text{and}\quad\nu_{\{u_{*}>0\}}(0)=e_{n}=(0,\cdots,0,1).
\]

Define the blow-up sequences 
\begin{align*}
& u_{n}(x)=nu_{*}(n^{-1}x),\quad\tilde{f}_{n}(x)=n^{-1}\tilde{f}(n^{-1}x)\\
& g(x)=g(n^{-1}x),\quad h_{n}(x)=h(n^{-1}x),\quad\Omega_{n}=\{u_{n}>0\}.
\end{align*}
Note that $u,f,g$ are scaled according to \cite[Remark~2.7 (with $\alpha = -1$)]{GS96FreeBoundaryPotential}.
It is also easy to see that 
\begin{subequations}
\begin{align}
|\tilde{f}_{n}| & \le C/n,\label{eq:approx-f}\\
\int_{B_{1}}|g(n^{-1}x)-g(0)|\,dx & \rightarrow0,\label{eq:approx-g}\\
\int_{\partial\Omega_{n}\cap B_{1}}|h(n^{-1}x)-h(0)|\,dS(x) & \rightarrow0,\label{eq:approx-h}
\end{align}
\end{subequations}
as $n\rightarrow\infty$. We define the half-space $H:=\{x_{n}<0\}$.
Using \cite[Theorem~1 in Section~5.7.2]{EG15MeasureTheory}, we know that
\begin{equation}
|(\Omega_{n}\triangle H)\cap B_{1}|\rightarrow0\quad\text{as }n\rightarrow\infty,\label{eq:difference-set}
\end{equation}
where $\triangle$ denotes the symmetric difference between the sets.
Using Proposition~\ref{prop:Lipschitz-regularity-minimizers} (together
with the assumption $\overline{\{u_{*}>0\}}\subset\Omega$) and Remark~\ref{rem:distinguishability},
we know that $|\nabla u_{n}|\le C$. It follows that there exists
a Lipschitz continuous limit function $u_{0}\ge0$ such that, for
a subsequence, 
\begin{align*}
u_{n} & \rightarrow u_{0}\quad\text{uniformly in }B_{1},\\
\nabla u_{n} & \rightarrow\nabla u_{0}\quad L^{\infty}(B_{1})\text{-weak}*.
\end{align*}
Using Remark~\ref{rem:distinguishability}, we know that $u_{n}(x)\ge C\,{\rm dist}\,(x,\mathbb{R}^{n}\setminus\Omega_{n})$.
Using \eqref{eq:PDE1b} in Proposition~\ref{prop:PDE1} and \eqref{eq:approx-f},
we know that 
\[
|\Delta u_{n}|=|\tilde{f}_{n}|\le C/n\quad\text{in }\Omega_{n}.
\]
Therefore we know that $u_{0}$ is harmonic in $\Omega_{0}$ and in
particular $\Omega_{n}\cap B_{1}\rightarrow\Omega_{0}\cap B_{1}$
in measure. Therefore, from \eqref{eq:difference-set} we have $|\Omega_{0}\triangle H|=0$.
Since $\Omega_{0}$ is open, then we conclude that 
\[
\Omega_{0}\subset H\quad\text{and}\quad|H\setminus\Omega_{0}|=0.
\]
Using \cite[Theorem~4.8]{AC81RegularityFreeBoundary} together with
\eqref{eq:to-show-technical1-reduction1a} and \eqref{eq:to-show-technical1-reduction1b},
we know that 
\[
\Omega_{0}=H\quad\text{and}\quad u_{0}(x)=h(0)(-x_{n})_{+}.
\]
On the other hand, using \cite[Lemma~5.4]{AC81RegularityFreeBoundary},
we know that $u_{0}$ is a global minimum of 
\[
\mathcal{J}_{0}(v):=\int_{B}(|\nabla v|^{2}+g(0)^{2}\chi_{\{v>0\}})\quad\text{among all }v\in\mathbb{K}(B_{1}).
\]
Therefore, using Proposition~\ref{prop:PDE2}, we conclude $h(0)=g(0)$,
and we complete the proof. 
\end{proof}

\section{\label{appen:computations}Computations related to Bessel functions}

The main purpose of this appendix is to exhibit the details of computation for Lemma~{\rm \ref{lem:radially1}}. 

\begin{proof}[Computations of \eqref{eq:explicit-formula}]

Since the solution space of $u''(r)+\frac{n-1}{r}u'(r)+\lambda u(r)=0$
for $r>0$ is spanned by 
\[
r^{\frac{2-n}{2}}J_{\frac{n-2}{2}}(\sqrt{\lambda}r)\quad\text{and}\quad r^{\frac{2-n}{2}}Y_{\frac{n-2}{2}}(\sqrt{\lambda}r),
\]
then the solution of \eqref{eq:ODE-radially} must satisfies 
\begin{align*}
u(r) & =\frac{b-a}{\lambda}+c_{1}r^{\frac{2-n}{2}}J_{\frac{n-2}{2}}(\sqrt{\lambda}r)\quad\text{for }r\in(0,r_{1}),\\
u(r) & =\frac{b}{\lambda}+c_{2}r^{\frac{2-n}{2}}J_{\frac{n-2}{2}}(\sqrt{\lambda}r)+c_{3}r^{\frac{2-n}{2}}Y_{\frac{n-2}{2}}(\sqrt{\lambda}r)\quad\text{for }r_{1}<r.
\end{align*}
Then we know that 
\begin{align*}
u'(r) & =-c_{1}\sqrt{\lambda}r^{\frac{2-n}{2}}J_{\frac{n}{2}}(\sqrt{\lambda}r)\quad\text{for }r\in(0,r_{1}),\\
u'(r) & =-c_{2}\sqrt{\lambda}r^{\frac{2-n}{2}}J_{\frac{n}{2}}(\sqrt{\lambda}r)-c_{3}\sqrt{\lambda}r^{\frac{2-n}{2}}Y_{\frac{n}{2}}(\sqrt{\lambda}r)\quad\text{for }r_{1}<r.
\end{align*}
Since $u$ is $C^{1}$ at $r_{1}$, and since $\lambda>0$ and $r>0$,
we have 
\begin{align*}
(c_{2}-c_{1})J_{\frac{n-2}{2}}(\sqrt{\lambda}r_{1})+c_{3}Y_{\frac{n-2}{2}}(\sqrt{\lambda}r_{1}) & =-\frac{ar_{1}^{\frac{n-2}{2}}}{\lambda},\\
(c_{2}-c_{1})J_{\frac{n}{2}}(\sqrt{\lambda}r_{1})+c_{3}Y_{\frac{n}{2}}(\sqrt{\lambda}r_{1}) & =0,
\end{align*}
that is, 
\[
\begin{pmatrix}\begin{array}{cc}
J_{\frac{n-2}{2}}(\sqrt{\lambda}r_{1}) & Y_{\frac{n-2}{2}}(\sqrt{\lambda}r_{1})\\
J_{\frac{n}{2}}(\sqrt{\lambda}r_{1}) & Y_{\frac{n}{2}}(\sqrt{\lambda}r_{1})
\end{array}\end{pmatrix}\begin{pmatrix}c_{2}-c_{1}\\
c_{3}
\end{pmatrix}=-\frac{ar_{1}^{\frac{n-2}{2}}}{\lambda}\begin{pmatrix}1\\
0
\end{pmatrix}.
\]
Note that 
\[
\det\begin{pmatrix}\begin{array}{cc}
J_{\frac{n-2}{2}}(\sqrt{\lambda}r_{1}) & Y_{\frac{n-2}{2}}(\sqrt{\lambda}r_{1})\\
J_{\frac{n}{2}}(\sqrt{\lambda}r_{1}) & Y_{\frac{n}{2}}(\sqrt{\lambda}r_{1})
\end{array}\end{pmatrix}=-\frac{2}{\pi\sqrt{\lambda}r_{1}},
\]
then 
\begin{align*}
\begin{pmatrix}c_{2}-c_{1}\\
c_{3}
\end{pmatrix} & =\frac{\pi\sqrt{\lambda}r_{1}}{2}\begin{pmatrix}\begin{array}{cc}
Y_{\frac{n}{2}}(\sqrt{\lambda}r_{1}) & -Y_{\frac{n-2}{2}}(\sqrt{\lambda}r_{1})\\
-J_{\frac{n}{2}}(\sqrt{\lambda}r_{1}) & J_{\frac{n-2}{2}}(\sqrt{\lambda}r_{1})
\end{array}\end{pmatrix}\frac{ar_{1}^{\frac{n-2}{2}}}{\lambda}\begin{pmatrix}1\\
0
\end{pmatrix}\\
 & =\frac{a\pi r_{1}^{\frac{n}{2}}}{2\sqrt{\lambda}}\begin{pmatrix}Y_{\frac{n}{2}}(\sqrt{\lambda}r_{1})\\
-J_{\frac{n}{2}}(\sqrt{\lambda}r_{1})
\end{pmatrix}.
\end{align*}
Hence we know that 
\begin{equation*}
\begin{aligned}
u(r) & =\chi_{\{r<r_{1}\}}\bigg[\frac{b-a}{\lambda}+c_{1}r^{\frac{2-n}{2}}J_{\frac{n-2}{2}}(\sqrt{\lambda}r)\bigg] \\
 & \quad+\chi_{\{r>r_{1}\}}\bigg[\frac{b}{\lambda}+c_{2}r^{\frac{2-n}{2}}J_{\frac{n-2}{2}}(\sqrt{\lambda}r)+c_{3}r^{\frac{2-n}{2}}Y_{\frac{n-2}{2}}(\sqrt{\lambda}r)\bigg] \\
 & =\chi_{\{r<r_{1}\}}\bigg[\frac{b-a}{\lambda}+c_{1}r^{\frac{2-n}{2}}J_{\frac{n-2}{2}}(\sqrt{\lambda}r)\bigg] \\
 & \quad+\chi_{\{r>r_{1}\}}\bigg[\frac{b}{\lambda}+\bigg(c_{1}+\frac{a\pi r_{1}^{\frac{n}{2}}}{2\sqrt{\lambda}}Y_{\frac{n}{2}}(\sqrt{\lambda}r_{1})\bigg)r^{\frac{2-n}{2}}J_{\frac{n-2}{2}}(\sqrt{\lambda}r) \\
 & \qquad -\frac{a\pi r_{1}^{\frac{n}{2}}}{2\sqrt{\lambda}}J_{\frac{n}{2}}(\sqrt{\lambda}r_{1})r^{\frac{2-n}{2}}Y_{\frac{n-2}{2}}(\sqrt{\lambda}r)\bigg] \\
 & =\frac{b-a}{\lambda}+c_{1}r^{\frac{2-n}{2}}J_{\frac{n-2}{2}}(\sqrt{\lambda}r) \\
 & \quad+\chi_{\{r>r_{1}\}}\bigg[\frac{a}{\lambda}+\frac{a\pi r_{1}^{\frac{n}{2}}}{2\sqrt{\lambda}}r^{\frac{2-n}{2}}\bigg(Y_{\frac{n}{2}}(\sqrt{\lambda}r_{1})J_{\frac{n-2}{2}}(\sqrt{\lambda}r)-J_{\frac{n}{2}}(\sqrt{\lambda}r_{1})Y_{\frac{n-2}{2}}(\sqrt{\lambda}r)\bigg)\bigg],
\end{aligned}
\end{equation*}
which conclude \eqref{eq:explicit-formula} and 
\begin{equation}
\begin{aligned}
u'(r) & =-c_{1}\sqrt{\lambda}r^{\frac{2-n}{2}}J_{\frac{n}{2}}(\sqrt{\lambda}r) \\
 & \quad-\chi_{\{r>r_{1}\}}\frac{a\pi r_{1}^{\frac{n}{2}}}{2}r^{\frac{2-n}{2}}\left(Y_{\frac{n}{2}}(\sqrt{\lambda}r_{1})J_{\frac{n}{2}}(\sqrt{\lambda}r)-J_{\frac{n}{2}}(\sqrt{\lambda}r_{1})Y_{\frac{n}{2}}(\sqrt{\lambda}r)\right).
\end{aligned}\label{eq:explicit-formula-derivative}
\end{equation}
\end{proof}

\begin{proof}[Proof of \eqref{eq:ODE-radially-BC}]
Since $J_{\frac{n}{2}}(\sqrt{\lambda}r)$ is non-negative on $(0,R)$
by the assumption on $\lambda$, we deduce that $u'$ has constant
sign on $(0,r_{1})$. We see that 
\begin{equation}
u'(r)=\overbrace{r^{\frac{2-n}{2}}J_{\frac{n}{2}}(\sqrt{\lambda}r)}^{>0}\bigg[-c_{1}\sqrt{\lambda}-\frac{a\pi r_{1}^{\frac{n}{2}}}{2}\bigg(Y_{\frac{n}{2}}(\sqrt{\lambda}r_{1})-J_{\frac{n}{2}}(\sqrt{\lambda}r_{1})\frac{Y_{\frac{n}{2}}(\sqrt{\lambda}r)}{J_{\frac{n}{2}}(\sqrt{\lambda}r)}\bigg)\bigg] \label{eq:explicit-formula-derivative2}
\end{equation}
for all $r > r_{1}$. Since 
\[
\frac{\partial}{\partial r}\bigg(\frac{Y_{\frac{n}{2}}(\sqrt{\lambda}r)}{J_{\frac{n}{2}}(\sqrt{\lambda}r)}\bigg)=\frac{2}{\pi rJ_{\frac{n}{2}}(\sqrt{\lambda}r)^{2}}>0\quad\text{for }r\in(0,R)\subset(0,j_{\frac{n-2}{2},1}\lambda^{-\frac{1}{2}}),
\]
we deduce that there is at most one point $r'\in(r_{1},R)$ where
$u'(r')=0$ and $u'$ is negative on $(r_{1},r')$ and positive on
$(r',R)$ (not excluding the possibility that $r'=r_{1}$ or $r=R$).
This implies that $u$ can at most have two zeros in $(0,R]$. Since
$u$ has at least one zero in $(0,R]$, then we have the following
cases: 
\begin{enumerate}[leftmargin = 25pt]
\renewcommand{\labelenumi}{\theenumi}
\renewcommand{\theenumi}{{\rm (\arabic{enumi})}}
\item If $u$ has exactly one zero $0<\rho_{0}\le R$, then $u'(\rho_{0})\le0$. 
\item If $u$ has exactly two zeros $0<\rho_{1}<\rho_{2}\le R$, then $u'(\rho_{1})\le0$
and $u'(\rho_{2})\ge0$. In this case, we choose $\rho_{0}=\rho_{1}$. 
\end{enumerate}
In either case, we conclude \eqref{eq:ODE-radially-BC}. 
\end{proof}

\begin{proof}[Proof of \eqref{eq:rho-special}]
Plugging
$u(\rho)=0$ in \eqref{eq:explicit-formula}, we have 
\begin{equation*}
\begin{aligned}
0 & =\frac{b-a}{\lambda}+c_{1}\rho^{\frac{2-n}{2}}J_{\frac{n-2}{2}}(\sqrt{\lambda}\rho)\\
 & \quad+\chi_{\{r>r_{1}\}}\bigg[\frac{a}{\lambda}+\frac{a\pi r_{1}^{\frac{n}{2}}}{2\sqrt{\lambda}}\rho^{\frac{2-n}{2}}\bigg(Y_{\frac{n}{2}}(\sqrt{\lambda}r_{1})J_{\frac{n-2}{2}}(\sqrt{\lambda}\rho)-J_{\frac{n}{2}}(\sqrt{\lambda}r_{1})Y_{\frac{n-2}{2}}(\sqrt{\lambda}\rho)\bigg)\bigg].
\end{aligned}
\end{equation*}
We now show that $\rho>r_{1}$. Suppose the contrary, that $\rho\le r_{1}$.
From \eqref{eq:explicit-formula} we have 
\[
c_{1}=\frac{a-b}{\lambda}\frac{1}{\rho^{\frac{2-n}{2}}J_{\frac{n-2}{2}}(\sqrt{\lambda}\rho)}>0.
\]
From \eqref{eq:explicit-formula-derivative} and \eqref{eq:ODE-radially-BC},
we know that 
\begin{equation}
-g(\rho)=u'(\rho)=-\frac{a-b}{\sqrt{\lambda}}\frac{J_{\frac{n}{2}}(\sqrt{\lambda}\rho)}{J_{\frac{n-2}{2}}(\sqrt{\lambda}\rho)}<0,\label{eq:derivativ1}
\end{equation}
this contradicts  the assumption $g=0$ in $\overline{B_{r_{1}}}$.
Therefore,  $\rho>r_{1}$, and hence we know that $\{\tilde{u} > 0\}\supset\overline{B_{r_{1}}}$. 

From \eqref{eq:explicit-formula} we have 
\begin{align*}
0 & =\frac{b}{\lambda}+c_{1}\rho^{\frac{2-n}{2}}J_{\frac{n-2}{2}}(\sqrt{\lambda}\rho)\\
 & \quad+\frac{a\pi r_{1}^{\frac{n}{2}}}{2\sqrt{\lambda}}\rho^{\frac{2-n}{2}}\bigg(Y_{\frac{n}{2}}(\sqrt{\lambda}r_{1})J_{\frac{n-2}{2}}(\sqrt{\lambda}\rho)-J_{\frac{n}{2}}(\sqrt{\lambda}r_{1})Y_{\frac{n-2}{2}}(\sqrt{\lambda}\rho)\bigg)\\
 & =\frac{b}{\lambda}+\rho^{\frac{2-n}{2}}J_{\frac{n-2}{2}}(\sqrt{\lambda}\rho)\bigg[c_{1}+\frac{a\pi r_{1}^{\frac{n}{2}}}{2\sqrt{\lambda}}Y_{\frac{n}{2}}(\sqrt{\lambda}r_{1})\bigg]-\frac{a\pi r_{1}^{\frac{n}{2}}}{2\sqrt{\lambda}}\rho^{\frac{2-n}{2}}J_{\frac{n}{2}}(\sqrt{\lambda}r_{1})Y_{\frac{n-2}{2}}(\sqrt{\lambda}\rho),
\end{align*}
which implies 
\[
-c_{1}\sqrt{\lambda}=\frac{b}{\sqrt{\lambda}}\frac{1}{\rho^{\frac{2-n}{2}}J_{\frac{n-2}{2}}(\sqrt{\lambda}\rho)}-\frac{a\pi r_{1}^{\frac{n}{2}}}{2}J_{\frac{n}{2}}(\sqrt{\lambda}r_{1})\frac{Y_{\frac{n-2}{2}}(\sqrt{\lambda}\rho)}{J_{\frac{n-2}{2}}(\sqrt{\lambda}\rho)}+\frac{a\pi r_{1}^{\frac{n}{2}}}{2}Y_{\frac{n}{2}}(\sqrt{\lambda}r_{1}).
\]
From \eqref{eq:explicit-formula-derivative2} we have 
\begin{align*}
u'(\rho) & =\rho^{\frac{2-n}{2}}J_{\frac{n}{2}}(\sqrt{\lambda}\rho)\bigg[\frac{b}{\sqrt{\lambda}}\frac{1}{\rho^{\frac{2-n}{2}}J_{\frac{n-2}{2}}(\sqrt{\lambda}\rho)}-\frac{a\pi r_{1}^{\frac{n}{2}}}{2}J_{\frac{n}{2}}(\sqrt{\lambda}r_{1})\overbrace{\bigg(\frac{Y_{\frac{n-2}{2}}(\sqrt{\lambda}\rho)}{J_{\frac{n-2}{2}}(\sqrt{\lambda}\rho)}-\frac{Y_{\frac{n}{2}}(\sqrt{\lambda}\rho)}{J_{\frac{n}{2}}(\sqrt{\lambda}\rho)}\bigg)}^{=\,\frac{2}{\pi\sqrt{\lambda}\rho J_{\frac{n-2}{2}}(\sqrt{\lambda}\rho)J_{\frac{n}{2}}(\sqrt{\lambda}\rho)}}\bigg]\\
 & =\rho^{\frac{2-n}{2}}J_{\frac{n}{2}}(\sqrt{\lambda}\rho)\bigg[\frac{b}{\sqrt{\lambda}}\frac{1}{\rho^{\frac{2-n}{2}}J_{\frac{n-2}{2}}(\sqrt{\lambda}\rho)}-\frac{a\pi r_{1}^{\frac{n}{2}}}{2}J_{\frac{n}{2}}(\sqrt{\lambda}r_{1})\frac{2}{\pi\sqrt{\lambda}\rho J_{\frac{n-2}{2}}(\sqrt{\lambda}\rho)J_{\frac{n}{2}}(\sqrt{\lambda}\rho)}\bigg]\\
 & =\frac{b\rho^{\frac{n}{2}}J_{\frac{n}{2}}(\sqrt{\lambda}\rho)-ar_{1}^{\frac{n}{2}}J_{\frac{n}{2}}(\sqrt{\lambda}r_{1})}{\rho^{\frac{n}{2}}\sqrt{\lambda}J_{\frac{n-2}{2}}(\sqrt{\lambda}\rho)}.
\end{align*}
Since $\rho^{\frac{n}{2}}\sqrt{\lambda}J_{\frac{n-2}{2}}(\sqrt{\lambda}\rho)>0$,
then the sign of $u'(\rho)$ is the same as that of 
\[
\phi(\rho)\equiv\frac{b}{a}-\frac{r_{1}^{\frac{n}{2}}J_{\frac{n}{2}}(\sqrt{\lambda}r_{1})}{\rho^{\frac{n}{2}}J_{\frac{n}{2}}(\sqrt{\lambda}\rho)}.
\]
Note that $\phi:[r_{1},R]\rightarrow\mathbb{R}$ is monotone increasing,
and $\phi(r_{1})=\frac{b-a}{a}<0$. Let $R'>r_{1}$ as in \eqref{eq:R-prime-choice},
then $\phi(R')=0$ as well as $u'(R')=0$. Since $u$ is monotone
decreasing on $(0,R')$, then we conclude \eqref{eq:rho-special}.
\end{proof}

\section{\label{sec:Real-analytic-case}Examples of hybrid \texorpdfstring{$k$}{k}-quadrature domains with real-analytic boundary}

In this appendix we construct some examples of hybrid $k$-quadrature domains using the Cauchy-Kowalevski theorem. We extend \cite[Theorems~3.1 and 3.3]{GSS97QuadratureDomains}
in the following theorems:

\begin{thm}
\label{thm:main-CK1} Let  $D$ be a bounded
domain in $\mathbb{R}^{n}$ with real-analytic boundary. Let $g$
be real analytic on a neighborhood of $\partial D$ with $g>0$ on
$\partial D$. For each $k\ge0$, there exists a bounded positive measure $\mu_{1}$ with ${\rm supp}\,(\mu_{1})\subset D$ such that $D$ is a hybrid $k$-quadrature domain corresponding to $\mu_{1}$ with  density $(g,0)$. 
\end{thm}

\begin{thm}
\label{thm:main-CK2} Let $0\le k<j_{\frac{n-2}{2},1}R^{-1}$,
and  $D$ be a bounded domain such that $\overline{D}\subset B_{R}$
and $\partial D$ is real-analytic. Let $h$ be a non-negative integrable
function on a neighborhood of $\overline{D}$, which is real-analytic
on some neighborhood of $\partial D$. There exists a bounded positive 
measure $\mu_{2}$ with ${\rm supp}\,(\mu_{2})\subset D$ such that
$D$ is a hybrid $k$-quadrature domain  corresponding
to $\mu_{2}$ with density $(0,h)$. 
\end{thm}
The main purpose of this section is to prove Theorems~\ref{thm:main-CK1}
and \ref{thm:main-CK2}. 
\begin{proof}
[Proof of Theorem~{\rm \ref{thm:main-CK1}}] Using the Cauchy-Kowalevski
theorem, there exists a neighborhood $\mathcal{N}$ of $\partial D$
and a real-analytic function $u$ satisfying 
\[
\begin{cases}
(\Delta+k^{2})u=0 & \text{in }\mathcal{N},\\
u=0 & \text{on }\partial D,\\
\partial_{\nu}u=-g & \text{on }\partial D.
\end{cases}
\]
Following \cite[Proposition~3.2]{GSS97QuadratureDomains},  there exists an open set $W\subset D\cap\mathcal{N}$ and a constant
$\epsilon>0$ such that 
\begin{enumerate}[leftmargin = 25pt]
\renewcommand{\labelenumi}{\theenumi}
\renewcommand{\theenumi}{{\rm (\arabic{enumi})}}
\item $u$ is positive on $W$, 
\item $\partial D\subset\partial W$, 
\item $u(x)=\epsilon$ for all $x\in\partial W\setminus\partial D$, 
\item $\partial_{\nu}u<0$ on $\partial W\setminus\partial D$. In particular, $\nabla u$ vanishes nowhere on $\partial W\setminus\partial D$,
so $\partial W\setminus\partial D$ is real-analytic.
\end{enumerate}
For each $w\in H^{1}(D)$ with $(\Delta+k^{2})w=0$ in $D$, we have 
\[
0=\int_{W}(w\Delta u-u\Delta w)\,dx=\int_{\partial W}(w\partial_{\nu}u-u\partial_{\nu}w)\,dS,
\]
where $dS(x)=d\mathscr{H}^{n-1}(x)$ is the surface measure, and hence
\[
\int_{\partial D}w\partial_{\nu}u\,dS-\int_{\partial W\setminus\partial D}w\partial_{\nu}u\,dS=\int_{\partial D}u\partial_{\nu}w\,dS-\int_{\partial W\setminus\partial D}u\partial_{\nu}w\,dS.
\]
Let $D'$ be the region with $\partial D'=\partial W\setminus\partial D$
(and thus $\overline{D'}\subset D$), then we see that 
\begin{align*}
&-\int_{\partial D}wg\,dS-\int_{\partial D'}w\partial_{\nu}u\,dS \\
& \quad =-\epsilon\int_{\partial W\setminus\partial D}\partial_{\nu}w\,dS =\epsilon\int_{D'}\Delta w\,dx=-\epsilon k^{2}\int_{D}\chi_{D'}w\,dx,
\end{align*}
that is, 
\[
\int_{\partial D}wg\,dS=\epsilon k^{2}\int_{D}\chi_{D'}w\,dx-\int_{\partial D'}w\partial_{\nu}u\,dS.
\]
Finally, we define the measure $\mu_{1}:=\epsilon k^{2}\mathscr{L}^{n}\lfloor D'-\partial_{\nu}u\mathscr{H}^{n-1}\lfloor\partial D'$.
Note that ${\rm supp}\,(\mu_{1})\subset D$ and $\mu_{1}$ is positive
bounded, hence we know that 
\begin{equation}
\langle \mu_{1},w \rangle = \int_{\partial D} gw \,dS \label{eq:QD-before}
\end{equation}
for all $w \in H^{1}(D)$ with $(\Delta + k^{2})w=0$ in $D$. For each $y \in \mathbb{R}^{n} \setminus \overline{D}$, choosing $w(x) = \Psi_{k}(x,y)$, where $\Psi$ be any fundamental solution of $-(\Delta + k^{2})$ in \eqref{eq:QD-before}, we reach 
\begin{equation}
\Psi_{k} * \mu_{1} = \Psi_{k} * (g \mathscr{H}^{n-1}\lfloor \partial D) \quad \text{in $\mathbb{R}^{n} \setminus \overline{D}$.} \label{eq:QD-before1}
\end{equation}
Since $\partial D$ is real analytic, then the right-hand-side of \eqref{eq:QD-before1} is the single layer potential. Therefore, by using the continuity of single layer potential \cite[Theorem~6.11]{McL00EllipticSystems}, and since $\mu_{1}$ is bounded, we conclude our theorem. 
\end{proof}

\begin{proof}[Proof of Theorem~{\rm \ref{thm:main-CK2}}] 
If $h$ vanishes identically on a neighborhood of $\partial D$, we have nothing to prove. We now assume that this is not the case. 

Let $v\in H_{0}^{1}(D)$ be the unique solution of 
\[
\begin{cases}
(\Delta+k^{2})v=h & \text{in }D,\\
v=0 & \text{on }\partial D.
\end{cases}
\]
Since $h\ge0$, using the strong maximum principle for Helmholtz operator (see~\cite[Appendix]{KLSS22QuadratureDomain}),
we know that $v(x)<0$ for all $x\in D$. By analyticity theorem for elliptic equations, $v$ extends real-analytically to some neighborhood of $\partial D$. Using integration by parts, we have 
\begin{equation}
\begin{aligned}
& \quad \langle\sigma_{2},w\rangle=\int_{D}wh\,dx=\int_{D}(w\Delta v-v\Delta w)\,dx \\
& =\int_{\partial D}(w\partial_{\nu}v-v\partial_{\nu}w)\,dS=\int_{\partial D}w\partial_{\nu}v\,dS,
\end{aligned}\label{eq:main-CK2-equation1}
\end{equation}
where $\sigma_{2}$ is the measure given as in \eqref{eq:measure-nu}. 

Again using the strong maximum principle for Helmholtz operator, we know that the unique solution $u_{0}$ of  
\[
\begin{cases}
(\Delta + k^{2})u_{0} = 0 &\text{in $D$,} \\
u_{0} = 1 &\text{on $\partial D$,}
\end{cases}
\]
must satisfies $u_{0} > 0$ in $\overline{D}$. By observing that 
\[
(\Delta +k^{2})v = u_{0}^{-1} \nabla \cdot u_{0}^{2} \nabla (u_{0}^{-1}v) \quad \text{in $D$,}
\]
we can apply the Hopf' maximum principle (see e.g.\ \cite[Lemma~3.4]{GT01Elliptic}) on $u_{0}^{-1}v$ to ensure $\partial_{\nu}(u_{0}^{-1}v)>0$ on $\partial D$. Since $v = 0$ on $\partial D$, then  
\[
\partial_{\nu}v = u_{0} \partial_{\nu}(u_{0}^{-1}v) > 0 \quad \text{on $\partial D$.}
\]
Since $\partial D$ is real analytic, then locally $\partial D$ has a representation as $\begin{Bmatrix}\begin{array}{l|l} y & \varphi(y)=0\end{array}\end{Bmatrix}$ for some real-analytic $\varphi$ with non-vanishing gradient. We may choose $\varphi$ to be positive outside $D$, so 
\[
\partial_{\nu}v=\nabla v\cdot\frac{\nabla\varphi}{|\nabla\varphi|}\quad\text{on }\partial D,
\]
that is, $\partial_{\nu}v$ can be extended real-analytically and strictly positive near $\partial D$. Using Theorem~\ref{thm:main-CK1}, there exists a positive bounded measure $\mu_{2}$ with ${\rm supp}\,(\mu_{2})\subset D$ such that 
\begin{equation}
\langle\mu_{2},w\rangle=\int_{\partial D}w\partial_{\nu}v\,dS.\label{eq:main-CK2-equation2}
\end{equation}
Combining \eqref{eq:main-CK2-equation1} and \eqref{eq:main-CK2-equation2}, we conclude our theorem. 
\end{proof}

\section{Some remarks on null \texorpdfstring{$k$}{k}-quadrature domains} \label{sec:Rellich}

In this appendix we give some remarks on null $k$-quadrature domains. They are defined by Definition~{\rm \ref{def:weighted-QD}} with $g \equiv 0$, $h \equiv 1$ and $\mu \equiv 0$. It was confirmed in \cite{EFW22NullQDClassification,ESW20NullQDClassification} that the null 0-quadrature domains in $\mathbb{R}^{n}$ with $n \ge 2$ must be either 
\begin{itemize}
\item half space
\item complement of an ellipsoid, 
\item complement of a paraboloid, or 
\item complement of a cylinder with an ellipsoid or a paraboloid as base,
\end{itemize}
see also \cite{FS86nullQD,Kar08nullQD,KM12nullQD,Sak81nullQD} for some classical works. 

In addition, it is worth mentioning that in the 2-dimensional case, starting from null 0-quadrature domains, we can always construct quadrature domains of positive measure \cite[Theorem~11.5]{Sak82QuadratureDomains}. This motivates us to study null $k$-quadrature domains for $k>0$. We also give some remarks showing that they are quite different.

As a consequence of mean value theorem for Helmholtz operator $-(\Delta + k^{2})$ (see e.g. \cite[Appendix]{KLSS22QuadratureDomain}), it is not difficult to see that a ball is a null $k$-quadrature domain (i.e.\ $\mu \equiv 0$) if and only if its radius $r$ satisfying $J_{\frac{n}{2}}(kr) = 0$, where $J_{\alpha}$ denotes the Bessel function of the first kind.

Following the ideas in \cite{BFHT16Rellich} one can show the following theorem. 

\begin{thm}\label{thm:trivial-unbounded-null-D}
Let $n\ge2$ be an integer, $k>0$ and $\theta\in(0,\frac{\pi}{2})$. We consider the conical domain 
\begin{equation}
\Sigma_{\theta}:=\begin{Bmatrix}\begin{array}{l|l}
(x,y)\in\mathbb{R}^{n-1}\times\mathbb{R} & y>-|x|\tan\theta\end{array}\end{Bmatrix}, \quad \text{see \cite[Figure~1]{BFHT16Rellich}}. \label{eq:conical-domain}
\end{equation}
If $w\in L^{1}(\Sigma_{\theta})$ satisfies $(\Delta+k^{2})w=0$ in $\Sigma_{\theta}$, then $w\equiv0$ in $\Sigma_{\theta}$. 
\end{thm}

Therefore, the notion of ``null $k$-quadrature domains'' for $k>0$ makes no sense for general unbounded sets. Hence it is natural to ask about the classification of \emph{bounded} null $k$-quadrature domain. As pointed out in our previous work \cite{KLSS22QuadratureDomain}, this problem is equivalent to the well-known \href{https://www.scilag.net/problem/G-180522.1}{Pompeiu's problem} \cite{Pom29PompeiuProblem,Zalcman1992}, which is also equivalent to an obstacle type free boundary problem \cite{Wil76PompeiuProblem,Wil81PompeiuProblem}.
The unanswered question is, whether any bounded (Lipschitz) null $k$-quadrature domain must be a ball. We also remark that \eqref{eq:Schiffer-second} is related to \href{https://www.scilag.net/problem/P-180522.1}{Schiffer's problem}, which asks whether the existence of a nontrivial solution $u$ of 
\[
\begin{cases}
(\Delta+k^{2})u = 0 & \text{in }D,\\
u=0 & \text{on }\partial D,\\
|\nabla u|=1 & \text{on }\partial D.
\end{cases}
\]
implies that such a (Lipschitz) bounded domain $D$ must be a ball. We refer to  \cite{BK82Pompeiu} for the Pompeiu problem for convex domains in $\mathbb{R}^{2}$, where several equivalent formulations (including the equivalence with Morera problem) as well as some partial results are given. See also \cite{Den12SchifferConjecture} for the case of domains in $\mathbb{R}^{2}$ which are strictly convex, and \cite{Ebe932DPompeiu} for domains in $\mathbb{R}^{2}$ under different assumptions.

It remains to prove Theorem~{\rm \ref{thm:trivial-unbounded-null-D}}.

\addtocontents{toc}{\SkipTocEntry}
\subsection{The case when \texorpdfstring{$n=2$}{n=2}}

We denote by $\mathscr{F}_{1}$ the 1-dimensional Fourier transform given by 
\[
\mathscr{F}_{1}\varphi(\xi):=\frac{1}{\sqrt{2\pi}}\int_{\mathbb{R}}\varphi(x)e^{-ix\xi}\,dx\quad\text{for all }\xi\in\mathbb{R},
\]
which is clearly well-defined for $\varphi\in L^{1}(\mathbb{R})$ and can be extended for tempered distributions $\varphi\in\mathscr{S}'(\mathbb{R})$. 
\begin{lem}\label{lem:representation-half-space} 
Given any $\epsilon>0$ and $k>0$, let $v\in L^{1}(\mathbb{R}\times(-\epsilon,\infty))$ satisfy $(\Delta+k^{2})v=0$ in $\mathbb{R}\times(-\epsilon,\infty)$. Then
\[
v(x,y)=\frac{1}{\sqrt{2\pi}}\int_{|\xi|>k}e^{ix\xi}(\mathscr{F}_{1}v)(\xi,0)e^{-y\sqrt{\xi^{2}-k^{2}}}\,d\xi
\]
for all $(x,y)\in\mathbb{R}\times[0,\infty)$ with 
\begin{align*}
|(\mathscr{F}_{1}v)(\xi,0)| & \le \frac{1}{\sqrt{2\pi}} \|v\|_{L^{1}(\mathbb{R}\times(-\epsilon,\infty))}\sqrt{\xi^{2}-k^{2}}e^{-\epsilon\sqrt{\xi^{2}-k^{2}}} &\text{for all } |\xi|> k,\\
(\mathscr{F}_{1}v)(\xi,0) & =0 &\text{for all }|\xi|\le k.
\end{align*}
 
\end{lem}

\begin{proof}
Since $v\in L^{1}(\mathbb{R}\times(-\epsilon,\infty))$, by Fubini's theorem we have $v(\cdot,y)\in L^{1}(\mathbb{R})$ for a.e. $y\in(-\epsilon,\infty)$, then $\mathscr{F}_{1}v(\cdot,y)$ is continuous and 
\begin{equation}
\int_{-\epsilon}^{\infty}\|\mathscr{F}_{1}v(\cdot,y)\|_{L^{\infty}(\mathbb{R})}\,dy \le \frac{1}{\sqrt{2\pi}} \int_{-\epsilon}^{\infty}\|v(\cdot,y)\|_{L^{1}(\mathbb{R})}\,dy = \frac{1}{\sqrt{2\pi}} \|v\|_{L^{1}(\mathbb{R}\times(-\epsilon,\infty))}.\label{eq:Rellich-Linfty}
\end{equation}
By applying the Fourier transform $\mathscr{F}_{1}$ on $(\Delta+k^{2})v=0$ in $\mathbb{R}\times(-\epsilon,\infty)$, then for a.e. $\xi\in\mathbb{R}$ we have 
\begin{equation}
\partial_{y}^{2}(\mathscr{F}_{1}v)(\xi,y)+(k^{2}-\xi^{2})(\mathscr{F}_{1}v)(\xi,y)\quad\text{in }\mathscr{D}'(-\epsilon,\infty).\label{eq:Rellich-ODE}
\end{equation}
Since $\mathscr{F}_{1}v(\cdot,y)\in L^{\infty}(\mathbb{R})$, then the general solution of the ODE \eqref{eq:Rellich-ODE} is given by
\begin{equation}
(\mathscr{F}_{1}v)(\xi,y)=\begin{cases}
A(\xi)e^{-(y+\epsilon)\sqrt{\xi^{2}-k^{2}}} & \text{for }|\xi|>k,\\
B_{1}(\xi)e^{-i(y+\epsilon)\sqrt{k^{2}-\xi^{2}}}+B_{2}(\xi)e^{i(y+\epsilon)\sqrt{k^{2}-\xi^{2}}} & \text{for }|\xi|\le k,
\end{cases}\label{eq:Rellich-ODE-general-solution}
\end{equation}
for some complex-valued functions $A,B_{1}B_{2}$. For each $|\xi|<k$, we see that $(\mathscr{F}_{1}v)(\xi,\cdot)$ is periodic with respect to variable $y$. However from \eqref{eq:Rellich-Linfty} we must have 
\[
(\mathscr{F}_{1}v)(\xi,\cdot)=0\quad\text{for all }|\xi|< k.
\]
When $|\xi|=k$, we have $(\mathscr{F}_{1}v)(\xi,y)=B_{1}(\xi)+B_{2}(\xi)$. Again by \eqref{eq:Rellich-Linfty} we must have  
\[
(\mathscr{F}_{1}v)(\xi,\cdot)=0\quad\text{for all }|\xi|= k.
\]
Therefore we can write \eqref{eq:Rellich-ODE-general-solution} as
\begin{equation}
(\mathscr{F}_{1}v)(\xi,y)=A(\xi)e^{-(y+\epsilon)\sqrt{\xi^{2}-k^{2}}}\quad\text{provided}\quad A(\xi)=0\text{ for all }|\xi|\le k.\label{eq:Rellich-ODE-general-solution1}
\end{equation}
Plugging \eqref{eq:Rellich-ODE-general-solution1} into \eqref{eq:Rellich-Linfty} yields 
\[
\frac{1}{\sqrt{2\pi}} \|v\|_{L^{1}(\mathbb{R}\times(-\epsilon,\infty))}\ge|A(\xi)|\int_{-\epsilon}^{\infty}e^{-(y+\epsilon)\sqrt{\xi^{2}-k^{2}}}\,dy=\frac{|A(\xi)|}{\sqrt{\xi^{2}-k^{2}}}\quad\text{for all } |\xi| > k.
\]
Finally, using the Fourier inversion formula on \eqref{eq:Rellich-ODE-general-solution1} we conclude our lemma. 
\end{proof}
\begin{cor}\label{cor:Rellich-representation-rotation}
Given any $\ell>0$, $k>0$ and $\theta\in(0,\frac{\pi}{2})$. We consider the half-planes 
\[
\Pi_{\theta,\epsilon}^{\pm}:=\begin{Bmatrix}\begin{array}{l|l}
(x,y) & y>-\ell\mp x\tan\theta\end{array}\end{Bmatrix}\equiv\begin{Bmatrix}\begin{array}{l|l}
(x,y) & \pm x\sin\theta+y\cos\theta>-\ell\cos\theta\end{array}\end{Bmatrix}.
\]
If $w_{\pm}\in L^{1}(\Pi_{\theta}^{\pm})$ satisfies $(\Delta+k^{2})w_{\pm}=0$ in $\Pi_{\theta}^{\pm}$, then 
\[
w_{\pm}(x,y)=\frac{1}{\sqrt{2\pi}}\int_{|\xi|>k}e^{i(x\cos\theta\mp y\sin\theta)\xi}\hat{\varphi}_{\pm}(\xi)e^{-(\pm x\sin\theta+y\cos\theta)\sqrt{\xi^{2}-k^{2}}}\,d\xi
\]
for all $(x,y)\in\begin{Bmatrix}\begin{array}{l|l} (x,y) & y\ge\mp x\tan\theta\end{array}\end{Bmatrix}$, for some function $\hat{\varphi}_{\pm}$ satisfies
\begin{equation} 
\begin{aligned}
|\hat{\varphi}_{\pm}(\xi)| & \le \frac{1}{\sqrt{2\pi}} \|w_{\pm}\|_{L^{1}(\Pi_{\theta}^{\pm})}\sqrt{\xi^{2}-k^{2}}e^{-\ell\cos\theta\sqrt{\xi^{2}-k^{2}}} &\text{for all } |\xi| >k,\\
\hat{\varphi}_{\pm}(\xi) & =0 &\text{for all }|\xi|\le k.
\end{aligned}\label{eq:Rellih-densities}
\end{equation}
\end{cor}

\begin{rem}\label{rem:Rellich-representation-rotation}
In particular, we have
\[
w_{\pm}(x,0)=\frac{1}{\sqrt{2\pi}}\int_{|\xi|>k}e^{i(x\cos\theta)\xi}\hat{\varphi}_{\pm}(\xi)e^{\mp x\sin\theta\sqrt{\xi^{2}-k^{2}}}\,d\xi
\]
for all $x$ with $\pm x\ge0$. 
\end{rem}

\begin{proof}[Proof of Corollary~{\rm \ref{cor:Rellich-representation-rotation}}]
Since the Laplacian is rotation invariant, then the function $v_{\pm}$ given by 
\[
w_{\pm}(x,y)=v_{\pm}(x\cos\theta\mp y\sin\theta,\pm x\sin\theta+y\cos\theta)
\]
satisfies $v_{\pm}\in L^{1}(\mathbb{R}\times(-\ell\cos\theta,\infty))$ with $(\Delta+k^{2})v_{\pm}=0$ in $\mathbb{R}\times(-\ell\cos\theta,\infty)$. Therefore using Lemma~{\rm \ref{lem:representation-half-space}} with $\epsilon=\ell\cos\theta$ yields 
\[
v_{\pm}(x,y)=\frac{1}{\sqrt{2\pi}}\int_{|\xi|>k}e^{ix\xi}(\mathscr{F}_{1}v_{\pm})(\xi,0)e^{-y\sqrt{\xi^{2}-k^{2}}}\,d\xi
\]
with 
\begin{align*}
|(\mathscr{F}_{1}v_{\pm})(\xi,0)| & \le \frac{1}{\sqrt{2\pi}} \|w_{\pm}\|_{L^{1}(\Pi_{\theta}^{\pm})}\sqrt{\xi^{2}-k^{2}}e^{-\ell\cos\theta\sqrt{\xi^{2}-k^{2}}} &\text{for all } |\xi| >k,\\
(\mathscr{F}_{1}v_{\pm})(\xi,0) & =0 &\text{for all }|\xi|\le k,
\end{align*}
because $\|v_{\pm}\|_{L^{1}(\mathbb{R}\times(-\ell\cos\theta,\infty))}=\|w_{\pm}\|_{L^{1}(\Pi_{\theta}^{\pm})}$, which conclude our corollary. 
\end{proof}

We are now ready to prove Theorem~{\rm \ref{thm:trivial-unbounded-null-D}} for the case when $n=2$.

\begin{proof}[Proof of Theorem~{\rm \ref{thm:trivial-unbounded-null-D}} when $n=2$]
For each $\ell>0$, we define 
\[
\Sigma_{\theta}^{\ell}:=\begin{Bmatrix}\begin{array}{l|l}
(x,y)\in\mathbb{R}^{n-1}\times\mathbb{R} & y+\ell>-|x|\tan\theta\end{array}\end{Bmatrix},
\]
and it is easy to see that $v(x,y)=w(x,y+\ell)$ satisfies $v\in L^{1}(\Sigma_{\theta}^{\ell})$ and $(\Delta+k^{2})v=0$ in $\Sigma_{\theta}^{\ell}$. Since $\mathbb{R}\times(-\ell,\infty)\subset\Sigma_{\theta}^{\ell}$, then $v$ admits representation as described in Lemma~{\rm \ref{lem:representation-half-space}} with $\epsilon=\ell$. Since $\mathscr{F}v(\xi,0)=0$ for all $|\xi|\le k^{2}$, to prove our theorem, we only need to show that 
\begin{equation}
\mathscr{F}_{1}v(\cdot,0)\text{ is analytic in a neighborhood of real axis.}\label{eq:Rellich-goal}
\end{equation}
Since $w(\cdot,y)\in L^{1}(\mathbb{R})$ for a.e. $y\in(0,\infty)$, then we can choose $\ell>0$ be such that 
\[
v(\cdot,0)\equiv w(\cdot,\ell)\in L^{1}(\mathbb{R}),
\]
so that we can write
\[
\mathscr{F}_{1}v(\xi,0)=\frac{1}{\sqrt{2\pi}}\int_{\mathbb{R}_{+}}v|_{\mathbb{R}_{+}}(x,0)e^{-ix\xi}\,dx+\frac{1}{\sqrt{2\pi}}\int_{\mathbb{R}_{-}}v|_{\mathbb{R}_{-}}(x,0)e^{-ix\xi}\,dx.
\]
Using Remark~{\rm \ref{rem:Rellich-representation-rotation}} with $w_{\pm}=v|_{\mathbb{R}_{\pm}}$, we have 
\begin{equation}
\mathscr{F}_{1}v(\xi,0)=\frac{1}{2\pi}\sum_{\pm}\int_{\mathbb{R}_{\pm}}\left(\int_{|\eta|>k}e^{i(x\cos\theta)\eta}\hat{\varphi}_{\pm}(\eta)e^{-|x|\sin\theta\sqrt{\eta^{2}-k^{2}}}\,d\eta\right)e^{-ix\xi}\,dx\label{eq:Rellich-integral}
\end{equation}
where $\hat{\varphi}_{\mp}$ satisfies \eqref{eq:Rellih-densities}. The integral \eqref{eq:Rellich-integral} is well-defined since 
\begin{align*}
 & \int_{\mathbb{R}_{\pm}}\left(\int_{|\eta|>k}|e^{i(x\cos\theta)\eta}\hat{\varphi}_{\pm}(\eta)e^{-|x|\sin\theta\sqrt{\eta^{2}-k^{2}}}|\,d\eta\right)|e^{-ix\xi}|\,dx\\
 & \quad =\int_{\mathbb{R}_{\pm}}\left(\int_{|\eta|>k}|\hat{\varphi}_{\pm}(\eta)|e^{-|x|\sin\theta\sqrt{\eta^{2}-k^{2}}}\,d\eta\right)\,dx\\
 & \quad \le \frac{1}{\sqrt{2\pi}} \|v\|_{L^{1}(\Sigma_{\theta}^{\ell})}\int_{|\eta|>k}\sqrt{\eta^{2}-k^{2}}e^{-\ell\cos\theta\sqrt{\eta^{2}-k^{2}}}\left(\int_{\mathbb{R}_{\pm}}e^{-|x|\sin\theta\sqrt{\eta^{2}-k^{2}}}\,dx\right)\,d\eta\\
 & \quad = \frac{1}{\sqrt{2\pi}} \frac{\|v\|_{L^{1}(\Sigma_{\theta}^{\ell})}}{\sin\theta}\int_{|\eta|>k}e^{-\ell\cos\theta\sqrt{\eta^{2}-k^{2}}}\,d\eta<\infty
\end{align*}
because $\sin\theta>0$ and $\cos\theta>0$. From this, we also able to use Fubini theorem on \eqref{eq:Rellich-integral} and we reach
\begin{align*}
\mathscr{F}_{1}v(\xi,0) & =\frac{1}{2\pi}\sum_{\pm}\int_{|\eta|>k}\left(\int_{\mathbb{R}_{\pm}}e^{i(x\cos\theta)\eta}e^{-|x|\sin\theta\sqrt{\eta^{2}-k^{2}}}e^{-ix\xi}\,dx\right)\hat{\varphi}_{\pm}(\eta)\,d\eta\\
 & =\int_{|\eta|>k}\frac{1}{2\pi}\sum_{\pm}\left(\int_{\mathbb{R}_{\pm}}e^{x(i(\eta\cos\theta-\xi)\mp\sin\theta\sqrt{\eta^{2}-k^{2}})}\,dx\right)\hat{\varphi}_{\pm}(\eta)\,d\eta =\int_{|\eta|>k}F(\eta,\xi)\,d\eta
\end{align*}
where 
\[
F(\eta,\xi)=\frac{1}{2\pi}\sum_{\pm}\frac{\hat{\varphi}_{\pm}(\eta)}{\mp i(\eta\cos\theta-\xi)+\sin\theta\sqrt{\eta^{2}-k^{2}}}.
\]
Finally, following the exactly proof in \cite[Theorem~1]{BFHT16Rellich}, we conclude \eqref{eq:Rellich-goal}, which complete the proof. 
\end{proof}

\addtocontents{toc}{\SkipTocEntry}
\subsection{The case when \texorpdfstring{$n>2$}{n>2}}

The result for $n>2$ can be proved exactly in the same way  as in \cite[Section~4]{BFHT16Rellich}, and hence its proof is omitted.


\section*{Acknowledgments}

\noindent This project was finalized while the authors stayed at Institute Mittag Leffler (Sweden), during the program Geometric aspects of nonlinear PDE.
The authors would like to express their deep gratitude to Simon Larson for his help with the initial stage of this project, and for valuable discussions. 
Kow and Salo were partly supported by the Academy of Finland (Centre of Excellence in Inverse Modelling and Imaging, 312121) and by the European Research Council under Horizon 2020 (ERC CoG 770924). Shahgholian was supported by Swedish Research Council. 

\section*{Declarations}

\noindent {\bf  Data availability statement:} All data needed are contained in the manuscript.

\medskip
\noindent {\bf  Funding and/or Conflicts of interests/Competing interests:} The authors declare that there are no financial, competing or conflict of interests.

\bibliographystyle{custom}
\bibliography{ref}

\end{sloppypar}
\end{document}